%% file: ICP1_submit_mis.tex
\documentclass[reqno,11pt]{amsart}

\usepackage{amsmath,amssymb,amsthm,mathrsfs}
\usepackage{epsfig,color,path}
\usepackage[latin1]{inputenc}

\usepackage[numbers]{natbib}

\numberwithin{equation}{section}

\def\tt{\mathbf{t}}
\def\HH{\mathbf H}

\def\Y{\mathcal Y}
\def\H{\mathcal H}
\def\K{\mathcal K}

\def\MM{\mathbf M}

\def\R{\mathbb R}
\def\N{\mathbb N}

\newcommand{\dist}{\mathop{\mathrm{dist}}}
\def\e{\varepsilon}
\def\s{\sigma}
\def\S{\Sigma}
\def\vphi{\varphi}
\def\om{\omega}
\def\l{\lambda}
\def\g{\gamma}

\def\Om{\Omega}
\def\de{\delta}
\def\Id{{\rm Id}}

\def\spt{{\rm spt}}

\def\pa{\partial}
\def\trace{{\rm tr}}

\def\kk{{\bf k}}

\def\jj{{\bf j}}

\def\00{{\bf 0}}

\def\SS{\mathbb S}

\def\bd{{\rm bd}\,}
\def\INT{{\rm int}\,}
\def\cl{{\rm cl}\,}

\def\E{\mathcal{E}}

\def\F{\mathcal{F}}

\newcommand{\vol}{\mathrm{vol}\,}

\renewcommand{\a}{\alpha}
\renewcommand{\b}{\beta}
\renewcommand{\d}{\mathrm{d}}
\newcommand{\hd}{\mathrm{hd}}
\newcommand{\D}{\Delta}
\renewcommand{\l}{\lambda}

\renewcommand{\om}{\omega}

\newcommand{\Div}{{\mathop{\rm div}}}
\newcommand{\ov}{\overline}

\newcommand{\diam}{\mathrm{diam}}
\newcommand{\cc}{\subset\subset}

\def\Lip{{\rm Lip}\,}

\def\weak{\stackrel{*}{\rightharpoonup}}

\def\ttau{\boldsymbol{\tau}}
\def\MM{\mathbf{M}}

\def\exc{\mathbf{exc}}
\def\C{\mathbf{C}}
\def\D{\mathbf{D}}

\def\tt{\mathbf{t}}

\newtheorem*{theorem*}{Theorem}
\newtheorem{theorem}{Theorem}[section]
\newtheorem{lemma}[theorem]{Lemma}
\newtheorem{proposition}[theorem]{Proposition}

\newtheorem{corollary}[theorem]{Corollary}
\newtheorem{remark}[theorem]{Remark}

\newtheorem{definition}[theorem]{Definition}

\definecolor{grey}{rgb}{.7,.7,.7}
\title[Improved convergence for planar clusters]{Improved convergence theorems for bubble clusters. \\ I. The planar case}

\author{M. Cicalese}
\address{Department of Mathematics, Technische Universit\"at M\"unchen, Boltzmannstrasse 3, 85747 Garching, GERMANY}
\email{cicalese@ma.tum.de}
\author{G. P. Leonardi}
\address{Dipartimento di Science Fisiche, Informatiche e Matematiche, Universit{\`a} degli Studi di Modena e Reggio Emilia, Via Campi 213/b, I-41100
    Modena, ITALY}
\email{gianpaolo.leonardi@unimore.it}
\author{F. Maggi}
\address{Department of Mathematics, University of Texas at Austin, Austin, TX, USA}
\email{maggi@math.utexas.edu}

\begin{document}

\begin{abstract}
{\rm We describe a quantitative construction of almost-normal diffeomorphisms between embedded orientable manifolds with boundary to be used in the study of geometric variational problems with stratified singular sets. We then apply this construction to isoperimetric problems for planar bubble clusters. In this setting we develop an {\it improved convergence theorem}, showing that a sequence of almost-minimizing planar clusters converging in $L^1$ to a limit cluster has actually to converge in a strong $C^{1,\a}$-sense. Applications of this improved convergence result to the classification of isoperimetric clusters and the qualitative description of perturbed isoperimetric clusters are also discussed. Analogous results for three-dimensional clusters are presented in part two, while further applications are discussed in some companion papers.}
\end{abstract}

\maketitle

\tableofcontents

\section{Introduction} \subsection{Overview} The aim of this two-part paper is developing a basic technique in the Calculus of Variations, that we call {\it improved convergence}, in the case of geometric variational problems where minimizers can exhibit stratified singularities. Here we think in particular to variational problems where the minimization takes place over families of generalized surfaces.

Stratified singularities appear in many problems of physical and geometrical interest. The term stratified indicates the possibility of decomposing minimizing surfaces into a hierarchy of manifolds with boundary meeting in specific optimal ways along lower dimensional manifolds of singular points. Although this behavior is well documented from the experimental point of view, its mathematical description is a quite challenging problem, which has been satisfactorily addressed only in a few specific cases. The most celebrated  example of this is probably the isoperimetric problem for bubble clusters (and, more generally, any other variational problem whose minimizers can be shown to be $(\MM,\xi,\de)$-minimal sets in the sense of Almgren \cite{Almgren76}). Indeed, Taylor \cite{taylor76} has shown that two-dimensional $(\MM,\xi,\de)$-minimal sets in the physical space $\R^3$ satisfy {\it Plateau's laws}, that is to say, they consist of regular surfaces meeting in threes at 120 degrees angles along regular curves, which in turn meet in fours at common end-points forming tetrahedral singularities.

By improved convergence we mean the principle -- usually exploited in the Calculus of Variations when showing that strict stability (positivity of the second variation) implies local minimality (in some suitable topology) -- according to which a sequence of almost-minimizing surfaces converging to some limit in a rough sense has actually to converge to that same limit in a smoother sense. This is a very familiar idea in PDE theory: for a sequence of, say, harmonic functions, $L^1$-convergence always improves to smooth convergence. In the context of geometric variational problems this kind of result is known to hold (and has been extensively used, see section \ref{section improved convergence and applications}) under the assumption that the limit surface is smooth. Our main goal here is discussing improved convergence theorems when the limit surface has stratified singularities. In this setting, by smooth convergence one means the existence of diffeomorphisms between the involved surfaces which converge in $C^1$ to the identity map, and are almost-normal (in the sense that, at fixed distance from the singularities, the displacement happens in the normal direction to the limit surface only), stratified (in the sense that singular points of a kind are mapped to singular points of the same kind), and whose tangential displacements (which cannot be zero if the singular sets do not coincide) are quantitatively controlled by their normal displacements. Obtaining this precise structure is fundamental in order to use these maps in applications: in other words, the matter here is not just constructing global diffeomorphisms between singular surfaces, but also doing it in a quite specific way, and with quantitative bounds.

From the technical point of view, our main result is Theorem \ref{thm main diffeo}, see section \ref{section curves}, which provides one with a quantitative method to construct almost-normal diffeomorphisms between embedded orientable manifolds with boundary. This result is proved in arbitrary dimension and codimension, and should have enough flexibility to be applied to different variational problems. Given a specific variational problem, the starting point for deducing an improved convergence theorem from Theorem \ref{thm main diffeo} is having at disposal a satisfactory local regularity theory around singular points. Bridging between such a local description of singularities and the global assumptions of Theorem \ref{thm main diffeo} is in general a non-trivial problem, which needs to be addressed by some {\it ad hoc} arguments.

Our two-part paper, in addition to Theorem \ref{thm main diffeo}, contains exactly this kind of analysis for those variational problems involving isoperimetric clusters. After a review of what is known about isoperimetric clusters in arbitrary dimension, see section \ref{section what is true in Rn}, we devote section \ref{section planar stuff} to address this problem in two-dimensions, thus obtaining an improved convergence theorem for almost-minimizing clusters in $\R^2$, see Theorem \ref{thm main planar} below. In part two \cite{CiLeMaIC2} we address this very same problem for isoperimetric clusters in $\R^3$.

The improved convergence theorem for planar clusters has various applications. Some are discussed in section \ref{section applications}, where we obtain structural results for planar isoperimetric clusters, see Theorem \ref{thm finitely many types} and Theorem \ref{thm potenziale}. Improve convergence is also used as the starting point to obtain quantitative stability inequalities for planar double bubbles \cite{CiLeMaFUGLEDE} and for periodic honeycombs \cite{carocciamaggi}. These companion papers provide one with a clear illustration of why it is so important to formulate improved convergence to a singular limit in terms of the existence of almost-normal, stratified diffeomorphisms converging to the identity map, with a quantitative control between the tangential and normal components of the displacements.

It is also interesting to note that the applicability of Theorem \ref{thm main diffeo} is definitely not limited to the problem of isoperimetric clusters. For example, in \cite{maggimihaila} we use Theorem \ref{thm main diffeo} and the free boundary regularity theory from \cite{dephilippismaggiCAP-ARMA} to obtain an improved convergence theorem for capillarity droplets in containers.

This introduction is organized as follows. In section \ref{section improved convergence and applications} we review some of the applications of improved convergence to smooth limit surfaces, and discuss which form an improved convergence theorem should take on singular limit surfaces. In section \ref{section perimeter minimizing clusters intro} we state our improved convergence theorem for planar minimizing clusters, Theorem \ref{thm main planar}, while section \ref{section some applications intro} presents the applications of Theorem \ref{thm main planar} discussed in this paper.

\subsection{Improved convergence to a regular limit and applications}\label{section improved convergence and applications} A basic fact about sequences of perimeter almost-minimizing sets, which comes as a direct consequence of the classical De Giorgi's regularity theory \cite{DeGiorgiREG}, is that $L^1$-convergence improves to $C^1$-convergence whenever the limiting set has smooth boundary, that is to say
\begin{equation}
  \label{improved convergence sets smooth limit}
  \left\{
  \begin{array}{l}
  \mbox{$\{E_k\}_{k\in\N}$ are perimeter almost-minimizing sets}
  \\
  \mbox{$E_k\to E$ in $L^1$ with $\pa E$ smooth}
  \end{array}
  \right .\qquad
  \Rightarrow\qquad\mbox{$\pa E_k\to\pa E$ in $C^1$}.
\end{equation}
Referring to section \ref{section sofp} for the (standard) notation and terminology about sets of finite perimeter used here and in the rest of the paper, let us recall that given $\Lambda\ge0$, $r_0>0$, and an open set $A\subset\R^n$ ($n\ge 2$), a set $E$ of locally finite perimeter in $A$  is a {\it $(\Lambda,r_0)$-minimizing set in $A$} if
\begin{equation}
  \label{almost-minimizer set}
  P(E;B_{x,r})\le P(F;B_{x,r})+\Lambda\,|E\Delta F|\,,
\end{equation}
whenever $E\Delta F\cc B_{x,r}=\{y\in\R^n:|y-x|<r\}\cc A$ and $r<r_0$. In this way, \eqref{improved convergence sets smooth limit} means that if $\{E_k\}_{k\in\N}$ is a sequence of $(\Lambda,r_0)$-minimizing sets in $\R^n$ with $|E_k\Delta E|\to 0$ as $k\to\infty$ and if $\pa E$ is a smooth hypersurface, then for every $\a\in(0,1)$ there exist $k_0\in\N$ and $\{\psi_k\}_{k\ge k_0}\subset C^{1,\a}(\pa E)$ such that, denoting by $\nu_E$ the outer unit normal to $E$ and for $k\ge k_0$,
\begin{equation}
  \label{improved convergence sets smooth limit 1}
  \pa E_k=(\Id+\psi_k\nu_E)(\pa E)\,,\qquad \sup_{k\ge k_0}\|\psi_k\|_{C^{1,\a}(\pa E)}<\infty\,,
\qquad \lim_{k\to\infty}\|\psi_k\|_{C^1(\pa E)}=0\,.
\end{equation}
(Here we have set $(\Id+\psi_k\nu_E)(\pa E)=\{x+\psi_k(x)\nu_E(x):x\in\pa E\}$.) A local version of this improved convergence result is found in \cite{mirandasucc} in the case $\Lambda=0$, but actually holds even for more general notions of almost-minimality than the one considered here; see \cite[Theorem 1.9]{tamanini}. It immediately implies a regularizing property of the sets $E_k$, in the sense that $\pa E_k$ must be a $C^1$-hypersurface as a consequence of \eqref{improved convergence sets smooth limit 1}. Improved convergence finds numerous applications to geometric variational problems. These include:

\medskip

\noindent {\it (A) Sharp quantitative inequalities}: In \cite{CicaleseLeonardi}, \eqref{improved convergence sets smooth limit} was used (with $E=B=B_{0,1}$) in combination with a selection principle and a result by Fuglede on nearly spherical sets \cite{Fuglede} to give an alternative proof of the sharp quantitative isoperimetric inequality of \cite{fuscomaggipratelli}, namely
\[
  P(E)\ge P(B)\,\Big\{1+c(n)\,\min_{x\in\R^n}|E\Delta (x+B)|^2\Big\}\,,\qquad\forall E\subset\R^n\,,|E|=|B|\,.
\]
This strategy of proof has been subsequently adopted to prove many other geometric inequalities in sharp quantitative form. Examples are
%the Gaussian isoperimetric inequality \cite{barchiesibrancolinijulin} (see also \cite{cianchifuscomaggipratelliGAUSS}),
the Euclidean isoperimetric inequality in higher codimension \cite{bogelainduzaarfusco}, the isoperimetric inequalities on spheres and hyperbolic spaces \cite{bogelainduzaarfusco,bogelainduzaarfusco2}, isoperimetric inequalities for eigenvalues \cite{brascoguidovelichov}, minimality inequalities of area minimizing hypersurfaces \cite{dephilippismaggi}, and non-local isoperimetric inequalities \cite{F2M3}; moreover, in \cite{fuscojulin} the same strategy is used to control by $P(E)-P(B)$ a more precise distance from the family of balls (see also \cite{neumayer} for the case of the Wulff inequality).

\medskip

\noindent {\it (B) Qualitative properties (and characterization) of minimizers}: Given a potential $g:\R^n\to\R$ with $g(x)\to+\infty$ as $|x|\to\infty$ and a one-homogeneous and convex integrand $\Phi:\R^n\to[0,\infty)$, in \cite{FigalliMaggiARMA} the variational problems (parameterized by $m>0$)
\begin{equation}
  \label{figallimaggi}
\inf\Big\{\int_{\pa^*E}\Phi(\nu_E)\,d\H^{n-1}+\int_{\R^n}g(x)\,dx:|E|=m\Big\}\,,
\end{equation}
are considered in the small volume regime $m\to 0^+$. Denoting by $E_m$ a minimizer with volume $m$, one expects $m^{-1/n}\,E_m$ to converge to $K$, the unit volume Wulff shape of $\Phi$. One of the main results proved in \cite{FigalliMaggiARMA} is that if $\Phi$ is a smooth elliptic integrand and $g$ is smooth, then $m^{-1/n}\,E_m\to K$ as $m\to0^+$ in every $C^{k,\a}$, with explicit rates of convergence in terms of $m$. The improved convergence theorem \eqref{improved convergence sets smooth limit}, applied with $E=K$ and on $(\Phi,\Lambda,r_0)$-minimizing sets, plays of course a basic role in this kind of analysis. The same circle of ideas has been exploited in the qualitative description of minimizers of the Ohta-Kawasaki energy for diblock copolymers \cite{cicalesespadaro}, and to characterize balls as minimizers in isoperimetric problems with competing nonlocal terms \cite{knupfermuratov,knupfermuratov2,bonacinicristo,F2M3}, and in isoperimetric problems with log-convex densities \cite{FigalliMaggiLOGCONV}.

\medskip

\noindent {\it (C) Stability and $L^1$-local minimality}: A classical problem in the Calculus of Variations is that of understanding whether stable critical points of a given functional are also local minimizers. This question was addressed in the case of the Plateau's problem by White \cite{whiteminimax}, who has proved that a {\it smooth} surface that is a strictly stable critical point of the area functional is automatically locally area minimizing  in $L^\infty$ (see \cite{MorganRos,dephilippismaggi} for the $L^1$-case). A key step in his argument is again an improved convergence theorem (for area almost-minimizing currents) towards a smooth limit. Similarly, in the case of the Ohta-Kawasaki energy, volume-constrained stable critical points with smooth boundary turn out to be volume-constrained $L^1$-local minimizers, see \cite{acerbifuscomorini}. Once again, \eqref{improved convergence sets smooth limit} is the starting point of the analysis.

We now try to address the question of the precise meaning one should give to an assertion like
\begin{equation}
  \label{improved convergence sets general}
  \left\{
  \begin{array}{l}
  \mbox{$\{E_k\}_{k\in\N}$ are perimeter almost-minimizing sets}
  \\
  \mbox{$E_k\to E$ in $L^1$}
  \end{array}
  \right .\qquad
  \Rightarrow\qquad\mbox{$\pa E_k\to\pa E$ in $C^1$}\,,
\end{equation}
when $\pa E$ is possibly singular. To this end we split $\pa E$ into its regular and singular parts: precisely, recalling that the reduced boundary $\pa^*E$ of a $(\Lambda,r_0)$-minimizing set in $\R^n$ is a $C^{1,\a}$-hypersurface for every $\a\in(0,1)$ (relatively open into $\pa E$), we define the singular set $\S(E)$ of $\pa E$ as the closed subset of $\pa E$ given by
\[
\S(E)=\pa E\setminus\pa^*E\,.
\]
It turns out that $\S(E)$ is empty if $2\le n\le 7$, discrete if $n=8$, and $\H^s$-negligible for every $s>n-8$ if $n\ge 9$; see, for example, \cite[Theorem 21.8, Theorem 28.1]{maggiBOOK}. The regularity theory behind these results also leads to a weak form of \eqref{improved convergence sets smooth limit 1}, which in turn reduces to \eqref{improved convergence sets smooth limit 1} when $\S(E)=\emptyset$. More precisely, given a sequence $\{E_k\}_{k\in\N}$ of $(\Lambda,r_0)$-minimizing sets with $E_k\to E$ in $L^1$, denoting by $I_\rho(S)$ the $\rho$-neighborhood of $S\subset\R^n$, and setting
\begin{equation}
  \label{def [paE]delta}
  [\pa E]_\rho=\pa E\setminus I_\rho(\S(E))\subset\pa^*E\,,\qquad\rho>0\,,
\end{equation}
one finds (see, e.g. Theorem \ref{thm normal representation part one} below) that, for every $\a\in(0,1)$ and $\rho$ small enough there exist $k_0\in\N$ and $\{\psi_k\}_{k\ge k_0}\subset C^{1,\a}([\pa E]_\rho)$  such that
\begin{eqnarray}
\label{improved convergence sets general 1}
    \pa E_k\setminus I_{2\rho}(\S(E))\subset (\Id+\psi_k\nu_E)([\pa E]_\rho)\subset\pa^* E_k\,,\qquad\forall k\ge k_0\,,
\end{eqnarray}
\begin{equation}
  \label{improved convergence sets general 2}
  \sup_{k\ge k_0}\|\psi_k\|_{C^{1,\a}([\pa E]_\rho)}\le C\,,\qquad    \lim_{k\to\infty}\|\psi_k\|_{C^1([\pa E]_\rho)}=0\,.
\end{equation}
Of course, if $\S(E)=\emptyset$, then \eqref{improved convergence sets general 1} and \eqref{improved convergence sets general 2} coincide with \eqref{improved convergence sets smooth limit 1}. Moreover, we notice that to replace $\pa E_k\setminus I_{2\rho}(\S(E))$ with, say, $[\pa E_k]_{3\rho}$ in the first inclusion in \eqref{improved convergence sets general 1}, one would need to prove Hausdorff convergence of $\S(E_k)$ to $\S(E)$. However, in this generality, one just knows that $\S(E_k)\subset I_\rho(\S(E))$ provided $k\ge k_0$, and actually $\S(E_k)$ may not converge in Hausdorff distance to $\S(E)$. Indeed, by a classical result of Bombieri, De Giorgi and Giusti \cite{bombieridegiorgigiusti}, the Simons's cone in $\R^8$ is the limit of perimeter minimizing sets with smooth boundary.

Even though \eqref{improved convergence sets general 1} and \eqref{improved convergence sets general 2} seem to contain all the information we can extract from the classical regularity theory, this is however not sufficient, for several reasons, to address any of the above mentioned applications. The first evident gap is that we do not parameterize the whole $\pa E_k$ on $\pa E$. Of course, in presence of singularities we cannot expect to do this by means of a normal diffeomorphism of $\pa E$;
\begin{figure}
  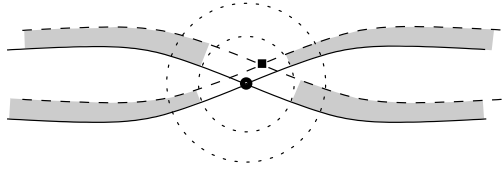\caption{{\small The limit boundary $\pa E$ is depicted with continuous lines, the approximating boundaries $\pa E_k$ by dashed lines, the singular set $\S(E)$ by a black disk, and its $\rho$ and $2\rho$-neighborhoods $I_\rho(\S(E))$ and $I_{2\rho}(\S(E))$ by concentric balls: $I_{\rho}(\S(E))$ contains the singular set of $\pa E_k$ (depicted by a black square), while \eqref{improved convergence sets general 1} says that $\pa E_k\setminus I_{2\rho}(\S(E))$ can be covered by a normal deformation of $[\pa E]_\rho=\pa E\setminus I_\rho(\S(E))$ (depicted as a grey region) which is $C^1$-close to the identity thanks to \eqref{improved convergence sets general 2}. Of course, we cannot describe $\pa E_k$ by a normal deformation of the four components of $\pa^* E$ unless $\S(E_k)=\S(E)$.}}\label{fig brutti}
\end{figure}
see Figure \ref{fig brutti}. Therefore, the best we can hope for is to find a sequence $\{f_k\}_{k\in\N}$ of $C^{1,\a}$-diffeomorphisms between $\pa E$ and $\pa E_k$ such that
\begin{equation}
  \label{improved convergence sets general 3}
  \sup_{k\in\N}\|f_k\|_{C^{1,\a}(\pa E)}<\infty\,,
\qquad \lim_{k\to\infty}\|f_k-\Id\|_{C^1(\pa E)}=0\,.
\end{equation}
A difficulty here is to specify what is meant by a $C^{1,\a}$-diffeomorphism between $\pa E$ and $\pa E_k$, since these are singular hypersurfaces. Moreover, in passing from \eqref{improved convergence sets general 1}--\eqref{improved convergence sets general 2} to \eqref{improved convergence sets general 3} we may lose the useful information that $\pa E_k$ is actually a $C^1$-small normal deformation of $\pa E$ away from the singular sets. It is therefore natural to require that
\begin{equation}
  \label{improved convergence sets general 4}
  f_k=\Id+\psi_k\,\nu_E\qquad\mbox{on $[\pa E]_\rho$}\,,
\end{equation}
with $\psi_k$ as in \eqref{improved convergence sets general 1}--\eqref{improved convergence sets general 2}. The maps $f_k$ must have a nontrivial tangential displacement
\[
u_k=(f_k-\Id)-\big((f_k-\Id)\cdot\nu_E\big)\,\nu_E\,,
\]
on $[\pa E]_\rho$ if $\S(E_k)\ne\S(E)$: and, actually, in order for the maps $f_k$ to be usable in addressing problems (A) and (C), it is crucial to have control of the $C^1$-norm of $u_k$ in terms of the distance between $\S(E_k)$ and $\S(E)$. A possibility here is requiring that $f_k(\S(E))=\S(E_k)$ (and this is something that makes sense only if, in the situation at hand, one has already shown the Hausdorff convergence of $\S(E_k)$ to $\S(E)$), with $f_k=\Id$ on $\S(E)$ if $\S(E_k)=\S(E)$, and, for some constant $C$ depending on $\pa E$,
\begin{equation}
  \label{improved convergence sets general 5}
  \|u_k\|_{C^1(\pa E)}\le C\,\|f_k-\Id\|_{C^1(\Sigma(E))}\,.
\end{equation}
Due to our limited understanding of singular sets, proving \eqref{improved convergence sets general 1}--\eqref{improved convergence sets general 5} seems a goal out of reach, and so the possibility of understanding improved convergence to singular limit sets. The theory of bubble clusters (partitions of the space into sets of finite perimeter) provides us with a (more complex) setting where singularities appear even in dimension $n=2$. However, at least when $n=2,3$, these singularities have been classified and understood, and the corresponding local regularity theory enables one to show the Hausdorff convergence of the singular sets (see Theorem \ref{thm singular hausdorff} in the case $n=2$ and \cite[Theorem 3.2]{CiLeMaIC2} in the case $n=3$). It thus makes sense to look for improved convergence theorems in this setting, and this problem is indeed addressed in this paper and in \cite{CiLeMaIC2}.

\subsection{An improved convergence theorem for planar clusters}\label{section perimeter minimizing clusters intro} Following the ideas discussed in the previous section, we now formulate our improved convergence theorem for sequences of almost-minimizing planar clusters. Given $n,N\in\N$ with $n,N\ge 2$ and an open set $A\subset\R^n$, we let $\E=\{\E(h)\}_{h=1}^N$ be a family of Lebesgue-measurable sets in $\R^n$ with
%$|\E(h)|<\infty$ for $h=1,...,N$ and
$|\E(h)\cap\E(k)|=0$ for $1\le h<k\le N$, and say that $\E$ is an {\it $N$-cluster in $A$} if $\E(h)$ is a set of locally finite perimeter in $A$ with $|\E(h)\cap A|>0$ for every $h=1,...,N$. The sets $\E(h)$ are called the {\it chambers} of $\E$, while $\E(0)=\R^n\setminus\bigcup_{h=1}^N\E(h)$ is called the {\it exterior chamber} of $\E$.
%(so that $|\E(0)|=\infty$).
The {\it perimeter} of $\E$ relative to some $F\subset\R^n$ is defined by setting
\begin{equation}
  \label{cluster perimeter relative to A}
P(\E;F)=\frac12\sum_{h=0}^NP(\E(h);F)\,,\qquad P(\E)=P(\E;\R^n)\,.%=\sum_{0\le h<k\le N}\H^{n-1}\Big(F\cap\E(h,k)\Big)\,,
\end{equation}
%Referring to section \ref{section sofp} for the formal definitions, let us say that by a $N$-cluster $\E$ in $\R^n$ we denote a disjoint family $\E=\{\E(h)\}_{h=1}^N$ of sets of finite perimeter in $\R^n$ with $0<|\E(h)|<\infty$, normalized so that $\pa\E(h)=\cl(\pa^*\E(h))$. We set
%\[
%\pa\E=\bigcup_{h=1}^N\pa\E(h)\,,\qquad\pa^*\E=\bigcup_{h=1}^N\pa^*\E(h)\,,\qquad \S(\E)=\pa\E\setminus\pa^*\E\,,
%\]
%and denote by $P(\E;F)$ the distributional perimeter of $\E$ relative to a region $F\subset\R^n$, see \eqref{cluster perimeter relative to A}, so that $P(\E;F)=\H^{n-1}(F\cap\pa\E)$ whenever each $\E(h)$ is an open set with Lipschitz boundary, and provided $\H^k$ stands for the $k$-dimensional Hausdorff measure. Setting $P(\E)=P(\E;\R^n)$ and
Setting $\vol(\E)=(|\E(1)|,...,|\E(N)|)$, a minimizer in the partitioning problem
\begin{equation}
  \label{partitioning problem}
  \inf\big\{P(\E):\vol(\E)=m\big\}\,,\qquad\mbox{$m\in\R^N_+$ given}\,,
\end{equation}
where $\R^N_+=\{m\in\R^N:m_h>0\,\forall h=1,...,N\}$, is called an {\it isoperimetric cluster}. It is of course natural to study partitioning problems in the presence of a potential energy term, like
\begin{equation}
  \label{partitioning problem with potential}
  \inf\Big\{P(\E)+\sum_{h=1}^N\int_{\E(h)}g(x)\,dx:\vol(\E)=m\Big\}\,,
\end{equation}
where, say, $g:\R^n\to\R$ with $g(x)\to+\infty$ as $|x|\to\infty$. The existence of minimizers in these two problems can be proved by a careful restoration of compactness argument due to Almgren, see \cite[Chapter 29]{maggiBOOK}.
%Another possibility is replacing perimeter with more general surface energies. On noticing that \eqref{cluster def2} implies $\nu_{\E(h)}=-\nu_{\E(k)}$ $\H^{n-1}$-a.e. on $\pa^*\E(h)\cap\pa^*\E(k)$, given an {\it even} anisotropic surface tension energy $\Phi:\R^n\to[0,\infty)$ (that is to say, $\Phi$ is positively one-homogeneous, convex and such that $\Phi(\nu)=\Phi(-\nu)$ for every $\nu\in\R^n$), and $\ss=\{\s_{h,k}\}_{0\le h<k\le N}\subset(0,\infty)$, one may consider in place of $P(\E)$ the general surface energy
%\[
%W[\PHI,\ss](\E)=\sum_{0\le h<k\le N}\s_{h,k}\,\int_{\pa^*\E(h)\cap\pa^*\E(k)}\Phi(\nu_{\E(h)})\,d\H^{n-1}\,,
%\]
%cf. with \cite[page 157]{Almgren76}. For example, when $\Phi\equiv1$ on $\SS^{n-1}$, $W[1,\ss]$ can be used to model a cluster of immiscible fluids. Our main focus here will be however on problems involving $P(\E)$ alone, as this is the only case where we have a satisfactory description of singular sets.
It turns out that if $\E$ is a minimizer either in  \eqref{partitioning problem} or in \eqref{partitioning problem with potential}, then there exist positive constants $\Lambda$ and $r_0$ such that $\E$ is {\it a $(\Lambda,r_0)$-minimizing cluster} in $\R^n$, that is (in analogy with \eqref{almost-minimizer set})
\begin{equation}
  \label{almost-minimizer cluster}
  P(\E;B_{x,r})\le P(\F;B_{x,r})+\Lambda\,\d(\E,\F)\,,
\end{equation}
whenever $x\in\R^n$, $r<r_0$ and $\E(h)\Delta\F(h)\cc B_{x,r}$ for every $h=1,...,N$, and where we have set
\begin{equation}
  \label{distance clusters}
  \d_F(\E,\F)=\frac12\sum_{h=0}^N\,\Big|F\cap\big(\E(h)\Delta\F(h)\big)\Big|\,,\qquad\d(\E,\F)=\d_{\R^n}(\E,\F)\,.
\end{equation}
In this case, as a consequence of the results obtained in \cite{Almgren76} (see also \cite[Chapter 30]{maggiBOOK} for the case $\Lambda=0$, and section \ref{section what is true in Rn} below otherwise), $\pa^*\E$ is a $C^{1,\a}$-hypersurface for every $\a\in(0,1)$ ($C^{1,1}$ if $n=2$) which is relatively open into $\pa\E$ and $\H^{n-1}(\S(\E))=0$, where
\begin{equation}
  \label{def bordi vari}
  \pa\E=\bigcup_{h=1}^N\pa\E(h)\,,\quad\pa^*\E=\bigcup_{h=1}^N\pa^*\E(h)\,,\quad \S_F(\E)=F\cap(\pa\E\setminus\pa^*\E)\,,\quad\S(\E)=\S_{\R^n}(\E)\,.
\end{equation}
One does not expect this almost-everywhere regularity result to be optimal in any dimension $n$, although the situation is clear only when $n=2$ (by elementary arguments) and when $n=3$  by \cite{taylor76}. We now review the structure of singular sets when $n=2$, and then exploit this description to formulate an improved convergence result for planar clusters. With the notation introduced in section \ref{section sets}, if $\E$ is a $(\Lambda,r_0)$-minimizing cluster in $\R^2$, then one has
\begin{equation}
  \label{structure planar cluster 1}
  \left\{\begin{array}
    {l}
    \pa\E=\bigcup_{i\in I}\g_i\,,
    \\
    \pa^*\E=\bigcup_{i\in I}\INT(\g_i)\,,
  \end{array}\right.
  \qquad
\begin{array}
  {l}
  \mbox{where $I$ is at most countable}\,,
  \\
  \mbox{$\g_i$ is a closed connected $C^{1,1}$-curve with boundary}\,,
  \\
  \mbox{$\{\g_i\}_{i\in I}$ is locally finite}\,,
\end{array}
\end{equation}
(see \cite{bleicher}, \cite{MorganPlanar},  or \cite[Section 30.3]{maggiBOOK} in the case $\Lambda=0$, and Theorem \ref{thm planar clusters} below in the general case -- which is a simple variant of the $\Lambda=0$ case). Moreover,
\begin{equation}
  \label{structure planar cluster 2}
\S(\E)=\bigcup_{j\in J}\{p_j\}=\bigcup_{i\in I}\bd(\g_i)\,,\qquad
\begin{array}
  {l}
  \mbox{where $J$ is at most countable}\,,
  \\
  \mbox{$\{p_j\}_{j\in J}$ is locally finite}\,.
\end{array}
\end{equation}
Finally, each $p_j\in\S(\E)$ is a common end-point to three different curves from $\{\g_i\}_{i\in I}$, which form three 120 degree angles at $p_j$.
% see
%\begin{figure}
%  \input{Y.pstex_t}\caption{{\small Structure of a $(\Lambda,r_0)$-minimizing cluster in the plane. The cluster depicted in the picture is also stationary and stable with respect to volumes-fixing variations.}}\label{fig Y}
%\end{figure}
%Figure \ref{fig Y}.

\begin{remark}
  {\rm As already noticed, if $\E$ is an isoperimetric cluster in $\R^2$, or if $\E$ is a minimizer in \eqref{partitioning problem with potential} with $n=2$ and $g$ is smooth, then $\E$ is a $(\Lambda,r_0)$-minimizing cluster in $\R^2$ for some $\Lambda$ and $r_0$, with the additional property of being bounded, so that $I$ and $J$ are finite. Moreover, if $\E$ is an isoperimetric cluster, then each $\g_i$ is either a circular arc or a segment; if $\E$ is a minimizer in \eqref{partitioning problem with potential}, then $\g_i$ is a closed connected smooth curve with boundary, whose curvature is equal to (the restriction to $\g_i$ of) $g$ up to an additive constant.}
\end{remark}

Motivated by these examples, we now give the following definitions, and then state our improved convergence theorem for planar clusters.

\begin{definition}\label{cluster cka}
  {\rm Let $\E$ be a cluster in $\R^2$. One says that $\E$ is a {\it $C^{k,\a}$-cluster in $\R^2$} if there exists a family of $C^{k,\a}$-curves with boundary  $\{\g_i\}_{i\in I}$ such that \eqref{structure planar cluster 1} and \eqref{structure planar cluster 2} hold.}
\end{definition}

\begin{definition}\label{def 1}
  {\rm Let $\E$ be a $C^{1,\a}$-cluster in $\R^2$. Given a map $f:\pa\E\to\R^2$ one says that $f\in C^{1,\a}(\pa\E;\R^2)$ if $f$ is continuous on $\pa\E$, $f\in C^{1,\a}(\g_i;\R^2)$ for every $i\in I$, and
  \[
  \|f\|_{C^{1,\a}(\pa\E)}:=\sup_{i\in I}\|f\|_{C^{1,\a}(\g_i)}<\infty\,.
  \]
  If $\E$ and $\E'$ are $C^{1,\a}$-clusters in $\R^2$, then one says that $f$ is a $C^{1,\a}$-diffeomorphism between $\pa\E$ and $\pa\E'$ provided $f$ is an homeomorphism between $\pa\E$ and $\pa\E'$ with $f\in C^{1,\a}(\pa\E;\R^2)$, $f^{-1}\in C^{1,\a}(\pa\E';\R^2)$, and $f(\S(\E))=\S(\E')$.}
\end{definition}

\begin{definition}
  {\rm Given a map $f:\R^2\to\R^2$ and a cluster $\E$ in $\R^2$, the tangential component of $f$ with respect to $\E$ is the map $\ttau_\E f:\pa^*\E\to\R^2$ defined by
\[
\ttau_\E f(x)=f(x)-(f(x)\cdot\nu_\E(x))\nu_\E(x)\,,\qquad x\in\pa^*\E\,,
\]
where $\nu_\E:\pa^*\E\to\SS^1$ is any Borel function such that either $\nu(x)=\nu_{\E(h)}(x)$ or $\nu(x)=\nu_{\E(k)}(x)$ for every $x\in\pa^*\E(h)\cap \pa^*\E(k)$, $h\ne k$.}
\end{definition}

\begin{theorem}[Improved convergence for planar almost-minimizing clusters]\label{thm main planar}
  Given $\Lambda\ge0$, $r_0>0$ and a bounded $C^{2,1}$-cluster $\E$ in $\R^2$, there exist positive constants $\mu_0$ and $C_0$ (depending on $\Lambda$ and $\E$) with the following property.

  If $\{\E_k\}_{k\in\N}$ is a sequence of perimeter $(\Lambda,r_0)$--minimizing clusters in $\R^2$ such that $\d(\E_k,\E)\to 0$ as $k\to\infty$, then  for every $\mu<\mu_0$ there exist $k(\mu)\in\N$ and a sequence of maps $\{f_k\}_{k\ge k(\mu)}$ such that each $f_k$ is a $C^{1,1}$-diffeomorphism between $\pa\E$ and $\pa\E_k$ with
  \begin{eqnarray*}
  \|f_k\|_{C^{1,1}(\pa\E)}&\le&C_0\,,
  \\
  \lim_{k\to\infty}\|f_k-\Id\|_{C^1(\pa\E)}&=&0\,,
  \\
  \|\ttau_\E(f_k-\Id)\|_{C^1(\pa^*\E)}&\le&\frac{C_0}\mu\,\|f_k-\Id\|_{C^0(\S(\E))}\,,
  \\
  \ttau_\E(f_k-\Id)&=&0\,,\qquad\mbox{on $\pa\E\setminus I_\mu(\S(\E))$}\,.
  \end{eqnarray*}
\end{theorem}

%\begin{remark}
%  {\rm When $\S(\E)=\S(\E_k)$ then $f_k=\Id$ on $\S(\E)$ and $f_k$ is a normal perturbation of the identity on $\pa^*\E$.  In general, the $C^1$-size of the tangential displacement is controlled by the distance between the singular sets. Moreover, for every $x\in\pa\E\setminus I_\mu(\S(\E))$, $f_k(x)$ is just the nearest point to $x$ on $\pa\E_k$, while for every $x\in\S(\E)$, $f_k(x)$ is the nearest point to $x$ on $\S(\E)$.}
%\end{remark}

\begin{remark}
  {\rm A natural question is of course whether the maps $f_k$ in Theorem \ref{thm main planar} can be extended to $C^{1,1}$-diffeomorphisms $g_k$ of $\R^2$ with $\|g_k\|_{C^{1,1}(\R^2)}\le C_0$ and $\|g_k-\Id\|_{C^1(\R^2)}\to 0$ as $k\to\infty$. The answer is yes, but at the cost of a longer proof which only employs ideas similar to the ones already used in the proof of Theorem \ref{thm main planar}. At the same time, in the applications of Theorem \ref{thm main planar} presented in \cite{CiLeMaFUGLEDE,carocciamaggi} there seems to be no real advantage in working with the maps $g_k$ in place of the maps $f_k$.}
\end{remark}

\begin{remark}\label{remark proof}
  {\rm We briefly comment on the proof of Theorem \ref{thm main planar}. The first step consists in exploiting the interior regularity theory to show (much in the spirit of \eqref{improved convergence sets general 1}--\eqref{improved convergence sets general 2}) the existence of normal diffeomorphisms between those parts of $\pa\E$ and $\pa\E_k$ that are at a fixed small distance from the singular sets $\S(\E)$ and $\S(\E_k)$. This step of the proof can be carried out in arbitrary dimension, and it is addressed in Theorem \ref{thm normal representation part one}. Next, one exploits the description of singular sets of planar clusters in order to prove the Hausdorff convergence of $\S(\E_k)$ to $\S(\E)$ (Theorem \ref{thm singular hausdorff}), and to prove that actually if $x_k\in\S(\E_k)$, $x\in\S(\E)$ and $x_k\to x$, then the tangent cones to $\pa\E_k$ at $x_k$ converge locally uniformly to the tangent cone to $\pa\E$, see step four in the proof of Theorem \ref{thm key planar}. In Theorem \ref{thm key planar} we actually show various other preliminary convergence properties of $\pa\E_k$ towards $\pa\E$, including the fact that for $k$ large enough, $\pa\E_k$ and $\pa\E$ share the same topological structure. Given all these preparatory facts, one is ready to extend the normal diffeomorphisms defined away from $\S(\E)$ to the whole $\pa\E$ by exploiting the construction of almost-normal diffeomorphism described in Theorem \ref{thm main diffeo}.}
\end{remark}

\begin{remark}
  {\rm The delicate extension of Theorem \ref{thm main planar} to clusters in $\R^3$ is discussed in \cite{CiLeMaIC2}.}
\end{remark}

\subsection{Some applications of Theorem \ref{thm main planar}}\label{section some applications intro} As explained in section \ref{section improved convergence and applications}, a result like Theorem \ref{thm main planar} opens the way to several applications. The ones given below, see Theorem \ref{thm finitely many types} and Theorem \ref{thm potenziale}, are inspired by a list of questions concerning partitioning problems proposed by Almgren in \cite[VI.1(6)]{Almgren76}, precisely ``to classify in some reasonable way the different minimizing clusters corresponding to different choices of $m\in\R^N_+$''. In this direction, let us consider the equivalence relation $\approx$ on the family of planar $C^{1,1}$-clusters such that $\E\approx\F$ if there exists a $C^{1,1}$-diffeomorphism between $\pa\E$ and $\pa\F$. Theorem \ref{thm finitely many types} shows that isoperimetric clusters of a given volume (or with volume sufficiently close to a given one) generate only finitely many $\approx$-equivalence classes.

\begin{theorem}\label{thm finitely many types}
  For every $m_0\in\R^N_+$ there exists $\de>0$ with the following property. If $\Omega$ is the family of all the isoperimetric $N$-clusters $\E$ in $\R^2$ with $|\vol(\E)-m_0|<\de$, then $\Omega/_{\approx}$ is a finite set.
\end{theorem}

By an entirely analogous principle, we can describe qualitatively minimizers in \eqref{partitioning problem with potential} when the potential energy is small enough. (In the case of planar double bubbles $N=n=2$ we can upgrade this description to a quantitative one in the spirit of \cite{FigalliMaggiARMA}, see \cite{CiLeMaFUGLEDE}.)

\begin{theorem}
  \label{thm potenziale}
  Let $m_0\in\R^N_+$ be such that there exists a unique (modulo isometries) isoperimetric cluster $\E_0$ in $\R^2$ with $\vol(\E_0)=m_0$, and let $g:\R^2\to[0,\infty)$ be a continuous function with $g(x)\to\infty$ as $|x|\to\infty$. Then there exists $\de_0>0$ (depending on $\E_0$ and $g$ only) such that for every $\de<\de_0$ and $|m-m_0|<\de_0$ there exist minimizers of
  \begin{equation}
  \label{partitioning problem with potential delta}
  \inf\Big\{P(\E)+\de\,\sum_{h=1}^N\int_{\E(h)}g(x)\,dx:\vol(\E)=m\Big\}\,.
  \end{equation}
  If $\E$ is a minimizer in \eqref{partitioning problem with potential delta}, then $\E\approx\E_0$. Moreover, if $H_{\E(h,k)}$ denotes the scalar mean curvature of the interface $\E(h,k)$ with respect to $\nu_{\E(h)}$, then $H_{\E(h,k)}$ is continuous on $\E(h,k)$, with
  \begin{equation}
    \label{convessita}
      \max_{0\le h<k\le N}\|H_{\E(h,k)}-H_{\E_0(h,k)}\|_{C^0(\E(h,k))}\le C_0\,\de\,,
  \end{equation}
  for a constant $C_0$ depending on $\E_0$ and $g$ only. (Notice that $H_{\E_0(h,k)}$ is a constant for every $0\le h<k\le N$.)
\end{theorem}

Of course, thanks to Theorem \ref{thm finitely many types}, if the uniqueness assumption on $m_0$ in Theorem \ref{thm potenziale} is dropped, then one can still infer that minimizers in \eqref{partitioning problem with potential delta} with $\de<\de_0$ and $|m-m_0|<\de_0$ generate only finitely many $\approx$-equivalence classes. Moreover, we explicitly notice that the novelty of Theorem \ref{thm potenziale} is not the existence part, which follows by standard arguments, but the fact that $\E\approx\E_0$.

Further applications of Theorem \ref{thm main planar} are discussed elsewhere. In \cite{CiLeMaFUGLEDE}, Theorem \ref{thm main planar} is the starting point for obtaining a sharp stability inequality for planar double-bubbles, while in \cite{carocciamaggi} we address a sharp quantitative version of Hales's isoperimetric theorem for the regular hexagonal tiling \cite{hales}.

\subsection{Organization of the paper}\label{section organization} The paper is essentially divided in two parts. The first part consists of sections \ref{section not and prel}--\ref{section curves}. The goal here is to provide in a reasonable generality the construction of almost-normal diffeomorphisms between manifolds with boundary. As said, this is the key result in constructing the maps appearing in Theorem \ref{thm main planar}. It is considered in arbitrary dimension and co-dimension (and not just for curves in the plane) in view of its applications to the improved convergence of clusters in $\R^3$ \cite{CiLeMaIC2} and to the description of capillarity droplets in containers  \cite{maggimihaila}. We provide two statements of this result, see Theorem \ref{thm main diffeo} and Theorem \ref{thm main diffeo k}, the second one being more practical in applications. These results are proved in section \ref{section curves}, after some preliminary facts concerning the implicit function theorem and Whitney's extension theorem are gathered in section \ref{section not and prel}. In the second part of the paper, which consists of sections \ref{section what is true in Rn}--\ref{section planar stuff} we gather the various ingredients needed to deduce Theorem \ref{thm main planar} from Theorem \ref{thm main diffeo}, as described in Remark \ref{remark proof}.  Finally, in section \ref{section applications} we give the (closely related proofs of) Theorem \ref{thm finitely many types} and Theorem \ref{thm potenziale}.

\medskip

\noindent {\bf Acknowledgement}: We thank Frank Morgan for improving the presentation of our paper with some useful comments. The work of FM was supported by NSF Grants DMS-1265910 and DMS-1361122 The work of GPL has been supported by GNAMPA (INdAM).

\section{Notation and preliminaries}\label{section not and prel} We gather here some basic notation and classical facts to be used here and in \cite{CiLeMaIC2}.

\subsection{Sets in $\R^n$}\label{section sets} Given $x\in\R^n$ and $r>0$ we set $B(x,r)=B_{x,r}=\{y\in\R^n:|y-x|<r\}$ and $B(0,r)=B_{0,r}=B_r$,
where $|v|^2=v\cdot v$ and $v\cdot w$ is the scalar product of $v,w\in\R^n$. We set $\SS^{n-1}=\{x\in\R^n:|x|=1\}$. Given $S\subset\R^n$, we denote by $\mathring{S}$, $\pa S$, and $\cl(S)$ the interior, the boundary, and the closure of $S$ respectively, while $I_\e(S)$ denotes the tubular $\e$-neighborhood of $S$ in $\R^n$, that is $I_\e(S)=\{x\in\R^n:\dist(x,S)<\e\}$, $\e>0$. Given $S,T\subset\R^n$ we define the Hausdorff distance between  $S$ and $T$ localized in $K\subset\R^n$ as
\begin{equation}
  \label{hausdorff distance localized def}
  \hd_K(S,T)=\max\big\{\sup\{\dist(y,S):y\in T\cap K\},\sup\{\dist(y,T):y\in S\cap K\}\big\}\,,
\end{equation}
so that $\hd_K(S,T)<\e$ if and only if $S\cap K\subset I_\e(T)$ and $T\cap K\subset I_\e(S)$, while
\[
\hd_{x,r}(S,T)=\hd_{B_{x,r}}(S,T)\,,\qquad \hd(S,T)=\hd_{\R^n}(S,T)\,.
\]
If $S$ is a $k$-dimensional $C^1$-manifold in $\R^n$ (we always work with {\it embedded} manifolds), then the geodesic distance on $S$ is given by
\[
\textstyle\dist_{S}(x,y)\displaystyle=\inf\big\{\int_0^1|\g'(t)|\,dt:\g\in C^1([0,1];S)\,,\g(0)=x\,,\g(1)=y\big\}\,,\qquad x,y\in S\,.
\]
We also define the normal $\e$-neighborhood of $S$ as
\begin{equation}
  \label{Nespilon}
N_\e(S)=\big\{x+\sum_{i=1}^{n-k}t_i\,\nu^{(i)}(x):x\in S\,,\sum_{i=1}^{n-k}t_i^2<\e^2\big\}\,,
\end{equation}
provided $\{\nu^{(i)}(x)\}_{i=1}^{n-k}$ denotes an orthonormal basis to $(T_xS)^\perp$. If $S$ is a $k$-dimensional $C^1$-manifold with boundary in $\R^n$, then $\INT(S)$ and $\bd(S)$ denote, respectively, the interior and the boundary points of $S$. If $x\in\bd(S)$, then we define $T_xS$ as a $k$-dimensional space (thus, not as an half-space), and we denote by $\nu^{co}_{S}(x)\in T_xS$ the outer unit normal to $\bd(S)$ with respect to $S$. Moreover, we set
\begin{equation}
  \label{S rho}
[S]_\rho=S\setminus I_\rho(\bd(S))\,,\qquad\forall \rho>0\,.
\end{equation}
Denoting by $\pi^S_x$ the projection of $\R^n$ onto $T_xS$, for every $f:S\to\R^n$ we define $\pi^Sf:S\to\R^n$ by taking
\[
(\pi^Sf)(x)=\pi^S_x[f(x)]\,,\qquad x\in S\,.
\]
The terms curve, surface and hypersurface are used in place of $1$-dimensional manifold, $2$-dimensional manifold and $(n-1)$-dimensional manifold in $\R^n$.

\subsection{Uniform inverse and implicit function theorems}\label{section maps on curves} If $S$ is a $k$-dimensional $C^{1,\a}$-manifold in $\R^n$ ($\a\in(0,1]$), $x\in S$, and $f:S\to\R^m$, then we say that $f$ is differentiable at $x$ with respect to $S$ if we can define a linear map $\nabla^Sf(x):\R^n\to\R^m$ by setting
\[
\nabla^Sf(x)[v]=
\begin{cases}
\lim\limits_{t\to 0}\frac{f(\g(t))-f(x)}{t}\,&\mbox{if $v\in T_xS$, $\g\in C^1((-\e,\e);S)$, $\g(0)=x$, $\g'(0)=v$}\,,
\\
0&\mbox{if $v\in (T_xS)^{\perp}$}\,.
\end{cases}
\]
Denoting by $\|L\|=\sup\{|L[v]|:|v|= 1\}$ the operator norm of a linear map $L:\R^n\to\R^m$, we set
\[
\|f\|_{C^1(S)}=\sup_{x\in S}|f(x)|+\|\nabla^S f(x)\|\,.
\]
Of course, if $f$ is differentiable on an open neighborhood of $S$, then $\nabla^Sf(x)$ is just the restriction of the differential $\nabla f(x)$ of $f$ at $x$ to $T_xS$, extended to $0$ on $(T_xS)^\perp$. For $\a\in(0,1]$ we set
\begin{eqnarray*}
  [\nabla^Sf]_{C^{0,\a}(S)}&=&\sup_{x,y\in S,\,x\ne y}\frac{\|\nabla^Sf(x)-\nabla^Sf(y)\|}{|x-y|^\a}\,,
  \\
  \|\nabla^Sf\|_{C^{0,\a}(S)}&=&\sup_{x\in S}\|\nabla^Sf(x)\|+[\nabla^Sf]_{C^{0,\a}(S)}\,,
  \\
  \|f\|_{C^{1,\a}(S)}&=&\sup_{x\in S}|f(x)|+\|\nabla^Sf\|_{C^{0,\a}(S)}\,,
\end{eqnarray*}
(note the use of the Euclidean distance in the definition of $[\cdot]_{C^{0,\a}(S)}$). If $\{\tau_i(x)\}_{i=1}^k$ is an orthonormal basis of $T_xS$, then we define the tangential Jacobian of $f$ as
\[
J^Sf(x)=\Big|\bigwedge_{i=1}^k\,\nabla^Sf(x)[\tau_i(x)]\Big|\,,\qquad x\in S\,.
\]
The following theorems are uniform versions of the inverse and implicit function theorems. The proof of the first result is included in Appendix \ref{section implicit} for the sake of clarity.

\begin{theorem}[Uniform inverse function theorem]\label{thm inverse fnct uniform}
Given $n\ge 2$, $1\le k\le n-1$, $\a\in(0,1]$, $L>0$, and $S_0$ a $k$-dimensional $C^{1,\a}$-manifold in $\R^n$ with $\diam(S_0)\le L$ and
\begin{eqnarray}
  \label{geo-euc}
  \textstyle \dist_{S_{0}}(x,y)\leq L\,|x-y|\,,&&\qquad\forall x,y\in S_0\,,
  \\
  \label{vai ganzo}
  |y-x|\le 2\,|\pi^{S_0}_x(y-x)|\,,&&\qquad\forall x\in S_0\,, y\in B_{x,1/L}\cap S_0\,,
  \\
    \label{vai ganzo1}
      \|\pi^{S_0}_x-\pi^{S_0}_y\|\le L\,|x-y|^\a\,,&&\qquad \forall x,y\in S_0\,,
\end{eqnarray}
there exist positive constants $\e_0$, $\rho_0$ and $C_0$, depending on $n$, $k$, $\a$, and $L$ only, with the following properties. If $f\in C^{1,\a}(S_0;\R^n)$ is such that
  \begin{equation}
    \label{inverse function1}
      \inf_{S_0}\,J^{S_0}f\,\ge \frac1{L}\,,\qquad \|\nabla^{S_0}f\|_{C^{0,\a}(S_0)}\le L\,,
  \end{equation}
  then $f$ is injective on $B_{x,\e_0}\cap S_0$ for every $x\in S_0$. If, moreover,
  \begin{equation}
    \label{inverse function2}
      \|f-\Id\|_{C^{0}(S_{0})}\le\rho_0\,,
  \end{equation}
then $S=f(S_0)$ is a $k$-dimensional $C^{1,\a}$-manifold in $\R^n$ and $f:S_{0}\to S$ is a $C^{1,\a}$-diffeomorphism satisfying $\|f^{-1}\|_{C^{1,\a}(S)}\le C_0$.
\end{theorem}

\begin{theorem}[Uniform implicit function theorem]\label{lemma implicit function}
  Let $n$, $k$, $\a$, $L$ and $S_0$ be as in Theorem \ref{thm inverse fnct uniform}. Then there exist positive constants $C_0$ and $\eta_0$ depending on $n$, $k$, $\a$, and $L$ only with the following property. If $x_0\in S_0$ and $u\in C^{1,\a}(S_0\times(-1,1)^{n-k};\R^{n-k})$ is such that
  \begin{equation}
    \label{hp implicit function}
      u(x_0,\00)=\00\,,\qquad \Big|\bigwedge_{i=1}^{n-k}\frac{\pa u}{\pa t_i}(x_0,\00)\Big|\ge\frac1{L}\,,\qquad \|u\|_{C^{1,\a}(S_0\times(-1,1)^{n-k})}\le L\,,
  \end{equation}
  where $\00=(0,...,0)\in\R^{n-k}$, then there exists a function $\zeta\in C^{1,\a}(S_0\cap B_{x_0,\eta_0};\R^{n-k})$ such that
  \begin{equation}
    \label{zeta properties}
      \zeta(x_0)=\00\,,\qquad u(z,\zeta(z))=\00\,,\qquad \forall z\in S_0\cap B_{x_0,\eta_0}\,,\qquad\|\zeta\|_{C^{1,\a}(S_0\cap B(x_0,\eta_0))}\le C_0\,.
  \end{equation}
\end{theorem}

\begin{proof}
 One applies the first conclusion of Theorem \ref{thm inverse fnct uniform} to the manifold $S_0\times(-1,1)^{n-k}$ and the function $f:S_0\times(-1,1)^{n-k}\to\R^n$ defined by $f(x,\tt)=(x,u(x,\tt))$; see, e.g. \cite{spivak}.
\end{proof}

\subsection{Whitney's extension theorem}\label{section whitney} Here we review some basic facts concerning Whitney's extension theorem. By $\kk=(k_1,...,k_n)$ we denote the generic element of $\N^n$, and set
\[
|\kk|=\sum_{i=1}^nk_i\,,\qquad \kk!=\prod_{i=1}^nk_i\,,\qquad z^\kk=\prod_{i=1}^n\,z_i^{k_i}\,,
\]
for every $\kk\in\N^n$ and $z\in\R^n$. If $f$ is $|\kk|$-times differentiable at $x\in\R^n$, we let
\[
D^\kk\,f(x)=\frac{\pa^{|\kk|}f}{\pa x_1^{k_1}...\pa x_n^{k_n}}(x)=\frac{\pa^{|\kk|}f}{\pa x^\kk}(x)\,,
\]
denote the $\kk$-partial derivative of $f$, with the convention that $D^\00 f=f$ (here, $\00=(0,...,0)\in\N^n$).

Let now $X$ be a compact set in $\R^n$. A {\it jet of order $h$ on $X$} is simply a family $\F=\{F^{\kk}\}_{|\kk|\le h}$ of continuous functions on $X$, see \cite{bierstone}. We denote by $J^h(X)$ the vector space of jets of order $h$ on $X$, and set
\[
\|\F\|_{J^h(X)}=\max_{|\kk|\le h}\|F^\kk\|_{C^0(X)}\,.
\]
%A jet of infinite order on $X$ is just a family of continuous functions $\F=\{F^{\kk}\}_{\kk\in\N^n}$ on $X$, and in this case we set $\F\in J^\infty(X)$.
One says that $\F\in J^h(X)$ is a {\it Whitney's jet of order $h$ on $X$} if, for every $|\kk|\le h$,
\[
\sup_{x,y\in X\,,0<|x-y|<r}\Big|F^\kk(y)-F^\kk(x)-\sum_{|j|=1}^{h-|\kk|}F^{\kk+\jj}(x)(y-x)^{\kk+\jj}\Big|=o(r^{h-|\kk|})\,.
\]
Given $\a\in[0,1]$,  we denote by $WJ^{h,\a}(X)$ the space of Whitney's jets of order $h$ on $X$ such that
% and set
%\begin{eqnarray*}
%\|\F\|_{WJ^h(X)}&=&
%\max_{|\kk|\le h}\|F^\kk\|_{C^0(X)}
%\\
%&&+\max_{|\kk|\le h}\,\sup_{x\,,y\in X\,,x\ne y}\frac{|F^\kk(y)-F^\kk(x)-\sum_{|j|=1}^{h-|\kk|}F^{\kk+\jj}(x)(y-x)^{\kk+\jj}|}{|x-y|^{h-|\kk|}}\,.
%\end{eqnarray*}
%Finally, given $\a\in(0,1]$ we define $WJ^{h,\a}(X)$ as the space of those jets $\F\in J^h(X)$ such that
\begin{eqnarray*}
\|\F\|_{WJ^{h,\a}(X)}&=&\max_{|\kk|\le h}\|F^\kk\|_{C^0(X)}
\\
&&+\max_{|\kk|\le h}\,\sup_{x\,,y\in X\,,x\ne y}\frac{|F^\kk(y)-F^\kk(x)-\sum_{|j|=1}^{h-|\kk|}F^{\kk+\jj}(x)(y-x)^{\kk+\jj}|}{|x-y|^{h-|\kk|+\a}}\,,
\end{eqnarray*}
is finite. We set $WJ^h(X)=WJ^{h,0}(X)$, and notice that $WJ^{h+1}(X)\subset WJ^{h,\a}(X)\subset WJ^h(X)$ for every $h\in\N$ and $\a\in(0,1]$.

\begin{theorem}
  [Whitney's extension theorem]\label{thm whitney} For every $n,h\ge 1$, $\a\in[0,1]$ and $L>0$ there exists a constant $C_0$ depending on $n$, $\a$ and $L$ only with the following property. If $X$ is a compact set in $\R^n$ with $X\subset B_L$ and $\F\in WJ^{h,\a}(X)$, then there exists $f\in C^\infty(\R^n\setminus X)\cap C^{h,\a}(\R^n)$ such that
  \begin{equation}
    \label{whitney extension}
    \mbox{$D^\kk f=F^\kk$ on $X$ for every $|\kk|\le h$}\,,
  \end{equation}
  \begin{equation}
    \label{whitney bounded}
      \|f\|_{C^{h,\a}(\R^n)}\le C_0\,\|\F\|_{WJ^{h,\a}(X)}\,,
      \qquad
      \|f\|_{C^h(\R^n)}\le C_0\,\|\F\|_{WJ^h(X)}\,.
  \end{equation}
  If, moreover, $X$ is connected by rectifiable arcs and its geodesic distance ${\rm \dist}_X$ satisfies
  \begin{equation}
  \label{k regular set}
  {\rm \dist}_X(x,y)\le \om\,|x-y|\,,\qquad \forall x,y\in X\,,
  \end{equation}
  for some $\om>0$, then $\|\F\|_{WJ^h(X)}\le 2\,\om\,\|\F\|_{J^h(X)}$, and thus, in particular,
  \begin{equation}
    \label{whitney bounded regular case}
      \|f\|_{C^h(\R^n)}\le 2\,\om\,C_0\,\|\F\|_{J^h(X)}\,.
  \end{equation}
\end{theorem}

\begin{proof}
  The classical construction introduced by Whitney (see \cite[Theorem 4, Chapter VI]{steinbook} or \cite[Theorem 2.3]{bierstone}) gives a function $g\in C^\infty(\R^n\setminus X)\cap C^{h,\a}(B_{2L})$ with
  \begin{gather}
    \label{whitney extension g}
    \mbox{$D^\kk g=F^\kk$ on $X$ for every $|\kk|\le h$}\,,
    \\
    \label{whitney bounded g}
    \|g\|_{C^{h,\a}(B_{2L})}\le C\,\|\F\|_{WJ^{h,\a}(X)}\,,\qquad \|g\|_{C^h(B_{2L})}\le C\,\|\F\|_{WJ^h(X)}\,,
%    \|g\|_{C^{1,\a}(B_{2L})}\le C\,\|\F\|_{WJ^{1,\a}(X)}\,,
  \end{gather}
  where the constant $C$ depends on $n$, $h$, $\a$ and $L$. If we now pick $\eta\in C^\infty_c(B_{2L};[0,1])$ with $\eta=1$ on $B_L$, then by setting $f=g\,\eta$ we prove the first part of the statement. % with
%  \[
%  \|f\|_{C^0(\R^n)}\le \|g\|_{C^0(B_{2L})}\,,\qquad \|\nabla f\|_{C^0(\R^n)}\le C\,\|g\|_{C^1(B_{2L})}\,,\qquad
%  [\nabla f]_{C^{0,\a}(\R^n)}
%  \le C\,\|g\|_{C^{1,\a}(B_{2L})}\,.
%%  \|f\|_{C^{1,\a}(\R^n)}\le C\,\|g\|_{C^{1,\a}(B_{2L})}\,.
%  \]
%  In particular, \eqref{whitney bounded} follows from \eqref{whitney bounded g}.
  The second part of the statement is \cite[Proposition 2.13]{bierstone}. For the sake of clarity, let us explain this point in the case $h=1$. If $X$ is connected by rectifiable arcs and $x,y\in X$ with $x\ne y$, then for every $\e>0$ there exists $\phi\in C^0([0,1];X)$ such that
  \begin{equation}
    \label{glaeser1}
    \ell(\phi)\le (1+\e)\,{\rm dist}_X(x,y)\le (1+\e)\,\om\,|x-y|\,,\qquad\phi(0)=x\,,\qquad\phi(1)=y\,,
  \end{equation}
  where $\ell(\phi)$ is the total variation of $\phi$. We can re-parameterize $\phi$ on $[0,1]$ so to have $\phi\in \Lip([0,1];X)$ with $|\phi'(t)|=\ell(\phi)$ for every $t\in[0,1]$. By \eqref{glaeser1} we thus find
  \begin{eqnarray*}
  &&\Big|F^\00(y)-F^\00(x)-\sum_{i=1}^n F^{e_i}(x)(y-x)_i\Big|
  \\&=& |f(y)-f(x)-\nabla f(x)\cdot(y-x)|
  =\Big|\int_0^1\,(\nabla f(\phi(t))-\nabla f(x))\cdot\,\phi'(t)\,dt \Big|
  \\
  &\le&2\,\|\nabla f\|_{C^0(X)}\,\int_0^1|\phi'(t)|\,dt
  \le2\,(1+\e)\,\|\F\|_{J^1(X)}\,\om\,|x-y|\,.\hspace{120pt}\qedhere
  \end{eqnarray*}
\end{proof}

The following two propositions are used in the proof of Theorem \ref{thm main diffeo}.

\begin{proposition}\label{proposition costruzione extbyfol}
If $n\ge 2$, $1\le k\le n-1$, $\a\in(0,1]$ and $L>0$, then there exist positive constants $C$ and $\e$ depending on $n$, $k$, $\a$ and $L$ only, with the following property. Let $X$ be a compact set in $\R^n$ with $\diam(X)\le L$, and assume that for every $x\in X$ one can define an orthonormal system of vectors $\{\nu^{(j)}(x)\}_{j=1}^{n-k}$ in such a way that for every $x,y\in X$ and $1\le j\le n-k$,
%\begin{equation}\label{glaser condition}
%  \textstyle\dist_{\g}(x,y)\le K\,|x-y|\,,\qquad \forall x,y\in\g\,,
%\end{equation}
%$\g\subset B_K$ and
\begin{equation}
  \label{bound whitney X}
|\nu^{(j)}(x)\cdot(y-x)|\le L\,|x-y|^{1+\a}\,,\qquad|\nu^{(j)}(x)-\nu^{(j)}(y)|\le L\,|x-y|^\a\,.
\end{equation}
Then there exists $d\in C^\infty(\R^n\setminus X;\R^{n-k})\cap C^{1,\a}(\R^n;\R^{n-k})$ with
\begin{equation}\label{bound whitney X tesi}
  \begin{split}
    &\mbox{$d(x)=0$ and $\nabla d(x)=\sum_{j=1}^{n-1}e_j\otimes\nu^{(j)}(x)$ for every $x\in X$}\,,
    \\
    &\mbox{$I_\e(X)\cap\{d=0\}$ is a $k$-dimensional $C^{1,\a}$-manifold in $\R^n$}\,,
    \\
    &\max\big\{\e^{-1}\,,\|d\|_{C^{1,\a}(\R^n)}\big\}\le C\,.
  \end{split}
\end{equation}
\end{proposition}

\begin{proof}
By \eqref{bound whitney X}, if one sets $F_j^{\00}(x)=0$ and $F_j^{e_i}(x)=\nu^{(j)}(x)\cdot e_i$ for $x\in X$ and $1\le i\le n$, then $\F_j\in WJ^{1,\a}(X)$ with $\|\F_j\|_{WJ^{1,\a}(X)}\le C$. Since $\diam(X)\le L$, by Theorem \ref{thm whitney} one finds $d_j\in C^\infty(\R^n\setminus X)\cap C^{1,\a}(\R^n)$ with $d_j=0$ and $\nabla d_j=\nu^{(j)}$ on $X$. The function $d=\sum_{j=1}^{n-k}d_j\,e_j$ satisfies the first property in \eqref{bound whitney X tesi}. If now $x\in I_\e(X)$, then there exists $y\in X$ such that $|y-x|<\e$, thus $\|d\|_{C^{1,\a}(\R^n)}\le C$ gives
\[
\big|\bigwedge_{j=1}^{n-k}\nabla d(x)[\nu^{(j)}(y)]\big|\ge\big|\bigwedge_{j=1}^{n-k}\nabla d(y)[\nu^{(j)}(y)]\big|-C\,\e=1-C\,\e\ge\frac12\,,
\]
provided $\e$ is small enough (depending only on $n$, $k$, $\a$ and $L$). In particular, $\nabla d(x)$ has rank $n-k$ for every $x\in I_\e(X)$, thus $I_\e(X)\cap\{d=0\}$ is a $k$-dimensional $C^{1,\a}$-manifold in $\R^n$.
\end{proof}

\begin{proposition}\label{proposition estensione bordo di S}
  If $n\ge 2$, $1\le k\le n-1$, $\a\in(0,1]$ and $L>0$, then there exists a constant $C$ depending on $n$, $k$, $\a$ and $L$ only, with the following property. If $S$ is a compact connected $k$-dimensional $C^{2,1}$-manifold with boundary in $\R^n$ with $\diam(S)\le L$ and
  \[
  {\rm \dist}_{\bd(S)}(x,y)\le L\,|x-y|\,,\qquad \forall x,y\in \bd(S)\,,
  \]
  and $\bar{a}\in C^{1,\a}(\bd(S))$, then there exist $a\in C^{1,\a}(\R^n)$ with $a=\bar{a}$ on $\bd(S)$  and
  \[
  \|a\|_{C^{1,\a}(\R^n)}\le C\,\|\bar{a}\|_{C^{1,\a}(\bd(S))}\,,\qquad \|a\|_{C^1(\R^n)}\le C\,\|\bar{a}\|_{C^1(\bd(S))}\,.
  \]
\end{proposition}

\begin{proof}
  We note that, by definition of tangential gradient, $\nabla^{\bd S}\bar{a}(x)\in T_x(\bd(S))$ for every $x\in\bd(S)$. We then define $\F\in J^1(\bd(S))$ by setting $F^\00(x)=\bar{a}(x)$ and $F^{e_i}(x)=e_i\cdot\nabla^{\bd S}\bar{a}(x)$ for $x\in\bd(S)$, and note that $\F\in WJ^{1,\a}(\bd(S))$ with
  \[
  \|\F\|_{WJ^{1,\a}(\bd(S))}\le\|\bar a\|_{C^{1,\a}(\bd(S))}\,,
  \qquad
  \|\F\|_{WJ^1(\bd(S))}\le\|\bar a\|_{C^1(\bd(S))}\,.
  \]
  We conclude by Theorem \ref{thm whitney}.
\end{proof}

\section{Almost-normal diffeomorphisms between manifolds with boundary}\label{section curves} The main result of this section is Theorem \ref{thm main diffeo}, where we address the following problem. We are given two $k$-dimensional manifolds with boundary $S_0$ and $S$, which are known to be close in Hausdorff distance. Moreover, we are given a diffeomorphism $f_0$ (close to the identity map) between the boundaries of $S_0$ and $S$, and we know that $S$ is a small normal deformation of $S_0$ up to some small distance from its boundary. (The motivation for considering this scenario is that -- by interior and boundary/free-boundary regularity theorems -- this is the typical starting point in addressing the improved convergence problem in presence of singularities). Then we would like to extend $f_0$ into a diffeomorphism $f$ between $S_0$ and $S$ while keeping the size of the tangential displacement $\pi_{S_0}(f-\Id)$ of $f$ as small as possible.

In section \ref{section diffeo construction theorem 3.1} we state and prove Theorem \ref{thm main diffeo}, while in section \ref{subsection application of theorem 3.1} we provide an alternative formulation of this result in terms of sequences of manifolds converging to a limit manifold $S_0$ which is more natural to invoke when addressing applications.

\subsection{Construction of the diffeomorphisms}\label{section diffeo construction theorem 3.1} Before stating the theorem we premise the following definition, which in turn is motivated by Proposition \ref{proposition costruzione extbyfol}. Given an orientable $k$-dimensional $C^{1,\a}$-manifold $S$ in $\R^n$ which admits a global normal frame of class $C^{1,\a}$ (i.e., such that for every $x\in S$ there exists an orthonormal basis $\{\nu^{(i)}_S(x)\}_{i=1}^{n-k}$ of $(T_xS)^\perp$ with the property $\nu^{(i)}\in C^{1,\a}(S)$ for each $i$) then one writes
\[
\|S\|_{C^{1,\a}}\le L\,,
\]
if
\begin{equation}\label{basta S C1alpha}
\left\{
\begin{split}
  &|\nu^{(i)}_S(x)-\nu^{(i)}_S(y)|\le L\,|x-y|^\a\,,
  \\
  &|\nu^{(i)}_S(x)\cdot(y-x)|\le L|y-x|^{1+\a}\,,
\end{split}\right .
    \qquad\forall x,y\in S\,,i=1,...,n-k\,.
\end{equation}
We are now ready to state the main result of this section (see Remark \ref{remark assumption a} below for some clarifications about the cumbersome assumption (a)).

\begin{theorem}\label{thm main diffeo}
  If $n\ge 2$, $1\le k\le n-1$, $\a\in(0,1]$, and $L>0$, then there exist $\mu_0\in(0,1)$ and $C_0>0$ (depending on $n$, $k$, $\a$, and $L$ only) with the following property.

  \noindent (a) Let $S_0$ be a compact connected $k$-dimensional $C^{2,1}$-manifold with boundary in $\R^n$, let $\widetilde{S}_0$ be a $k$-dimensional $C^{2,1}$-manifold in $\R^n$, and assume that
  \begin{eqnarray}\label{basta S0tilde limitata}
    \bd(S_0)\ne\emptyset\,,\qquad S_0\subset \widetilde{S}_0\,,&&\hspace{0.2cm}\diam(\widetilde{S}_0)\le L\,,
    \\
    \label{basta geodetica bordo S0}
    {\rm \dist}_{\bd(S_0)}(x,y)\le L\,|x-y|\,,&&\qquad \forall x,y\in \bd(S_0)\,\quad\mbox{(if $k\ge 2$)}\,,
    \\
    \label{basta geodetica S0}
    {\rm \dist}_{S_0}(x,y)\le L\,|x-y|\,,&&\qquad\forall x,y\in S_0\,,
    \\
    \label{basta geodetica S0tilde}
    {\rm \dist}_{\widetilde{S}_0}(x,y)\le L\,|x-y|\,,&&\qquad\forall x,y\in\widetilde{S}_0\,.
  \end{eqnarray}
  Moreover, let $\{\nu_0^{(i)}\}_{i=1}^{n-k}\subset C^{1,1}(\widetilde{S}_0;\SS^{n-1})$ be such that $\{\nu_0^{(i)}(x)\}_{i=1}^{n-k}$ is an orthonormal basis of $(T_x\widetilde{S}_0)^\perp$ for every $x\in \widetilde{S}_0$, and
  \begin{equation}
    \label{basta nu0 lipschitz}
    \max_{1\le i\le n-k}\|\nu^{(i)}_0\|_{C^{1,1}(\widetilde{S}_0)}\le L\,.
  \end{equation}

  \noindent (b) Let $S$ be a compact connected $k$-dimensional $C^{1,\a}$-manifold with boundary such that, for some $\rho\in(0,\mu_0^2)$, one has
  \begin{equation}
    \label{basta S C1alphaL and hd rho}
      \bd(S)\ne\emptyset\,,\qquad \|S\|_{C^{1,\a}}\le L\,,\qquad \hd(S,S_0)\le\rho\,.
  \end{equation}
  In addition:

  \medskip

  \noindent (i) if $k=1$, assume that, setting $\bd(S_0)=\{p_0,q_0\}$, $\bd(S)=\{p,q\}$, $f_0(p_0)=p$ and $f_0(q_0)=q$,
  \begin{equation}
    \label{basta ii k=1}
    \begin{split}
      &\hspace{3cm}\frac1{L}\le |p_0-q_0|\,,
      \\
      &\|f_0-\Id\|_{C^0(\bd(S_0))}+\|\nu^{co}_S(f_0)-\nu^{co}_{S_0}\|_{C^0(\bd(S_0))}\le\rho\,;
    \end{split}
  \end{equation}
  if $k\ge 2$, assume that there exists a $C^{1,\a}$-diffeomorphism $f_0$ between $\bd(S_0)$ and $\bd(S)$ with
  \begin{equation}
    \label{basta ii k maggiore 1}
    \begin{split}
      \|f_0\|_{C^{1,\a}(\bd(S_0))}\le L\,,
      \\
      \|f_0-\Id\|_{C^1(\bd(S_0))}\le \rho\,,
      \\
      \max_{1\le i\le n-k}\|\nu_S^{(i)}(f_0)-\nu_0^{(i)}\|_{C^0(\bd(S_0))}\le\rho\,,
      \\
      \|\nu_S^{co}(f_0)-\nu_{S_0}^{co}\|_{C^0(\bd(S_0))}\le\rho\,,
    \end{split}
  \end{equation}
  where $\{\nu_S^{(i)}\}_{i=1}^{n-k}$ is as in \eqref{basta S C1alpha}.

  \medskip

  \noindent (ii) there exists $\{\psi_i\}_{i=1}^{n-k}\subset C^{1,\a}([S_0]_\rho)$ such that, setting $\psi=\sum_{i=1}^{n-k}\psi_i\,\nu_{S_0}^{(i)}$, one has
  \begin{equation}\label{basta iii}
    \begin{split}
      &\hspace{0.9cm}[S]_{3\rho}\subset(\Id+\psi)([S_0]_\rho)\subset S\,,
      \\
      &\|\psi\|_{C^{1,\a}([S_0]_\rho)}\le L\,,\qquad \|\psi\|_{C^1([S_0]_\rho)}\le \rho\,.
    \end{split}
  \end{equation}
  Then, for every $\mu\in(\sqrt{\rho},\mu_0)$ there exists a $C^{1,\a}$-diffeomorphism $f$ between $S_0$ and $S$ such that
  \begin{eqnarray}
    \label{basta f uguale f0}
    f&=&f_0\,,\hspace{0.7cm}\qquad\mbox{on $\bd(S_0)$}\,,
    \\
    \label{basta f uguale id + psi}
    f&=&\Id+\psi\,,\qquad\mbox{on $[S_0]_\mu$}\,,
    \\
    \label{basta C1alpha}
    \|f\|_{C^{1,\a}(S_0)}&\le&C_0\,,
    \\
    \label{basta C0}
    \|f-\Id\|_{C^0(S_0)}&\le&C_0\,\big(\hd(S,S_0)+\|f_0-\Id\|_{C^1(\bd(S_0))}+\|\psi\|_{C^0([S_0]_\rho)}\big)\,,
    \\\label{basta C1}
    \|f-\Id\|_{C^1(S_0)}&\le&\frac{C_0}\mu\,\rho^\a\,,
    \\
    \label{basta f-Id tangenziale}
    \|\pi^{S_0}(f-\Id)\|_{C^1(S_0)}&\le&\frac{C_0}\mu\,\left\{
    \begin{split}
    \|(f-\Id)\cdot\nu_{S_0}^{co}\|_{C^0(\bd(S_0))}\,,\qquad\mbox{if $k=1$}\,,
    \\
    \|f_0-\Id\|_{C^1(\bd(S_0))}\,,\qquad\mbox{if $k\ge 2$}\,.
    \end{split}\right .
  \end{eqnarray}
\end{theorem}

\begin{remark}\label{remark bello C0}
{\rm One would expect the $C^0$ norm of $f_0-\Id$, and not its $C^1$-norm, to appear in \eqref{basta C0}. When $k=1$ we indeed prove this, as in that case $\bd(S_0)$ consists of two points and thus $\|f_0-\Id\|_{C^1(\bd(S_0))}=\|f_0-\Id\|_{C^0(\bd(S_0))}$. However, when $k\ge 2$, our construction of $f$ requires a preliminary rough extension of $f_0$ from $\bd(S_0)$ to $\R^n$ by means of Whitney's theorem. Although the $C^{1,\a}(\R^n)$ and $C^1(\R^n)$-norms of this rough extension will be controlled by the $C^{1,\a}(\bd(S_0))$ and $C^1(\bd(S_0))$-norms of $f_0$, because of how Whitney's extension procedure works, the $C^0(\R^n)$-norm will only be controlled by the full $C^1(\bd(S_0))$-norm of $f_0$.}
\end{remark}

\begin{remark}
  {\rm In order to obtain (in the spirit of \eqref{basta C0}) a more precise estimate than \eqref{basta C1}, that is, in order to replace $\rho^\a$  by some function of $\hd(S,S_0)$, $\|f_0-\Id\|_{C^1(\bd(S_0))}$, $\|\psi\|_{C^1([S_0]_\rho)}$ etc., one would need to relate to these quantities the smallest value of $\rho$ which makes the inclusion $[S]_{3\rho}\subset(\Id+\psi)([S_0]_\rho)$ in \eqref{basta iii} hold. More precisely, with such a control at hand one could prove such a strengthened form of \eqref{basta C1} by the same argument used below.}
\end{remark}

\begin{remark}\label{remark assumption a}
  {\rm We claim that assumption (a) can be replaced by
  \begin{equation}\label{assumption a}
    \begin{split}
      &\mbox{$S_0$ is a compact connected $k$-dimensional $C^{2,1}$-manifold with boundary in $\R^n$}
      \\
      &\mbox{and there exists $\{\nu_{S_0}^{(i)}\}_{i=1}^{n-k}\subset C^{1,1}(S_0;\SS^{n-1})$ such that}
      \\
      &\mbox{$\{\nu_{S_0}^{(i)}\}_{i=1}^{n-k}$ is an orthonormal basis of $(T_xS_0)^\perp$ for every $x\in S_0$}\,.
    \end{split}
  \end{equation}
  (In the case $k=1$, \eqref{assumption a} simply amounts in requiring that $S_0$ is a compact connected $C^{2,1}$-curve with boundary in $\R^n$.) More precisely, we claim that \eqref{assumption a} implies the existence of an extension $\widetilde{S}_0$ of $S_0$ and of a normal frame $\{\nu_0^{(i)}\}_{i=1}^{n-k}$ to $\widetilde{S}_0$ such that assumption (a) holds for a suitable value of $L$: correspondingly, the constants $C_0$ and $\mu_0$ given by the theorem will depend on the particular extension $\widetilde{S}_0$ we have considered. We now prove the claim. By compactness of $S_0$ one immediately finds a constant $L'$ such that \eqref{basta geodetica bordo S0} and \eqref{basta geodetica S0} hold with $L'$ in place of $L$, $\diam(S_0)\le L'$, and $\|\nu_{S_0}^{(i)}\|_{C^{1,1}(S_0)}\le L'$. Now let us fix $\ell=1,...,n-k$, and for $x\in S_0$ set
  \[
  F^\00(x)=0\,,\qquad F^{e_i}(x)=\nu_{S_0}^{(\ell)}(x)\cdot e_i\,,\qquad F^{e_i+e_j}(x)=e_i\cdot \nabla^{S_0}\nu_{S_0}^{(\ell)}(x)[e_j]\,.
  \]
  By compactness of $S_0$ we find that $\F=\{F^\kk\}_{|\kk|\le2}\in WJ^{2,1}(S_0)$. Hence, by arguing as in the proof of Proposition \ref{proposition costruzione extbyfol}, there exist $d_{S_0}\in C^{2,1}(\R^n;\R^{n-k})$ and $\e_0>0$ such that
  \begin{equation}\label{basta dS0}
    \begin{split}
    &\mbox{$d_{S_0}(x)=\00$ and $\nabla d_{S_0}(x)=\sum_{i=1}^{n-k}e_i\otimes\nu_{S_0}^{(i)}(x)$ for every $x\in S_0$}\,,
    \\
    &\mbox{$I_{\e_0}(S_0)\cap\{d_{S_0}=\00\}$ is a $k$-dimensional $C^{2,1}$-manifold in $\R^n$}\,,
    \\
    &\max\big\{\e_0^{-1}\,,\|d_{S_0}\|_{C^{2,1}(\R^n)}\big\}\le C\,,
  \end{split}
  \end{equation}
  where $C$ depends on $n$, $k$ and $S_0$ only. Let us set $\widetilde{S}_0=I_{\e_0}(S_0)\cap\{d_{S_0}=\00\}$. Up to further decreasing the value of $\e_0$ one immediately deduces \eqref{basta S0tilde limitata} and \eqref{basta geodetica S0tilde} for some value of $L$. Moreover, by construction, for every $i=1,...,n-k$ there exists $\{h_{i,j}\}_{j=1}^n\subset C^{1,1}(\R^n)$ such that
  \[
  \nabla d_{S_0}(x)=\sum_{i=1}^{n-k}e_i\otimes\big(\sum_{j=1}^nh_{i,j}(x)\,e_j\big)\,,\qquad\forall x\in\R^n\,.
  \]
  Up to further decreasing the value of $\e_0$ we can define $\{\nu_0^{(i)}\}_{i=1}^{n-k}\in C^{1,1}(\widetilde{S}_0;\SS^{n-1})$ in such a way that \eqref{basta nu0 lipschitz} holds by simply applying the Gram-Schmidt orthogonalization process to the vectors $\{\sum_{j=1}^nh_{i,j}(x)\,e_j\}_{i=1}^{n-k}$.}
\end{remark}

\begin{proof}[Proof of Theorem \ref{thm main diffeo}]
  In the following we denote by $C$ a constant which may depend on $n$, $k$, $\a$ and $L$ only. We start our argument by extending $S$ into a larger manifold $\widetilde{S}$. More precisely, by $\|S\|_{C^{1,\a}}\le L$ and Proposition \ref{proposition costruzione extbyfol}, there exist $d_S\in C^{1,\a}(\R^n;\R^{n-k})$ and $\e>0$ such that
  \begin{equation}\label{basta dS}
    \begin{split}
    &\mbox{$d_S(x)=\00$ and $\nabla d_S(x)=\sum_{i=1}^{n-k}e_i\otimes\nu_S^{(i)}(x)$ for every $x\in S$}\,,
    \\
    &\mbox{$I_{\e}(S)\cap\{d_S=\00\}$ is a $k$-dimensional $C^{1,\a}$-manifold in $\R^n$}\,,
    \\
    &\max\big\{\e^{-1}\,,\|d_S\|_{C^{1,\a}(\R^n)}\big\}\le C\,,
  \end{split}
  \end{equation}
  where $\00=(0,...,0)\in\R^{n-k}$. We shall use $d_S$ to locate the position of $S$ in $\R^n$ (see the proof of the claim below). We set
  \[
  \widetilde{S}=I_{\e}(S)\cap\{d_S=\00\}\,,
  \]
%  and define $\nu^{(j)}\in C^{0,\a}(\widetilde{S};\SS^{n-1})$ with the property that
%  \begin{equation}\label{frenet0 forse}
%  \nu^{(j)}=\nu_{S}^{(j)}\qquad\mbox{on $S$}\,,\qquad \max_{1\le j\le n-k}\|\nu^{(j)}\|_{C^{0,\a}(\widetilde{S})}\le C\,.
%  \end{equation}
  and we record for future use that, by \eqref{basta dS}, if $v\in \SS^{n-1}$, $\de>0$, and $x\in S$, then
  \begin{eqnarray}
    \label{dg 1}
    |\nabla d_S(x)v|\le C\,\de\,,&&\qquad \mbox{if $|\pi^{S}_x(v)|\ge 1-\de$}\,,
    \\\label{dg 2}
    |\nabla d_S(x)v|\ge 1-C\,\de\,,&&\qquad \mbox{if $|\pi^{S}_x(v)|\le \de$}\,,
    %,
%    \\\label{dg 3}
%    \Big|\bigwedge_{i=1}^{n-k}\nabla d_S(x)v_i\Big|&\ge&1-C\,\de\,,\qquad\hspace{0.4cm}\mbox{if $v_i\cdot v_j=\de_{i,j}$ and $|\pi^{S}_x(v_i)|\le \de$}\,.
  \end{eqnarray}
  Next, we note that
  \begin{eqnarray}
  \label{basta stima quadratica normali}
    \max_{1\le i\le n-k}|\nu^{(i)}_0(x)\cdot(y-x)|\le C\,|\pi_x^{\widetilde{S}_0}(y-x)|^2\,,\qquad\forall x\in\widetilde{S}_0\,, y\in B_{x,1/C}\cap\widetilde{S}_0\,.
  \\
  \label{basta condizioni su widetilde S0}
    \begin{split}
    &|y-x|\le 2\,|\pi^{\widetilde{S}_0}_x(y-x)|\,,\qquad\forall x\in \widetilde{S}_0\,, y\in B_{x,1/C}\cap \widetilde{S}_0\,,
    \\
    &\|\pi^{\widetilde{S}_0}_x-\pi^{\widetilde{S}_0}_y\|\le C\,\,|x-y|\,,\qquad \forall x,y\in \widetilde{S}_0\,.
    \end{split}
  \end{eqnarray}
  Indeed, \eqref{basta stima quadratica normali} follows from \eqref{basta nu0 lipschitz} and the fact that $\{\nu^{(i)}(x)\}_{i=1}^{n-k}$ is an orthonormal basis of $(T_x\widetilde{S}_0)^\perp$, the first condition in \eqref{basta condizioni su widetilde S0} follows from \eqref{basta stima quadratica normali}, and the second condition in \eqref{basta condizioni su widetilde S0} is an immediate consequence of $[\nu^{(i)}]_{C^{0,1}(\widetilde{S}_0)}\le L$. We now set
  \[
  U_{x,\de}=\widetilde{S}_0\cap B_{x,\de}\,,
  \quad K_\de=I_{\de}(\bd(S_0))\cap \widetilde{S}_0\,,\quad K_\de^+=I_{\de}(\bd(S_0))\cap S_0\,,
  \qquad x\in\widetilde{S}_0\,,\de>0\,,
  \]
  and then we make the following claim:
%
%  We consider a unit speed parametrization $\Phi_0$ of $\widetilde\g_0$, that is $\Phi_0\in C^{2,1}(I;\R^n)$ with
%  \begin{equation}
%    \label{alpha0 concorde tau}
%    \widetilde\g_0=\{\Phi_0(s):s\in I\}\,,\qquad  \Phi_0'(s)=\tau_0(\Phi_0(s))\,,\qquad\forall s\in I\,,
%  \end{equation}
%  where $I\subset\R$ is an interval such that $\H^1(I)=\H^1(\widetilde\g_0)$. Clearly, by \eqref{basta condizioni su widetilde S0},
%  \begin{equation}
%    \label{frenet}
%    \|\tau_0(\Phi_0)\|_{C^{1,1}(I)}+\max_{1\le i\le n-1}\|\nu_0^{(i)}(\Phi_0)\|_{C^{1,1}(I)}\le C\,.
%  \end{equation}
%  If $s_0\in I$ is such that $\Phi_0(s_0)=p_0$, then we set
%  \[
%  U_{p_0,t}=\Phi_0(I\cap (s_0-t,s_0+t))\subset\widetilde{\g}_0\,.
%  \]
%
%
%  and notice that
%  \begin{equation}
%    \label{usa anche me}
%    B_{x_0,t}\cap \widetilde{\g}_0\subset U_{p_0,t}\subset B_{x_0,C\,t}\cap \widetilde{\g}_0\,,\qquad\mbox{whenever $(s_0-t,s_0+t)\subset I$}\,.
%  \end{equation}

  \medskip

  \noindent {\it Claim}: There exists $\eta_0$ depending on $n$, $k$, $\a$ and $L$ only such that, if $\mu_0$ is small enough with respect to $\eta_0$, then one can construct $f:K_{\eta_0}\to \widetilde{S}$ with
  \begin{eqnarray}
    \label{f uguale f0}
    f&=&f_0\,,\hspace{0.7cm}\qquad\mbox{on $\bd(S_0)$}\,,
    \\
    \label{f uguale a psi}
    f&=&\Id+\psi\,,\qquad\mbox{on $K_{\eta_0}^+\setminus K_\mu$}\,,
    \\
    \label{curvette f C11}
    \|f\|_{C^{1,\a}(K_{\eta_0})}&\le& C\,,
    \\
    \label{curvette f va a zero in C0}
    \|f-\Id\|_{C^0(K_{\eta_0}^+)}&\le& C\,\big(\hd(S,S_0)+\|f_0-\Id\|_{C^1(\bd(S_0))}\big)\,,
    \\
    \label{curvette f va a zero in C1}
    \|f-\Id\|_{C^1(K_{\eta_0}^+)}&\le& \frac{C}\mu\,\rho^\a\,,
    %\Big(\hd(S,S_0)+\|f_0-\Id\|_{C^1(\bd(S_0))}+\|\psi\|_{C^1([S_0]_\rho)}
%    \\
%    \nonumber
%    &&\hspace{3cm}+\max_{1\le i\le n-k}\|\nu_S^{(i)}(f_0)-\nu_{S_0}^{(i)}\|_{C^0(\bd(S_0))}\Big)\,,
    \\
    \label{curvette f tangenziale displacement}
    \|\pi^{\widetilde{S}_0}(f-\Id)\|_{C^1(K_{\eta_0})}&\le& \frac{C}\mu\,
    \left\{\begin{split}
    \|(f-\Id)\cdot\nu_{S_0}^{co}\|_{C^0(\bd(S_0))}\,,\qquad\mbox{if $k=1$}\,,
    \\
    \|f_0-\Id\|_{C^1(\bd(S_0))}\,,\qquad\mbox{if $k\ge 2$}\,,
    \end{split}\right .
    \\
    \label{curvette f jacobiano}
    J^{\widetilde{S}_0}f&\ge&\frac12\,,\qquad\mbox{on $K_{\eta_0}$}\,,
    \\
    \label{curvette attacco normale}
    \pi^{\widetilde{S}_0}(f-\Id)&=&0\,,\hspace{0.1cm}\qquad \mbox{on $K_{\eta_0}\setminus K_{\mu}$}\,,
    \\
    \label{curvette inclusione}
    f(K_{\eta_0}^+)&\subset&S\,.
  \end{eqnarray}
  {\it Given the claim, the theorem follows}: Indeed, if one extends $f$ from $K_{\eta_0}$ to $K_{\eta_0}\cup S_0$ by setting $f=\Id+\psi$ on $S_0\setminus K_{\eta_0}$, then thanks to \eqref{f uguale a psi}, \eqref{basta iii} and \eqref{curvette f C11} we find that $f\in C^{1,\a}(K_{\eta_0}\cup S_0;\R^n)$ and that \eqref{basta f uguale f0}, \eqref{basta f uguale id + psi} and \eqref{basta C1alpha} hold; similarly, \eqref{basta C0} and \eqref{basta C1} follow by \eqref{curvette f va a zero in C0} and \eqref{curvette f va a zero in C1}, while \eqref{curvette f tangenziale displacement} and \eqref{curvette attacco normale} imply \eqref{basta f-Id tangenziale}. By Theorem \ref{thm inverse fnct uniform}, \eqref{basta geodetica S0}, \eqref{basta condizioni su widetilde S0}, \eqref{curvette f C11} and \eqref{curvette f jacobiano} there exists $r_0>0$ (depending on $n$, $k$, $\a$ and $L$ only) such that if $\|f-\Id\|_{C^0(S_0)}\le r_0$ (as we can entail thanks to \eqref{curvette f va a zero in C0}, \eqref{basta S C1alphaL and hd rho}, \eqref{basta ii k maggiore 1}, and \eqref{basta iii} provided we take $\mu_0^2\le r_0$), then $f$ is a $C^{1,\a}$-diffeomorphism between $\INT(S_0)$ and $f(\INT(S_0))$. Let us set
  \[
  S^*=\cl(f(\INT(S_0)))\,,
  \]
  so that $S^*\subset S$ by \eqref{basta iii} and \eqref{curvette inclusione}. Moreover, $S^*$ is a compact connected $k$-dimensional $C^{1,\a}$-manifold with boundary in $\R^n$ with
  \[
  \INT(S^*)=f(\INT(S_0))\,,\qquad \bd(S^*)=S^*\setminus f(\INT(S_0))=f(\bd(S_0))=\bd(S)\,,
  \]
  thus, by connectedness of $S$, one has $S=S^*=f(S_0)$. Indeed, in order to obtain a contradiction it suffices to consider $y\in \INT(S)\setminus S^*$, together with a curve $\g$ with $\INT(\g)\subset \INT(S)\setminus S^*$, i.e. which lives in the connected component of $\INT(S)\setminus S^*$ determined by $y$, such that $\bd(\g)=\{y,x\}$ with $x\in\bd(S)$.

  \medskip

  \noindent {\it Proof of the claim}: We first describe the case $k\ge 2$, and then explain the minor variants needed when $k=1$. We fix $\phi\in C^\infty(\R^n\times(0,\infty);[0,1])$ such that, setting $\phi_\mu=\phi(\cdot,\mu)$ for $\mu>0$,
  \begin{eqnarray}
  \label{cutoff phis}
  &&\phi_\mu\in C^\infty_c(I_\mu(\bd(S_0)))\,,\qquad\mbox{$\phi_{\mu}=1$ on $I_{\mu/2}(\bd(S_0))$}\,,
  \\
  \label{nabla phis}
  &&|\nabla\phi_\mu(x)|\le\frac{C}\mu\,,\qquad |\nabla^2\phi_\mu(x)|\le\frac{C}{\mu^2}\,,\qquad\forall (x,\mu)\in\R^n\times(0,\infty)\,.
  \end{eqnarray}
  Let us define $\bar{a}_i:\bd(S_0)\to\R$, $i=1,...,n-k$, and $\bar{b}:\bd(S_0)\to\R^n$ by setting
  \begin{equation}
    \label{sxtx}
      \bar{a}_i(x)=(f_0(x)-x)\cdot\nu_0^{(i)}(x)\,,\qquad \bar{b}(x)=f_0(x)-x-\sum_{i=1}^{n-k}a_i(x)\,\nu_0^{(i)}(x)\,,\qquad x\in\bd(S_0)\,,
  \end{equation}
  so that, trivially,
  \begin{equation}
    \label{f0 decompo 1}
      f_0(x)=x+\bar{b}(x)+\sum_{i=1}^{n-k}\bar{a}_i(x)\,\nu^{(i)}_0(x)\,,\qquad\forall x\in\bd(S_0)\,.
  \end{equation}
  By \eqref{basta ii k maggiore 1} one has
  \begin{equation}\label{ai b k maggiore 2}
  \begin{split}
    &\|\bar{a}_i\|_{C^{1,\a}(\bd(S_0))}+\|\bar{b}\|_{C^{1,\a}(\bd(S_0))}\le C\,,
    \\
    &\|\bar{a}_i\|_{C^1(\bd(S_0))}+\|\bar{b}\|_{C^1(\bd(S_0))}\le C\,\|f_0-\Id\|_{C^1(\bd(S_0))}\le C\,\rho\,,
  \end{split}
  \end{equation}
  %while for $k=1$ (and thus $\bd(S_0)=\{p_0,q_0\}$) we simply need to notice that by \eqref{basta ii k=1}
%  \begin{equation}\label{ai b k maggiore 2}
%  \|\bar{a}_i\|_{C^0(\bd(S_0))}\le\rho\,,\qquad \|\bar{b}\|_{C^0(\bd(S_0))}=\|(f-\Id)\cdot\nu_{S_0}^{co}\|_{C^0(\bd(S_0))}\le\rho\,.
%  \end{equation}
%  When $k\ge 2$, b
  By Proposition \ref{proposition estensione bordo di S} and by \eqref{basta geodetica bordo S0} we find $a_i\in C^{1,\a}(\R^n)$, $i=1,...,n-k$, and $b\in C^{1,\a}(\R^n;\R^n)$ such that
  \begin{equation}
    \begin{split}
      \label{ai bi estesi e stime}
      &\mbox{$a_i=\bar{a}_i$ and $b=\bar{b}$}\,,\qquad\mbox{on $\bd(S_0)$}\,,
      \\
      \|a_i\|_{C^{1,\a}(\R^n)}+\|b\|_{C^{1,\a}(\R^n)}\le C\,,&\qquad\|a_i\|_{C^1(\R^n)}+\|b\|_{C^1(\R^n)}\le C\,\|f_0-\Id\|_{C^1(\bd(S_0))}\,.
    \end{split}
  \end{equation}
  Correspondingly we define $G\in C^{1,\a}(\widetilde{S}_0;\R^n)$ by setting
  \begin{equation}
    \label{defi G}
      G(x)=\phi_\mu(x)\,b(x)+\sum_{i=1}^{n-k}a_i(x)\,\nu_0^{(i)}(x)\,,\qquad x\in \widetilde{S}_0\,.
  \end{equation}
  By \eqref{cutoff phis} and \eqref{f0 decompo 1},
  \begin{equation}
    \label{f0 decompo 2}
      f_0(x)=x+G(x)\,,\qquad\forall x\in\bd(S_0)\,,
  \end{equation}
  while \eqref{nabla phis}, \eqref{ai bi estesi e stime} and $\rho\le\mu^2$ give
  \begin{equation}
    \label{stime G}
    \|G\|_{C^{1,\a}(\widetilde{S}_0)}\le C\,,\qquad
    \left\{
    \begin{split}
          &\|G\|_{C^0(\widetilde{S}_0)}\le C\,\|f_0-\Id\|_{C^1(\bd(S_0))}\,,
          \\
          &\|G\|_{C^1(\widetilde{S}_0)}\le\frac{C}\mu\,\|f_0-\Id\|_{C^1(\bd(S_0))}\le \frac{C}{\mu}\,\rho\le C\,\mu_0\,.
    \end{split}
    \right .
  \end{equation}
  We now define $F\in C^{1,\a}(\widetilde{S}_0\times(-1,1)^{n-k};\R^n)$ by setting, for $(x,\tt)\in \widetilde{S}_0\times(-1,1)^{n-k}$,
  \begin{equation}
  \label{definition of F}
  \begin{split}
  F(x,\tt)&=x+\phi_{\mu}(x)\,b(x)+\sum_{i=1}^{n-k}(a_i(x)+t_i)\,\nu_0^{(i)}(x)
  \\
  &=x+G(x)+\sum_{i=1}^{n-k}t_i\,\nu_0^{(i)}(x)\,,
  \end{split}
  \end{equation}
  and then exploit $d_S\in C^{1,\a}(\R^n;\R^{n-k})$ to define $u\in C^{1,\a}(\widetilde{S}_0\times(-1,1)^{n-k};\R^{n-k})$ as
  \[
  u(x,\tt)=d_S(F(x,\tt))\,,\qquad (x,\tt)\in \widetilde{S}_0\times(-1,1)^{n-k}\,.
  \]
  By \eqref{f0 decompo 2},
  \begin{equation}
    \label{F fa zero su bdS0}
      F(x,\00)=f_0(x)\,,\qquad\forall x\in\bd(S_0)\,,
  \end{equation}
  which combined with $S\subset\{d_S=\00\}$ implies
  \begin{equation} \label{frilli1}
  u(x,\00)=\00\,,\qquad\forall x\in\bd(S_0)\,.
  \end{equation}
  By \eqref{stime G} and by \eqref{basta dS} one has
    \begin{equation}
    \label{F ed u C1alpha}
    \|F\|_{C^{1,\a}(\widetilde{S}_0\times(-1,1)^{n-k})}\le C\,,
    \qquad
    \|u\|_{C^{1,\a}(\widetilde{S}_0\times(-1,1)^{n-k})}\le C\,.
  \end{equation}
  We claim that if $\mu_0$ is small enough (and up to identify $(n-k)$-vectors in $\R^{n-k}$ with real numbers, with the convention that $e_1\wedge\dots\wedge e_{n-k}=1$), then
  \begin{equation}
    \label{frilli2}
     \bigwedge_{i=1}^{n-k}\frac{\pa u}{\pa t_i}(x,\00)\ge\frac12\,,\qquad\forall x\in\bd(S_0)\,.
  \end{equation}
  Indeed, by \eqref{F fa zero su bdS0}, \eqref{basta dS}, and by $\pa F/\pa t_i(x,t)=\nu_0^{(i)}(x)$ we find that
  \begin{equation}
    \label{jacobiano}
    \bigwedge_{i=1}^{n-k}\frac{\pa u}{\pa t_i}(x,\00)=  \bigwedge_{i=1}^{n-k} \nabla d_S(f_0(x))[\nu_0^{(i)}(x)]
    =\prod_{i=1}^{n-k}\nu^{(i)}(f_0(x))\cdot\nu_0^{(i)}(x)\,,\qquad\forall x\in\bd(S_0)\,,
  \end{equation}
  so that \eqref{frilli2} follows by \eqref{basta ii k maggiore 1} provided $\mu_0$ is small enough (recall that $\rho<\mu_0^2$). By \eqref{frilli1}, \eqref{F ed u C1alpha}, \eqref{frilli2} and Theorem \ref{lemma implicit function} (that can be applied thanks to \eqref{basta geodetica S0tilde} and \eqref{basta condizioni su widetilde S0}) there exists a positive constant $\eta_0>0$ (depending on $n$, $k$, $\a$, and $L$) such that for each $x_0\in\bd(S_0)$ one can find $\zeta_{x_0}\in C^{1,\a}(U_{x_0,\eta_0};\R^{n-k})$ with
  \begin{eqnarray}  \label{cazzi e mazzi20}
  &&u(x,\zeta_{x_0}(x))=\00\,,\qquad \forall x\in U_{x_0,\eta_0}\,,
  \\
  \label{zeta0 C1alpha}
  &&\zeta_{x_0}(x_0)=\00\,,
  \qquad  \|\zeta_{x_0}\|_{C^{1,\a}(U_{x_0,\eta_0})}\le C\,.
  \end{eqnarray}
  Note that we had to put constraint on the smallness of $\mu_0$ to assert the existence of $\eta_0$. We are of course free to decrease the value of $\mu_0$ without affecting the value of $\eta_0$. We shall require that $\mu_0$ is suitably smaller than $\eta_0$, precisely that $\mu_0\le\eta_0/C_*$ for some suitable $C_*=C_*(n,k,\a,L)$, and we shall further decrease the value of $\eta_0$ depending on $n$, $k$, $\a$ and $L$ only.

  Let us now prove that if $x_0,x_1\in\bd(S_0)$, then
  \begin{equation}
    \label{zeta0 zeta1}
    \zeta_{x_0}(x)=\zeta_{x_1}(x)\,,\qquad\forall x\in U_{x_0,\eta_0}\cap U_{x_1,\eta_0}\,.
  \end{equation}
  Indeed, by $[\zeta_{x_0}]_{C^{0,1}(U_{x_0,\eta_0})}\le C$ and $\zeta_{x_0}(x_0)=\00$ one has
  \begin{equation}
    \label{zeta eta0}
      \|\zeta_{x_0}\|_{C^0(U_{x_0,\eta_0})}\le C_1\,\eta_0\,,
  \end{equation}
  for some constant $C_1$ depending on $n$, $k$, $\a$ and $L$ only. In particular, up to further decreasing the value of $\eta_0$ in dependence of the $C^{1,\a}$-bound on $u$ in \eqref{F ed u C1alpha} and of $C_1$, we can entail
  \begin{equation}
    \label{frilli2x}
     \bigwedge_{i=1}^{n-k}\frac{\pa u}{\pa t_i}(x,\tt)\ge\frac13\,,\qquad\forall (x,\tt)\in U_{x_0,\eta_0}\times(-C_1\eta_0,C_1\,\eta_0)^{n-k}\,.
  \end{equation}
  Now, if $x\in U_{x_0,\eta_0}\cap U_{x_1,\eta_0}$ and we set $A_0=(-C_1\eta_0,C_1\,\eta_0)^{n-k}$, then by \eqref{F ed u C1alpha} and \eqref{frilli2x} one has $u(x,\cdot)\in C^{1,\a}(A_0;\R^{n-k})$ with
  \[
  \|u(x,\cdot)\|_{C^{1,\a}(A_0)}\le C\,,\qquad J^{A_0}u(x,\cdot)\ge\frac13\qquad\mbox{on $A_0$}\,.
  \]
  By Theorem \ref{thm inverse fnct uniform}, there exists $\e_0$ (depending on $n$, $k$, $\a$ and $L$ only) such that $u(x,\cdot)$ is invertible on $A_0^*=(-\e_0,\e_0)^{n-k}$. By requiring that $C_1\,\eta_0<\e_0$, we thus find that $u(x,\cdot)$ is invertible on $A_0$, and since $\zeta_{x_0}(x),\zeta_{x_1}(x)\in A_0$ with $u(x,\zeta_{x_0}(x))=u(x,\zeta_{x_1}(x))$ by \eqref{cazzi e mazzi20}, we deduce \eqref{zeta0 zeta1}. Moreover, by an entirely analogous argument, we deduce from \eqref{frilli1} and \eqref{cazzi e mazzi20} that
  \begin{equation}
    \label{zeta fa zero}
    \zeta_{x_0}(x)=\00\,,\qquad\forall x\in\bd(S_0)\cap U_{x_0,\eta_0}\,.
  \end{equation}

  By \eqref{cazzi e mazzi20}, \eqref{zeta0 C1alpha} \eqref{zeta0 zeta1}, and \eqref{zeta fa zero}, if we define $\zeta\in C^{1,\a}(K_{\eta_0};\R^{n-k})$ (recall that $K_{\eta_0}=I_{\eta_0}(\bd(S_0))\cap\widetilde{S}_0$) by setting $\zeta=\zeta_{x_0}$ on $U_{x_0,\eta_0}$ for each $x_0\in\bd(S_0)$, then
  \begin{eqnarray}
  \label{cazzi e mazzi2}
  &&u(x,\zeta(x))=\00\quad \forall x\in K_{\eta_0}\,,\qquad \zeta(x)=\00\quad\forall x\in\bd(S_0)\,,
  \\
  \label{zeta C0 e C1alpha}
  &&\|\zeta\|_{C^0(K_{\eta_0})}\le C_1\,\eta_0\,,\qquad \|\zeta\|_{C^{1,\a}(K_{\eta_0})}\le C\,.
  \end{eqnarray}
  We finally set
  \begin{eqnarray}\label{formula fgamma}
    f(x)=F(x,\zeta(x))=x+G(x)+\sum_{i=1}^{n-k}\zeta_i(x)\,\nu_0^{(i)}(x)\,,\qquad x\in K_{\eta_0}\,,
  \end{eqnarray}
  where $\zeta_i=e_i\cdot\zeta$, and show that $f$ has the required properties. By \eqref{F fa zero su bdS0} and \eqref{cazzi e mazzi2} we prove \eqref{f uguale f0}, while \eqref{curvette f C11} follows from \eqref{F ed u C1alpha} and \eqref{zeta C0 e C1alpha}. Similarly, \eqref{cazzi e mazzi2} and the definition of $u$ give
  \begin{equation}
    \label{curvette inclusione tilde p}
      f(K_{\eta_0})\subset\{d_S=\00\}\,.
  \end{equation}
  By \eqref{f uguale f0}, \eqref{curvette f C11}, and $f(\bd(S_0))=f_0(\bd(S_0))=\bd(S)$, we find that $f(K_{\eta_0})\subset I_{C\,\eta_0}(\bd(S))$, so that, up to decrease $\eta_0$ and thanks to $\e>C^{-1}$ (recall \eqref{basta dS}), we can entail $f(K_{\eta_0})\subset I_\e(S)$. In particular \eqref{curvette inclusione tilde p} gives
  \begin{equation}
    \label{curvette inclusione tilde}
    f(K_{\eta_0})\subset\widetilde{S}\,.
  \end{equation}
  By \eqref{formula fgamma} and \eqref{defi G},
  \begin{equation}
    \label{your mom1}
      \pi^{\widetilde{S}_0}(f-\Id)(x)=\phi_{\mu}(x)\,b(x)\,,\qquad\forall x\in K_{\eta_0}\,,
  \end{equation}
  so that \eqref{curvette attacco normale} follows by $\spt\,\phi_{\mu}\cc I_{\mu}(\bd(S_0))$. By differentiating \eqref{your mom1} along $\tau\in T_x\widetilde{S}_0$ we find
  \[
  \nabla^{\widetilde{S}_0}[\pi^{\widetilde{S}_0}(f-\Id)](x)\,[\tau]=\Big(\nabla\phi_{\mu}(x)\cdot\tau\Big)\,b(x)
  +\phi_\mu(x)\nabla^{\widetilde{S}_0}b(x)[\tau]\,,
  \]
  which implies \eqref{curvette f tangenziale displacement} (recall we are addressing the case $k\ge 2$) once combined with \eqref{ai bi estesi e stime} and \eqref{your mom1}. By differentiating \eqref{formula fgamma} along $\tau\in T_x\widetilde{S}_0$ we find that
  \begin{eqnarray}
  \label{curvette grad f}
  \nabla^{\widetilde{S}_0}f(x)[\tau]&=&\tau+\sum_{i=1}^{n-k}  \nabla^{\widetilde{S}_0}\zeta_i(x)[\tau]\,\nu^{(i)}_0(x)
  \\\nonumber
  &&+  \nabla^{\widetilde{S}_0}G(x)[\tau]+\sum_{i=1}^{n-k}\zeta_i(x)  \nabla^{\widetilde{S}_0}\nu^{(i)}_0(x)[\tau]\,.
  \end{eqnarray}
  The first term on the second line is bounded by $C\mu_0$ thanks to \eqref{stime G}, while the second term on the second line is bounded by $C\,\eta_0$ thanks to \eqref{basta nu0 lipschitz} and \eqref{zeta C0 e C1alpha}, so that, as we are requiring $\mu_0\le \eta_0/C_*\le\eta_0$,
  \begin{equation}
    \label{diavolo}
      \Big|\nabla^{\widetilde{S}_0}f(x)[\tau]-\Big(\tau+\sum_{i=1}^{n-k}  \nabla^{\widetilde{S}_0}\zeta_i(x)[\tau]\,\nu^{(i)}_0(x)\Big)\Big|\le\,C\,\eta_0\,.
  \end{equation}
  Thus, if $\{\tau_i\}_{i=1}^k$ is an orthonormal basis of $T_x\widetilde{S}_0$, then
  \begin{eqnarray}\nonumber
      J^{\widetilde{S}_0}f(x)\ge\Big|\bigwedge_{i=1}^k \Big(\tau_i+\sum_{j=1}^{n-k}  \nabla^{\widetilde{S}_0}\zeta_j(x)[\tau_i]\,\nu^{(j)}_0(x)\Big)\Big|-C\,\eta_0
      %\ge \Big|\Big(\nabla^{\widetilde{\g}_0}f_{p_0}(x)\tau_0(x)\Big)\cdot\tau_0(x)\Big|
%      \\    \label{nuota2}
%      &\ge&1-C\,\Big(\frac{|b|}{\mu}+\sum_{i=1}^{n-1}|a_i|+|\zeta(x)|\Big)\ge1-C\,(\mu_0+\eta_0)\ge\frac12\,,
  \end{eqnarray}
  Since $\bigwedge_{i=1}^k\tau_i$ is orthogonal to $\bigwedge_{i\in I}\tau_i\wedge\bigwedge_{j\in J}\nu_0^{(j)}(x)$ for every $I\subset\{1,...,k\}$ and $J\subset\{1,...,n-k\}$ with $\#I+\#J=k$ and $\#I<k$, by projecting over $\bigwedge_{i=1}^k\tau_i$ one finds
  \begin{eqnarray}\label{nuota2}
      J^{\widetilde{S}_0}f(x)&\ge&\Big|\bigwedge_{i=1}^k \Big(\tau_i+\sum_{j=1}^{n-k}  \nabla^{\widetilde{S}_0}\zeta_j(x)[\tau_i]\,\nu^{(j)}_0(x)\Big)\cdot\bigwedge_{i=1}^k\tau_i\Big|-C\,\eta_0= 1-C\,\eta_0\ge\frac12\,,
      %\ge \Big|\Big(\nabla^{\widetilde{\g}_0}f_{p_0}(x)\tau_0(x)\Big)\cdot\tau_0(x)\Big|
%      \\    \label{nuota2}
%      &\ge&1-C\,\Big(\frac{|b|}{\mu}+\sum_{i=1}^{n-1}|a_i|+|\zeta(x)|\Big)\ge1-C\,(\mu_0+\eta_0)\ge\frac12\,,
  \end{eqnarray}
  provided $\eta_0$ is small enough; this proves \eqref{curvette f jacobiano}. Again by \eqref{diavolo} we find that if $x\in\bd(S_0)$, then
  \[
  \nabla^{\widetilde{S}_0}f(x)[\nu^{co}_{S_0}(x)]\cdot\nu^{co}_S(f(x))\ge
  \nu^{co}_{S_0}(x)\cdot\nu^{co}_S(f(x))-C\,\max_{1\le i\le n-k}|\nu^{(i)}_0(x)\cdot\nu^{co}_S(f(x))|-C\,\eta_0\,.
  \]
  By \eqref{basta ii k maggiore 1}, $\nu^{co}_{S_0}(x)\cdot\nu^{co}_S(f(x))\ge 1-C\,\rho$ and $|\nu^{(i)}_0(x)\cdot\nu^{co}_S(f(x))|\le C\,\rho$, so that
  \begin{eqnarray}\label{nuota3}
  \nabla^{\widetilde{S}_0}f(x)[\nu^{co}_{S_0}(x)]\cdot\nu^{co}_S(f(x))\ge\frac12\,,\qquad\forall x\in\bd(S_0)\,,
  \end{eqnarray}
  provided $\eta_0$ (thus $\rho\le\mu_0^2$) is small enough. By \eqref{f uguale f0}, \eqref{curvette inclusione tilde}, and \eqref{curvette f jacobiano}, for every $x\in\bd(S_0)$ one has
  \[
  \nabla^{\widetilde{S}_0}f(x)[T_x\widetilde{S}_0]=T_{f(x)}\widetilde{S}\,,\qquad
  \nabla^{\widetilde{S}_0}f(x)[T_x(\bd(S_0))]=T_{f(x)}(\bd(S))\,,
  \]
  so that \eqref{nuota3} gives
  \[
  \nabla^{\widetilde{S}_0}f(x)\Big[\big\{v\in T_x\widetilde{S}_0:v\cdot\nu_{S_0}^{co}(x)\le0\big\}\Big]
  =\big\{w\in T_{f(x)}\widetilde{S}:w\cdot\nu_{S}^{co}(f(x))\le0\big\}\,.
  \]
  By combining this fact with \eqref{curvette inclusione tilde} we deduce \eqref{curvette inclusione} (up to possibly further decreasing $\eta_0$ in dependence of the bound in \eqref{curvette f C11}). We are thus left to prove \eqref{f uguale a psi}, \eqref{curvette f va a zero in C0} and \eqref{curvette f va a zero in C1}.

  We first prove \eqref{curvette f va a zero in C0}. By \eqref{curvette inclusione} one has
  \begin{equation}
    \label{rhorho}
  \hd(S,S_0)\ge\dist(f(x),S_0)\,,\qquad \forall x\in K_{\eta_0}^+\,.
  \end{equation}
  Let $\e_0>0$ be the inverse of the maximum of the largest principal curvature of $S_0$, so that, by \eqref{basta nu0 lipschitz}, $\e_0$ depends on $L$ only. Then
  \begin{equation}
    \label{intorno normale curvette}
      \dist\Big(x+\sum_{i=1}^{n-k}t_i\,\nu_0^{(i)}(x),S_0\Big)=|\tt|\,,\qquad\forall x\in S_0\,,|\tt|<\e_0\,.
  \end{equation}
  By $\spt\phi_\mu\cc I_\mu(\bd(S_0))$ and by \eqref{formula fgamma}
  \[
  f(x)=x+\sum_{i=1}^{n-k}(a_i(x)+\zeta_i(x))\,\nu_0^{(i)}(x)\,,\qquad\forall x\in K_{\eta_0}\setminus K_{\mu}\,,
  \]
  where $\|a_i+\zeta_i\|_{C^0(K_{\eta_0})}\le C\,\eta_0$ by \eqref{ai bi estesi e stime} and \eqref{zeta C0 e C1alpha}. Up to decrease $\eta_0$ in order to obtain $\|a_i+\zeta_i\|_{C^0(K_{\eta_0})}\le \e_0$, we can apply \eqref{intorno normale curvette}, \eqref{rhorho} and $\|a_i\|_{C^0(\R^n)}\le C\|f_0-\Id\|_{C^1(\bd(S_0))}$ to find
  \begin{equation}
    \label{zetaman}
      \|\zeta\|_{C^0(K_{\eta_0}^+\setminus K_{\mu})}
      \le C\,\big(\hd(S,S_0)+\|f_0-\Id\|_{C^1(\bd(S_0))}\big)\,.
  \end{equation}
  In order to estimate $\|\zeta\|_{C^0(K_\mu^+)}$ we consider, for every $x\in K_{\eta_0}$, a point $g(x)\in S_0$ such that $|f(x)-g(x)|=\dist(f(x),S_0)$: we claim that then one must have
  \begin{equation}
    \label{nervous2}
    |g(x)-x|\le \hd(S,S_0)+C\,\mu\,,\qquad\forall x\in K_\mu^+\,.
  \end{equation}
  Indeed, let $x\in K_\mu^+$ so that there exists $y\in\bd(S_0)$ with $|x-y|\le\mu$: since $f(x)\in S$ implies $|f(x)-g(x)|=\dist(f(x),S_0)\le\hd(S,S_0)$, by \eqref{curvette f C11} we find
  \begin{eqnarray*}
    |g(x)-x|\le|g(x)-f(x)|+|f(x)-f(y)|+|x-y|\le \hd(S,S_0)+C\,|x-y|
  \end{eqnarray*}
  that is \eqref{nervous2}. By \eqref{nervous2}, provided $\mu_0$ is small enough with respect to the constant $1/C$ appearing in \eqref{basta stima quadratica normali}, we find that
  \begin{equation}
    \label{caaaaaaaaaaaaz}
      \max_{1\le i\le n-k}|(g(x)-x)\cdot\nu_0^{(i)}(x)|\le C |\pi_x^{S_0}(g(x)-x)|^2\,,\qquad\forall x\in K_\mu^+\,.
  \end{equation}
  Now, by \eqref{rhorho} and \eqref{your mom1} we find that, if $x\in K_{\mu}^+$, then
  \begin{eqnarray*}
    \hd(S,S_0)&\ge&\dist(f(x),S_0)=|f(x)-g(x)|\ge |\pi_x^{S_0}(f(x)-g(x))|
    \\
    &=&|\pi_x^{S_0}(x-g(x))|-|b(x)|\,\phi_\mu(x)
  \end{eqnarray*}
  so that \eqref{caaaaaaaaaaaaz} and \eqref{ai bi estesi e stime} give
  \[
  \max_{1\le i\le n-k}|(g(x)-x)\cdot\nu_0^{(i)}(x)| \le C\,\big(\hd(S,S_0)+\|f_0-\Id\|_{C^1(\bd(S_0))}\big)^2\,,\qquad\forall x\in K_\mu^+\,.
  \]
  By exploiting this last inequality, \eqref{formula fgamma} and \eqref{ai bi estesi e stime} we deduce that if $x\in K_\mu^+$, then
  \begin{eqnarray*}
    \hd(S,S_0)&\ge&\dist(f(x),S_0)=|f(x)-g(x)|\ge|(f(x)-g(x))\cdot\nu_0^{(i)}(x)|
    \\
    &\ge&|(x-g(x))\cdot\nu_0^{(i)}(x)+(a_i(x)+\zeta_i(x))|-|b(x)|\,\phi_\mu(x)
    \\
    &\ge&|\zeta_i(x)|-C\,\big(\hd(S,S_0)+\|f_0-\Id\|_{C^1(\bd(S_0))}\big)\,.
  \end{eqnarray*}
  By combining this estimate with \eqref{zetaman} we thus conclude that
  \begin{equation}
    \label{zetaman2}
      \|\zeta\|_{C^0(K_{\eta_0}^+)}\le C\,\big(\hd(S,S_0)+\|f_0-\Id\|_{C^1(\bd(S_0))}\big)\,.
  \end{equation}
  By combining \eqref{zetaman2}, \eqref{formula fgamma} and \eqref{stime G} we prove \eqref{curvette f va a zero in C0}.
  %\begin{equation}
%    \label{curvette f va a zero in C0}
%    \|f_{p_0}-\Id\|_{C^0(U_{p_0,\eta_0})}\le C\,\rho\,.
%  \end{equation}

  We now prove \eqref{f uguale a psi}. First, we claim that there exists a constant $M$ depending on $n$, $\a$, $k$ and $L$ only such that
  \begin{equation}
    \label{adamallert}
    f(x)\in[S]_{3\rho}\,,\qquad \forall x\in K_{\eta_0}^+\setminus K_{M\rho}\,.
  \end{equation}
  Indeed, let $x\in K_{\eta_0}^+\setminus K_{M\rho}$ and let $y\in\bd(S_0)$ be such that $|f(x)-f(y)|=\dist(f(x),\bd(S))$ (we can find such a point $y$ as $f_0$ is a bijection between $\bd(S_0)$ and $\bd(S)$ and since $f=f_0$ on $\bd(S_0)$). By \eqref{curvette f va a zero in C0}, we have
  \begin{eqnarray*}
  \dist(f(x),\bd(S))&=&|f(x)-f(y)|\ge|x-y|-|f(x)-x|-|f(y)-y|
  \\
  &\ge&\dist(x,\bd(S_0))-C\,\rho
  \ge(M-C)\,\rho\ge 3\rho\,,
  \end{eqnarray*}
  provided $M$ is large enough. This proves \eqref{adamallert}, which, combined with assumption (ii) and \eqref{intorno normale curvette}, gives in particular
  \begin{equation}
    \label{lanalisiarmonica}
      f(x)=g(x)+\psi(g(x))\qquad g(x)\in[S_0]_\rho\,,\qquad \forall x\in K_{\eta_0}^+\setminus K_{M\,\rho}\,.
  \end{equation}
  By \eqref{curvette attacco normale} and \eqref{lanalisiarmonica}, we find $g(x)=x$ for every $x\in K_{\eta_0}^+\setminus K_\mu$, so that, in particular,
  \begin{equation}
    \label{bella}
    f(x)=x+\psi(x)\,,\qquad\forall x\in K_{\eta_0}^+\setminus K_\mu\,,
  \end{equation}
  that is \eqref{f uguale a psi}. Note that this argument also gives $\psi_i=a_i+\zeta_i$ on $K_{\eta_0}^+\setminus K_\mu$, so that \eqref{ai bi estesi e stime} gives us
  \[
  \|\zeta\|_{C^1(K_{\eta_0}^+\setminus K_\mu)}\le C\,\big(\|f_0-\Id\|_{C^1(\bd(S_0))}+\|\psi\|_{C^1([S_0]_\mu)}\big)\,,
  \]
  and thus, by \eqref{formula fgamma} and \eqref{stime G}
  \begin{equation}
    \label{utilina}
      \|f-\Id\|_{C^1(K_{\eta_0}^+\setminus K_\mu)}\le \frac{C}\mu\,\rho\le C\,\mu_0\,,
  \end{equation}
  which will be useful in proving \eqref{curvette f va a zero in C1}, as we are now going to do. We first note that by \eqref{curvette grad f}, \eqref{stime G}, \eqref{curvette f va a zero in C0}, and \eqref{basta nu0 lipschitz}, it is enough to show that
  \begin{equation}
    \label{enoughhh}
    \|\nabla^{\widetilde S_0}\zeta\|_{C^0(K_{\eta_0}^+)}\le \frac{C}\mu\,\rho^\a\,.
  \end{equation}
  To this end, the natural starting point is differentiating $d_S(f)=0$ on $K_{\eta_0}$ at some fixed $x\in K_{\eta_0}$ along $\tau\in T_x\widetilde{S}_0$. By combining the resulting identity $\nabla d_S(f(x))[\nabla^{\widetilde{S}_0}f(x)[\tau]]=0$ with \eqref{curvette grad f}, \eqref{stime G} and \eqref{zetaman2} one finds that, if $x\in K_{\eta_0}^+$ and $\tau\in T_x S_0$ with $|\tau|=1$, then
  \[
  \Big|\nabla d_S(f(x))\Big[\tau+\sum_{i=1}^{n-k}\Big(\nabla^{\widetilde{S}_0}\zeta_{i}(x)[\tau]\Big)\,\nu_0^{(i)}(x)\Big]\Big|\le \frac{C}{\mu}\,\big(\hd(S,S_0)+\|f_0-\Id\|_{C^1(\bd(S_0))}\big)\le \frac{C}\mu\,\rho\,,
  \]
  that is
  \begin{equation}
    \label{lanalisiarmonica4}
      \Big|\nabla d_S(f(x))\Big[\sum_{i=1}^{n-k}\Big(\nabla^{\widetilde{S}_0}\zeta_{i}(x)[\tau]\Big)\,\nu_0^{(i)}(x)\Big]\Big|\le \frac{C}\mu\,\big(\rho+\big|\nabla d_S(f(x))[\tau]\big|\big)\,.
  \end{equation}
  We claim that
  \begin{equation}\label{quasi normale}
    |\nabla d_S(f(x))[v]|\ge\frac{|v|}2\,,\qquad\forall x\in K_{\eta_0}^+\,,v\in (T_xS_0)^\perp\,.
  \end{equation}
  Indeed, if $x\in\bd(S_0)$, then, by \eqref{basta ii k maggiore 1}, $\nu_0^{(i)}(x)\cdot\nu_S^{(i)}(f(x))\ge 1-C\,\rho$ for every $i=1,...,n-k$, that is, $|\pi^S_{f(x)}[v]|\le C\,\rho\,|v|$: thus by \eqref{dg 2} and provided $\mu_0$ is small enough
  \begin{equation}
    \label{exactly the same}
      |\nabla d_S(f(x))[v]|\ge\frac{2}3\,|v|\,,\qquad\forall x\in \bd(S_0)\,,v\in (T_xS_0)^\perp\,,
  \end{equation}
  which immediately gives us \eqref{quasi normale} for $x\in K_\mu^+$ provided $\mu_0$ is small enough depending on $C\ge\|d_S\|_{C^{1,\a}(\R^n)}$. If instead $x\in K_{\eta_0}^+\setminus K_\mu$, then by \eqref{utilina} we find that $|\pi_{f(x)}^S[v]|=|\pi_{f(x)}^S[v]-\pi_x^{S_0}[v]|\le C\,\mu_0\,|v|$. Thus we deduce that \eqref{quasi normale} holds for $x\in K_{\eta_0}^+\setminus K_\mu$ too, once again, thanks to \eqref{dg 2} and provided $\mu_0$ is small enough. By combining \eqref{quasi normale} with \eqref{lanalisiarmonica4} we thus find
  \begin{equation}
    \label{lanalisiarmonica4000}
      \big|\nabla^{\widetilde{S}_0}\zeta(x)[\tau]\big|\le \frac{C}\mu\,\big(\rho+|\nabla d_S(f(x))[\tau]|\big)\,,\qquad\forall x\in K_{\eta_0}^+\,,\tau\in T_xS_0\cap\SS^{n-1}\,.
  \end{equation}
  We are now going to show that
  \begin{equation}
    \label{alla fine dio bono}
    |\nabla d_S(f(x))[\tau]|\le C\,\rho^\a\,,\qquad\forall x\in K_{\eta_0}^+, \tau\in T_xS_0\cap\SS^{n-1}\,.
  \end{equation}
   Indeed, if $x\in\bd(S_0)$, then \eqref{alla fine dio bono} follows by exactly the same argument used to prove \eqref{exactly the same} (with $\rho$ in place of $\rho^\a$). By exploiting $\|\nabla d_S\|_{C^{1,\a}(\R^n)}\le C$, one deduces the validity of \eqref{alla fine dio bono} for every $x\in K_{M\rho}^+$ (here is the point where $\rho^\a$ appears in place of $\rho$). In order to prove \eqref{alla fine dio bono} on $K_{\eta_0}^+\setminus K_{M\,\rho}$ we first notice that if $x\in K_{\eta_0}^+$, then $|g(x)-f(x)|=\dist(f(x),S_0)\le|f(x)-x|$, so that \eqref{curvette f va a zero in C0} implies the following improvement of \eqref{nervous2}:
  \begin{equation}
    \label{nervous a mille}
    |g(x)-x|\le C\,\big(\hd(S,S_0)+\|f_0-\Id\|_{C^1(\bd(S_0))}\big)\le C\,\rho\,,\qquad\forall x\in K_{\eta_0}^+\,.
  \end{equation}
  At the same time, by $(\Id+\psi)([S_0]_\rho)\subset S$ and $\|\psi\|_{C^1([S_0]_\rho)}\le\rho$ one finds
  \[
  |\pi^S_{x+\psi(x)}[\tau]|\ge (1-C\,\rho)\,|\tau|,\qquad\forall x\in[S_0]_\rho\,,\tau\in T_xS_0\,,
  \]
  which, by \eqref{dg 2}, gives
  \[
  |\nabla d_S(x+\psi(x))[\tau]|\le C\,\rho\,|\tau|\,,\qquad\forall x\in[S_0]_\rho\,,\tau\in T_xS_0\,.
  \]
  By \eqref{basta dS}, \eqref{basta iii} and \eqref{nervous a mille}
  \[
  |\nabla d_S(g(x)+\psi(g(x)))[\tau]|\le C\,\rho^\a\,|\tau|\,,\qquad\forall x\in  K_{\eta_0}^+\setminus K_\rho\,,\tau\in T_xS_0\,,
  \]
  which implies \eqref{alla fine dio bono} for $x\in  K_{\eta_0}^+\setminus K_{M\rho}$ thanks to \eqref{lanalisiarmonica}. This completes the proof of \eqref{alla fine dio bono}, which combined with \eqref{lanalisiarmonica4000} gives us \eqref{enoughhh}. The claim, thus theorem, is then proved in the case $k\ge 2$. Concerning the case $k=1$, the main difference is that the extensions $a_i$ and $b$ of $\bar{a_i}$ and $\bar{b}$ satisfying \eqref{ai bi estesi e stime} can now be defined by elementary means by exploiting the assumption $|p_0-q_0|\ge 1/L$, with their $C^1(\R^n)$-norms controlled  in terms of $\|f_0-\Id\|_{C^0(\bd(S_0))}$ (see also Remark \ref{remark bello C0}). The rest of the proof carries on almost {\it verbatim}, and we thus omit the details.
\end{proof}

\subsection{A reformulation of Theorem \ref{thm main diffeo}}\label{subsection application of theorem 3.1} In the situations in which we plan to apply Theorem \ref{thm main diffeo} we are usually given a sequence of manifolds $\{S_j\}_j$ converging to a limit manifold $S_0$ rather than a pair of nearby manifolds $S$ and $S_0$. In order to apply Theorem \ref{thm main diffeo} one thus needs to pass from the former situation to the latter, and this can indeed be done by a simple argument. Instead of having to repeat this argument at each application of Theorem \ref{thm main diffeo}, it seems preferable to prove once and for all an alternative version Theorem \ref{thm main diffeo} which is already tailored for the case of sequences.

\begin{theorem}\label{thm main diffeo k}
Let $n\ge 2$, $1\le k\le n-1$, $\a\in(0,1]$, and $L>0$. Let $S_0$ be a compact connected $k$-dimensional $C^{2,1}$-manifold with boundary in $\R^n$ and let $\{\nu_{S_0}^{(i)}\}_{i=1}^{n-k}\subset C^{1,1}(S_0;\SS^{n-1})$ be such that $\{\nu_{S_0}^{(i)}\}_{i=1}^{n-k}$ is an orthonormal basis of $(T_xS_0)^\perp$ for every $x\in S_0$. Then there exist $\mu_0\in(0,1)$ and $C_0>0$ (depending on $n$, $k$, $\a$, $L$ and $S_0$ only) with the following property.

Let $\{S_j\}_{j\in\N}$ be a sequence of a compact connected $k$-dimensional $C^{1,\a}$-manifold with boundary in $\R^n$ such that
  \begin{equation}
    \label{basta S C1alphaL and hd rho j}
      \bd(S_j)\ne\emptyset\,,\qquad \|S_j\|_{C^{1,\a}}\le L\,,\qquad \lim_{j\to\infty}\hd(S_j,S_0)=0\,,
  \end{equation}
and assume in addition that:

  \medskip

  \noindent (i) if $k=1$, then, setting $\bd(S_0)=\{p_0,q_0\}$, $\bd(S_j)=\{p_j,q_j\}$, $f_{0,j}(p_0)=p_j$ and $f_{0,j}(q_0)=q_j$,
  \begin{equation}     \label{basta ii k uguale 1 j}
  \lim_{j\to\infty}\|f_{0,j}-\Id\|_{C^0(\bd(S_0))}+\|\nu^{co}_{S_j}(f_{0,j})-\nu^{co}_{S_0}\|_{C^0(\bd(S_0))}=0\,;
  \end{equation}
  if $k\ge 2$, then there exist $C^{1,\a}$-diffeomorphisms $f_{0,j}$ between $\bd(S_0)$ and $\bd(S_j)$ with
  \begin{equation}
    \label{basta ii k maggiore 1 j}
    \begin{split}
      &\sup_{j\in\N}\|f_{0,j}\|_{C^{1,\a}(\bd(S_0))}\le L\,,
      \\
      &\lim_{j\to\infty}\|f_{0,j}-\Id\|_{C^1(\bd(S_0))}=0\,,
      \\
      &\lim_{j\to\infty}\max_{1\le i\le n-k}\|\nu_{S_j}^{(i)}(f_{0,j})-\nu_0^{(i)}\|_{C^0(\bd(S_0))}=0\,,
      \\
      &\lim_{j\to\infty}\|\nu_{S_j}^{co}(f_{0,j})-\nu_{S_0}^{co}\|_{C^0(\bd(S_0))}=0\,,
    \end{split}
  \end{equation}
  where $\{\nu_{S_j}^{(i)}\}_{i=1}^{n-k}$ is satisfies \eqref{basta S C1alpha} with $S=S_j$;

  \medskip

  \noindent (ii) for every $\rho<\mu_0^2$ and $i=1,...,n-k$ there exist $j(\rho)\in\N$ and $\{\psi_{i,j}\}_{j\ge j(\rho)}\subset C^{1,\a}([S_0]_\rho)$ such that, setting $\psi_j=\sum_{i=1}^{n-k}\psi_{i,j}\,\nu_{S_0}^{(i)}$, one has
  \begin{equation}\label{basta iii j}
    \begin{split}
      &\hspace{0.9cm}[S_j]_{3\rho}\subset(\Id+\psi_j)([S_0]_\rho)\subset S\,,\qquad\forall j\ge j(\rho)\,,
      \\
      &\sup_{j\ge j(\rho)}\|\psi_j\|_{C^{1,\a}([S_0]_\rho)}\le L\,,\qquad \lim_{j\to\infty}\|\psi_j\|_{C^1([S_0]_\rho)}=0\,.
    \end{split}
  \end{equation}
  Then, for every $\mu\in(0,\mu_0)$ there exist $j(\mu)\in\N$ and, for each $j\ge j(\mu)$, a $C^{1,\a}$-diffeomorphisms $f_j$ between $S_0$ and $S_j$ such that
  \begin{equation}
    \label{jjj}
      \begin{split}
    &\hspace{2.5cm}\mbox{$f_j=f_{0,j}$ on $\bd(S_0)$}\,,\qquad
    \mbox{$f_j=\Id+\psi_j$ on $[S_0]_\mu$}\,,
    \\
    &\hspace{2cm}\sup_{j\ge j(\mu)}\|f_j\|_{C^{1,\a}(S_0)}\le C_0\,,\qquad \lim_{j\to\infty}\|f_j-\Id\|_{C^1(S_0)}=0\,,
    \\
    &\|\pi^{S_0}(f_j-\Id)\|_{C^1(S_0)}\le\frac{C_0}\mu\,\left\{
    \begin{split}
    \|(f_{0,j}-\Id)\cdot\nu_{S_0}^{co}\|_{C^0(\bd(S_0))}\,,\qquad\mbox{if $k=1$}\,,
    \\
    \|f_{0,j}-\Id\|_{C^1(\bd(S_0))}\,,\qquad\mbox{if $k\ge 2$}\,.
    \end{split}\right .
  \end{split}
  \end{equation}
\end{theorem}

\begin{proof}
   By Remark \ref{remark assumption a}, up to increasing the value of $L$ depending on $S_0$, one can entail the existence of $\widetilde{S}_0$ such that assumption (a) in Theorem \ref{thm main diffeo} holds, and also that $|p_0-q_0|\ge 1/L$ in the case $k=1$. Now let $\mu_0$ and $C_0$ be determined as in Theorem \ref{thm main diffeo} by $n$, $k$, $\a$ and the increased $S_0$-depending value of $L$, and let us fix $\mu\in(0,\mu_0)$. Given $\rho\in(0,\mu^2)$, by \eqref{basta ii k uguale 1 j}, \eqref{basta ii k maggiore 1 j}, and \eqref{basta iii j}, and up to increasing the value of $j(\rho)$, then for each $j\ge j(\rho)$, $S_j$, $f_{0,j}$ and $\psi_j$ satisfy assumption (b) of Theorem \ref{thm main diffeo}, that is, referring from now on to the case $k\ge 2$, for every $j\ge j(\rho)$ one has $[S_j]_{3\rho}\subset(\Id+\psi_j)([S_0]_\rho)\subset S_j$ with
   \begin{eqnarray*}
   &&\max\big\{\hd(S_0,S_j),\|f_{0,j}-\Id\|_{C^1(\bd(S_0))},\|\nu_S^{(i)}(f_{0,j})-\nu_0^{(i)}\|_{C^1(\bd(S_0))},
   \\
   &&\hspace{6cm}\|\nu_S^{co}(f_{0,j})-\nu_0^{co}\|_{C^1(\bd(S_0))},\|\psi_j\|_{C^1([S_0]_\rho)}\big\}\le \rho\,,
   \\
   &&\max\big\{\|S_j\|_{C^{1,\a}},\|f_{0,j}\|_{C^{1,\a}(\bd(S_0))},\|\psi_j\|_{C^{1,\a}([S_0]_\rho)}\big\}\le L\,.
   \end{eqnarray*}
   Hence, by Theorem \ref{thm main diffeo}, for each $j\ge j(\rho)$ we can construct $C^{1,\a}$-diffeomorphisms $f_j^{\rho}$ between $S_0$ and $S_j$ such that
   \[\begin{split}
   &\hspace{2.5cm}\mbox{$f_j^\rho=f_{0,j}$ on $\bd(S_0)$}\,,\qquad\mbox{$f_j^\rho=\Id+\psi$ on $[S_0]_\mu$}\,,
   \\
   &\hspace{2cm}\|f_j^\rho\|_{C^{1,\a}(S_0)}\le C_0\,,\qquad \|f_j^\rho-\Id\|_{C^1(S_0)}\le \frac{C_0}\mu\,\rho^\a\,,
   \\
   &\|\pi^{S_0}(f_j^\rho-\Id)\|_{C^1(S_0)}\le\frac{C_0}\mu\,\left\{
    \begin{split}
    \|(f_{0,j}-\Id)\cdot\nu_{S_0}^{co}\|_{C^0(\bd(S_0))}\,,\qquad\mbox{if $k=1$}\,,
    \\
    \|f_{0,j}-\Id\|_{C^1(\bd(S_0))}\,,\qquad\mbox{if $k\ge 2$}\,.
    \end{split}\right .
   \end{split}
   \]
   Finally, let us set, for $\ell\ge 2$, $\rho_\ell=\mu^{2/\a}/(2+\ell)$. For each $\ell\ge 2$, $\rho_\ell\in(0,\mu^2)$. By iteratively applying the construction above we can find a strictly increasing sequence $\{j_{\ell}\}_{\ell\ge 2}\subset\N$ such that if $j_{\ell}\le j<j_{\ell+1}$, then $f_j=f_j^{\rho_\ell}$ defines a $C^{1,\a}$-diffeomorphism between $S_0$ and $S_j$ such that \eqref{jjj} holds with
   \[
   \|f_j-\Id\|_{C^1(S_0)}\le \frac{C_0}\mu\,\rho_\ell^\a=\frac{C_0\mu}{(2+\ell)^\a}\,.
   \]
   This completes the proof of the theorem.
\end{proof}

\section{Perimeter almost-minimizing clusters in $\R^n$}\label{section what is true in Rn} The goal of this section is preparing the ground for the application of Theorem \ref{thm main diffeo} to the proof of Theorem \ref{thm main planar}. Specifically, in this section we discuss those preliminary facts that we can prove in arbitrary dimension $n$. (In particular, these results shall also be used in part two \cite{CiLeMaIC2}.) For the most part the arguments of this section should be familiar to some readers, but we have nevertheless included some details of most of the proofs for the sake of clarity. In section \ref{section sofp} we gather some relevant definitions from Geometric Measure Theory. In section \ref{section reg criterion basic} we recall the classic regularity criterion for almost-minimizing sets (Theorem \ref{thm small excess criterion}) and derive from it a very useful technical statement (Lemma \ref{lemma fondamentale} -- which is well-known to experts, although, apparently, not explicitly stated in the literature). In section \ref{section infiltratio lemma} we exploit a simple ``infiltration lemma'' to construct normal diffeomorphisms away from the singular sets (Theorem \ref{thm normal representation part one}) and to prove Hausdorff convergence of the boundaries (Theorem \ref{thm boundary hausdorff}). Finally, in section \ref{section blowups clusters} we briefly discuss blow-up limits of clusters.

\subsection{Basic definitions and terminology}\label{section sofp} Here we gather various definitions from Geometric Measure Theory needed in the sequel.

\medskip

\noindent {\bf Rectifiable sets.} Let $\H^k$ denote the $k$-dimensional Hausdorff measure on $\R^n$. A set $S\subset\R^n$ is locally $k$-rectifiable in $A\subset\R^n$ open, if $\H^k\llcorner S$ is a Radon measure on $A$ and $S$ is contained, modulo an $\H^k$-null set, into a countable union of $k$-dimensional $C^1$-surfaces. If $S$ is locally $\H^k$-rectifiable in $A$ then for $\H^k$-a.e. $x\in S\cap A$ there exists a $k$-plane $T_xS$ in $\R^n$, the {\it approximate tangent space to $S$ at $x$}, with $\H^k\llcorner (S-x)/r\weak \H^k\llcorner T_xS$ when $r\to 0^+$ as Radon measures; see \cite[Theorem 10.2]{maggiBOOK}. Given such $x\in S$, $T\in C^1_c(\R^n;\R^n)$, and $\{\tau_i(x)\}_{i=1}^k$ an orthonormal basis of $T_xS$, the {\it tangential divergence} $\Div_ST$ of $T$ over $S$ at $x$ is defined by $\Div_S\,T(x)=\sum_{i=1}^k\,\tau_i(x)\cdot(\nabla T(x)\tau_i(x))$. One says that $S$ has {\it generalized mean curvature} $\HH_S\in L^1_{{\rm loc}}(\H^k\llcorner(A\cap S);\R^n)$ in  $A$, if
\begin{equation}
  \label{varifold curvature}
  \int_S\,\Div_S\,T\,d\H^k=\int_S\,T\cdot\,\HH_S\,d\H^k\,,\qquad\forall T\in C^1_c(A;\R^n)\,.
\end{equation}
If $\HH_S\in L^\infty(\H^k\llcorner(A\cap S);\R^n)$ one says that $S$ has bounded generalized mean curvature.

\medskip

\noindent {\bf Sets of finite perimeter.} A Lebesgue-measurable set $E\subset\R^n$ is a {\it set of locally finite perimeter} in an open set $A\subset\R^n$ if $\sup\{\int_E\Div\,T:T\in C^1_c(A;B)\}<\infty$, or, equivalently, if there exists a $\R^n$-valued Radon measure $\mu$ on $A$ with
\begin{equation}
  \label{distributonal gg}
  \int_E\nabla\vphi(x)\,dx=\int_{\R^n}\vphi(x)\,d\mu(x)\,,\qquad\forall \vphi\in C^1_c(A)\,.
\end{equation}
The Gauss--Green measure $\mu_E$ of $E$ is defined as the Radon measure appearing in \eqref{distributonal gg} for the largest open set $A$ such that $E$ is of locally finite perimeter in $A$. The {\it reduced boundary} $\pa^*E$ of $E$ is defined as the set of those $x\in \spt\,\mu_E\subset A$ such that
\begin{equation}
  \label{reduced boundary}
  \nu_E(x)=\lim_{r\to 0^+}\frac{\mu_E(B_{x,r})}{|\mu_E|(B_{x,r})}\qquad\mbox{exists and belongs to $\SS^{n-1}$}\,.
\end{equation}
It turns out that $\pa^*E$ is a locally $\H^{n-1}$-rectifiable set in $A$, and the Borel vector field $\nu_E:\pa^*E\to \SS^{n-1}$ (called the {\it measure-theoretic outer unit normal to $E$}) is such that $\mu_E=\nu_E\,\H^{n-1}\llcorner\pa^*E$ on bounded Borel subsets of $A$.
%In particular, \eqref{distributonal gg} takes the more explicit form $\int_E\nabla\vphi(x)\,dx=\int_{\pa^*E}\vphi(x)\,\nu_E(x)\,d\H^{n-1}(x)$ for every $\vphi\in C^1_c(A)$.
If $F\subset A$ is a Borel set, then the {\it perimeter of $E$ relative to the Borel set $F$} is defined as $P(E;F)=|\mu_E|(F)=\H^{n-1}(F\cap\pa^*E)$, and we set $P(E)=P(E;\R^n)$. One always has
\[
A\cap\cl(\pa^*E)=\spt\,\mu_E=\big\{x\in A:0<|E\cap B_{x,r}|<\om_n\,r^n\quad\forall r>0\big\}\subset A\cap\pa E\,,
\]
where $\om_n$ is the volume of the Euclidean unit ball in $\R^n$; moreover, $\mu_E$ is invariant by modifications of $E\cap A$ on and by a set of volume zero, and up to such modifications (see, for example, \cite[Proposition 12.19]{maggiBOOK}) we can assume that
\begin{equation}
  \label{spt normalization}
A\cap\cl(\pa^*E)=\spt\,\mu_E=A\cap\pa E\,.
\end{equation}
Throughout this paper, {\bf all sets of finite perimeter shall be normalized so to have identity \eqref{spt normalization} in force} (where $A$ denotes the largest open set such that $E$ is of locally finite perimeter in $A$).

Let us now recall from the introduction that a family $\E=\{\E(h)\}_{h=1}^N$ of Lebesgue-measurable sets in $\R^n$ with
%$|\E(h)|<\infty$ for $h=1,...,N$ and
$|\E(h)\cap\E(k)|=0$ for $1\le h<k\le N$ is an {\it $N$-cluster in $A$} if each $\E(h)$ is a set of locally finite perimeter in $A$ and $|\E(h)\cap A|>0$ for every $h=1,...,N$. If $A$ is the largest open set such that $\E$ is a cluster in $A$, then, according to \eqref{reduced boundary}, $\pa^*\E(h)$ is well-defined as a subset of $A$ and so are the {\it interfaces} $\E(h,k)=\pa^*\E(h)\cap\pa^*\E(k)$; thus $\pa^*\E$, as defined in \eqref{def bordi vari}, is automatically a subset of $A$, with
\[
\pa^*\E=\bigcup_{0\le h<k\le N}\E(h,k)\,.
\]
It will be useful to keep in mind that, by \eqref{spt normalization}, one has
\[
\cl(\pa^*\E)=A\cap\bigcup_{h=1}^N\spt\mu_{\E(h)}=\bigcup_{h=1}^N\Big\{x\in A:0<|\E(h)\cap B_{x,r}|<\om_n\,r^n\,\,\forall r>0\Big\}=A\cap\pa\E\,,
\]
%where $\pa\E=\bigcup_{h=1}^N\pa\E(h)$. We also set
%\[
%\S_F(\E)=(F\cap\pa\E)\setminus \pa^*\E\qquad \forall F\subset A\,,\qquad\S(\E)=\S_{\R^n}(\E)\,.
%\]
%Finally, recall \eqref{almost-minimizer cluster}, one says that $\E$ is a {\it $(\Lambda,r_0)$-minimizing cluster in $A$} if
%\[
%P(\E;B_{x,r})\le P(\F;B_{x,r})+\Lambda\,\d(\E,\F)\,,
%\]
%whenever $x\in\R^n$, $r<r_0$ and $\E(h)\Delta\F(h)\cc B_{x,r}$ for every $h=1,...,N$, and where
%\begin{equation}
%  \label{distance clusters}
%  \d_F(\E,\F)=\frac12\sum_{h=0}^N\,\Big|F\cap\big(\E(h)\Delta\F(h)\big)\Big|\,,\qquad\d(\E,\F)=\d_{\R^n}(\E,\F)\,.
%\end{equation}

\subsection{A regularity criterion for $(\Lambda,r_0)$-minimizing sets}\label{section reg criterion basic} Given $x\in\R^n$, $r>0$ and $\nu\in\SS^{n-1}$, let us set
\begin{eqnarray*}
&&\C^\nu_{x,r}=\big\{y\in\R^n:|(y-x)\cdot\nu|<r\,,|(y-x)-((y-x)\cdot\nu)\nu|<r\big\}\,,
  \\
  &&\D^\nu_{x,r}=\big\{y\in\R^n:|(y-x)\cdot\nu|=0\,,|(y-x)-((y-x)\cdot\nu)\nu|<r\big\}\,,
\end{eqnarray*}
and define the {\it cylindrical excess} of $E\subset\R^n$ at $x$, in direction $\nu$, and at scale $r$, as
\[
\exc^\nu_{x,r}(E)=\frac1{r^{n-1}}\int_{\C^\nu_{x,r}\cap\pa^*E}|\nu_E-\nu|^2\,d\H^{n-1}\,,
\]
provided $E$ is of finite perimeter on $\C^\nu_{x,r}$. When $\nu=e_n$ and $x=0$ we simply set
\[
\C_r=\C^{e_n}_{0,r}\,,\qquad \D_r=\D^{e_n}_{0,r}\,,\qquad \exc_r(E)=\exc^{e_n}_{0,r}(E)\,.
\]
The next result is a classical local regularity criterion for $(\Lambda,r_0)$-minimizing sets.

%$E$ is a $(\Lambda,r_0)$-minimizing set in $A$, then $A\cap \pa^*E$ is a $C^{1,\b}$-hypersurface for every $\b\in(0,1/2)$, see \cite[Theorem 26.5]{maggiBOOK}, and that the singular set of $E$ in $A$
%\[
%\S(E;A)=\Big(\pa E\setminus\pa^*E\Big)\cap A\,,
%\]
%is empty if $2\le n\le 7$, it is locally finite if $n=8$, and satisfies $\H^s(\S(E;A))=0$ for every $s>n-8$ if $n\ge 9$, see \cite[Theorem 28.1]{maggiBOOK}. A crucial step in the proof of these assertions is the ``small excess regularity criterion'', see Theorem \ref{thm small excess criterion}, which shall also play an important role in our analysis of clusters.

\begin{theorem}[Small excess regularity criterion]\label{thm small excess criterion}
  For every $n\ge 2$ and $\a\in(0,1)$ there exist positive constants $\e_*(n)$, $C(n)$ and $C(n,\a)$ with the following property. If $E$ is a $(\Lambda,r_0)$-minimizing set in $\C^\nu_{x_0,r}$ with $x_0\in\pa E$, $r<r_0$, and
  \begin{equation}
    \label{small exc 0}
      \exc^\nu_{x_0,r}(E)+\Lambda\,r\le\e_*(n)\,,
  \end{equation}
  then there exists a Lipschitz function $v:\D^\nu_{x_0,r/2}\to\R$ with $v(x_0)=0$,
  \begin{eqnarray}\label{small exc 1}
  \|v\|_{C^0(\D^\nu_{x_0,r/2})}&\le& C(n)\,r\,\exc^\nu_{x_0,r}(E)^{1/2(n-1)}\,,
  \\
  \label{small exc 1.5}
  \|\nabla v\|_{C^0(\D^\nu_{x_0,r/2})}&\le& C(n)\,\big(\exc^\nu_{x_0,r}(E)+\Lambda\,r\big)^{1/2(n-1)}\,,
  \\\label{small exc 2}
  r^\a\,[\nabla v]_{C^{0,\a}(\D^\nu_{x_0,r/2})}&\le& C(n,\a)\,\big(\exc^\nu_{x_0,r}(E)+\Lambda\,r\big)^{1/2(n-1)}\,,\qquad\forall\a\in(0,1)\,,
  \end{eqnarray}
and such that
\begin{eqnarray}
%  \C_\nu(x_0,r/2)\cap E&=&\Big\{x+t\,\nu:x\in\D^\nu_{x_0,r/2}\,,v(x)\le t\le r\Big\}\,,
%  \\
\label{small exc 3}
\C^\nu_{x_0,r/2}\cap \pa E=(\Id+v\,\nu)(\D^\nu_{x_0,r/2})
%=\Big\{z+v(z)\,\nu:z\in\D^\nu_{x_0,r/2}\Big\}
\,.
\end{eqnarray}
Moreover, if $n=2$ then one can replace \eqref{small exc 2} with $\|v''\|_{L^\infty(\D^\nu_{x_0,r/2})}\le C\,\Lambda$.
\end{theorem}

\begin{proof} Without loss of generality we set $x_0=0$ and $\nu=e_n$. By \cite[Theorem 26.3]{maggiBOOK} (applied, in the notation of that theorem, with $\g=1/4$) there exist positive constants $\e_*(n)$ and $C(n)$ such that if \eqref{small exc 0} holds then \eqref{small exc 3} holds for a Lipschitz function $v:\D_{2r/3}\to\R$ with $v(0)=0$ and
\begin{eqnarray}\label{small exc x}
\frac{|v(x)|}r+|\nabla v(x)|+r^{1/4}\,\frac{|\nabla v(x)-\nabla v(y)|}{|x-y|^{1/4}}\le
C(n)\,\big(\exc_r(E)+\Lambda\,r\big)^{1/2(n-1)}\,,
\end{eqnarray}
for every $x\ne y\in \D_{2r/3}$. We now prove \eqref{small exc 2}. By \eqref{almost-minimizer set} and \eqref{small exc 3} one finds that
\begin{equation}
  \label{directlyfrom}
\int_{\D_{2r/3}}\sqrt{1+|\nabla v|^2}\le \int_{\D_{2r/3}}\sqrt{1+|\nabla (v+\vphi)|^2}+\Lambda\,\int_{\D_{2r/3}}|\vphi|\,,
\end{equation}
for every $\vphi\in C^1_c(\D_{2r/3})$. In particular, there exists $g\in L^\infty(\D_{2r/3})$ such that
\[
\|g\|_{L^\infty(\D_{2r/3})}\le \Lambda\,,\qquad-\int_{\D_{2r/3}}\frac{\nabla v}{\sqrt{1+|\nabla v|^2}}\cdot\nabla\vphi=\int_{\D_{2r/3}}g\,\vphi\qquad\forall \vphi\in C^1_c(\D_{2r/3})\,.
\]
By taking incremental ratios one sees that $v\in W^{2,2}_{loc}(\D_{2r/3})$ with
\[
\trace(A(x)\nabla^2v(x))=g(x)\,,\qquad\mbox{for a.e. $x\in\D_{2r/3}$}\,,
\]
where $A=(1+|\nabla v|^2)^{-3/2}\,[(1+|\nabla v|^2)\Id-\nabla v\otimes\nabla v]=F(\nabla v)$ for a Lipschitz map $F:\R^n\to\R^n\otimes\R^n$. Thanks to \eqref{small exc x},
\[
\frac{\Id}{C(n)}\le A(x)\le C(n)\,\Id\,,\qquad r^{1/4}\,|A(x)-A(y)|\le C(n)\,|x-y|^{1/4}\,,\qquad\forall x,y\in\D_{2r/3}\,.
\]
If we set $A^*(x)=A(r\,x)$, $v^*(x)=v(r\,x)$ and $g^*(x)=g(r\,x)$ for $x\in \D_{2/3}$, then
\[
\trace(A^*(x)\nabla^2v^*(x))%=r^2\trace(A(r\,x)\nabla^2v(r\,x))
=r^2\,g^*(x)\,,\qquad\mbox{for a.e. $x\in \D_{2/3}$}\,,
\]
with $|A^*(x)-A^*(y)|\le C(n)\,|x-y|^{1/4}$ for every $x,y\in \D_{2/3}$. If $n\ge 3$, then by \cite[Theorem 9.11]{GTbook}, for every $p\in(1,\infty)$ one has
\begin{eqnarray*}
\|v^*\|_{W^{2,p}(\D_{1/2})}&\le& C(n,p)\big(\|v^*\|_{L^p(\D_{2/3})}+\|r^2\,g^*\|_{L^p(\D_{2/3})}\big)
\\
&\le&C(n,p)\,r\,\big(\exc_r(E)+\Lambda\,r\big)^{1/2(n-1)}
\end{eqnarray*}
thanks to \eqref{small exc x}. At the same time, by Morrey's inequality, if we pick $p>n-1$ such that $\a=1-(n-1)/p$ then
\[
C(n,p)\|v^*\|_{W^{2,p}(\D_{1/2})}\ge  \|v^*\|_{C^{1,\a}(\D_{1/2})}\ge  [\nabla v^*]_{C^{0,\a}(\D_{1/2})}=r^{1+\a}\,[\nabla v]_{C^{0,\a}(\D_{r/2})}\,,
\]
which gives us \eqref{small exc 2} in the case $n\ge 3$. If $n=2$, then \eqref{directlyfrom} directly implies that $(1+(v')^2)^{-1/2}v'$ has a bounded distributional derivative $g$ on the interval $\D_{2r/3}$. By the chain rule we immediately find $\|v''\|_{L^\infty(\D_{2r/3})}\le (1+\|v'\|_{C^0(\D_{2r/3})}^2)^{3/2}\Lambda\le C\,\Lambda$.
\end{proof}

\begin{remark}\label{remark reduced boundary}
  {\rm Recall that $\lim_{r\to 0^+}\inf_{\nu\in \SS^{n-1}}\exc^\nu_{x,r}(E)=0$ for every $x\in\pa^*E$; see, for example, \cite[Proposition 22.3]{maggiBOOK}. In particular, if $E$ is a $(\Lambda,r_0)$-minimizing set in $A$, then $A\cap\pa^*E$ is a $C^{1,\a}$-hypersurface for every $\a\in(0,1)$ ($C^{1,1}$ if $n=2$).}
\end{remark}

Theorem \ref{thm small excess criterion} can be used to locally represent the boundaries of $(\Lambda,r_0)$-minimizing sets $E_k$ converging to a set $E$ as graphs with respect to $\pa E$, at least provided $\pa E$ is smooth enough. This basic idea is made precise in Lemma \ref{lemma fondamentale} below. Before stating this lemma,  we prove the following technical statement where, given $u\in C^{k,\a}(\D_r)$ we set
\[
\|u\|_{C^{k,\a}(\D_r)}^*=\sum_{j=0}^k\,r^{j-1}\,\|D^j u\|_{C^0(\D_r)}+r^{k-1+\a}\,[D^ku]_{C^{0,\a}(\D_r)}\,.
\]
In this way, if we set $\l_r(u)(x)=r^{-1}\,u(r\,x)$ for $x\in\D$, then
\[
\|\l_r(u)\|_{C^{k,\a}(\D)}=\|\l_r(u)\|_{C^{k,\a}(\D)}^*=\|u\|_{C^{k,\a}(\D_r)}^*\,,\qquad\forall r>0\,.
\]
Moreover, given $u:\D_{4r}\to\R$ with $|u|< 4r$ on $\D_{4r}$ we set
\[
\Gamma_r(u)=(\Id+u\,e_n)(\D_{4r})\subset\C_{4r}\,,
\]
and let $\a\wedge\b=\min\{\a,\b\}$.

\begin{lemma}\label{lemma facile}
Given $n\ge2$, $L>0$ and $\a,\b\in[0,1]$ there exist positive constants $\s_0<1$ and $C_0$ with the following property. If $u_1\in C^{2,\a}(\D_{4r})$, $u_2\in C^{1,\b}(\D_{4r})$, and
\begin{equation}
    \label{pattymelt 1}
\max_{i=1,2}\|u_i\|_{C^1(\D_{4r})}^*\le \s_0\,,\qquad
\max\big\{\|u_1\|_{C^{2,\a}(\D_{4r})}^*,\|u_2\|_{C^{1,\b}(\D_{4r})}^*\big\}\le L\,,
\end{equation}
then there exists $\psi\in C^{1,\a\wedge\b}(\C_{2r}\cap\Gamma_r(u_1))$ such that
\begin{gather}\label{pattymelt fine 1}
\C_r\cap\Gamma_r(u_2)\subset (\Id+\psi\nu)(\C_{2r}\cap\Gamma_r(u_1))\subset \Gamma_r(u_2)\,,
\\
\label{pattymelt fine 2 C1alphabeta}
\frac{\|\psi\|_{C^0(\C_{2r}\cap\Gamma_r(u_1))}}r+\|\nabla\psi\|_{C^0(\C_{2r}\cap\Gamma_r(u_1))}
+r^{\a\wedge\b}\,[\nabla\psi]_{C^{0,\a\wedge\b}(\C_{2r}\cap\Gamma_r(u_1))}\le C_0\,,
\\
\label{pattymelt fine 2 C1}
\frac{\|\psi\|_{C^0(\C_{2r}\cap\Gamma_r(u_1))}}{r}+\|\nabla\psi\|_{C^0(\C_{2r}\cap\Gamma_r(u_1))}\le C_0\,\|u_1-u_2\|_{C^1(\D_{4r})}\,.
\end{gather}
Here, $\nu\in C^{1,\a}(\Gamma_r(u_1);\SS^{n-1})$ is the normal unit vector field to $\Gamma_r(u_1)$ defined by
\begin{equation}
  \label{pattymelt 3}
  \nu(z,u_1(z))=\frac{(-\nabla u_1(z),1)}{\sqrt{1+|\nabla u_1(z)|^2}}\,,\qquad \forall z\in\D_{4r}\,.
\end{equation}
\end{lemma}

\begin{proof} Up to replacing $u_i$ with $\l_r(u_i)$ we may directly assume that $r=1$. Correspondingly, we write $\Gamma(u_i)$ in place of $\Gamma_1(u_i)$ for the sake of simplicity. We define $F:\D_4\times\R\to\R^n$ and $\phi:\D_4\times\R\to\R$ by setting
\begin{eqnarray}
  F(z,t)&=&\Big(z-t\,\frac{\nabla u_1(z)}{\sqrt{1+|\nabla u_1(z)|^2}},u_1(z)+\frac{t}{\sqrt{1+|\nabla u_1(z)|^2}}\Big)\,,
  \\
  \phi(z,t)&=&u_2(z)-t\,,
\end{eqnarray}
for $(z,t)\in\D_4\times\R$. Notice that $F\in C^{1,\a}(\C_4)$ and $\phi\in C^{1,\b}(\C_4)$ with
\begin{equation}
  \label{pattymelt 2}
  \|F\|_{C^{1,\a}(\C_4)}\le C\,,\qquad\|\phi\|_{C^{1,\b}(\C_4)}\le C\,,
\end{equation}
where $C$ is a constant depending on $n$, $\a$, $\beta$ and $L$ only. Provided $\s_0$ is small enough we also find $F(\C_2)\subset \C_4$, so that we can define $\Phi:\C_2\to\R$ by setting
\[
\Phi(z,t)=\phi(F(z,t))=u_2\Big(z-t\,\frac{\nabla u_1(z)}{\sqrt{1+|\nabla u_1(z)|^2}}\Big)-
u_1(z)-\,\frac{t}{\sqrt{1+|\nabla u_1(z)|^2}}\,.
\]
By exploiting \eqref{pattymelt 1} and \eqref{pattymelt 2} we find that, provided $\s_0$ is small enough,
\begin{eqnarray*}
  \|\Phi\|_{C^{1,\a\wedge\b}(\C_2)}\le C\,,\quad \Phi(z,2)\le-1\,,\quad\Phi(z,-2)\ge1\,,\quad \frac{\pa\Phi}{\pa t}(z,t)\le-\frac12\,,
\end{eqnarray*}
for every $(z,t)\in\C_2$; hence there exists $\zeta\in C^{1,\a\wedge\b}(\D_2;(-1,1))$ with
\begin{equation}\label{pattymelt 7}
  \|\zeta\|_{C^{1,\a\wedge\b}(\D_2)}\le C\,,\qquad \Phi(z,\zeta(z))=0\,,\qquad\forall z\in\D_2\,.
\end{equation}
By \eqref{pattymelt 3} and \eqref{pattymelt 7} we find
\begin{equation}
  \label{pattymelt 4}
\big\{(z,u_1(z))+\zeta(z)\,\nu(z,u_1(z)):z\in\D_2\big\}\subset \Gamma(u_2)\,.
\end{equation}
Again by $\Phi(z,\zeta(z))=0$ we deduce that
\begin{equation}
  \label{pattymelt 5}
  \zeta(z)=\sqrt{1+|\nabla u_1(z)|^2}\bigg(u_2\Big(z-\zeta(z)\,\frac{\nabla u_1(z)}{\sqrt{1+|\nabla u_1(z)|^2}}\Big)-
u_1(z)\bigg)\,,
\end{equation}
so that, by \eqref{pattymelt 1},
\begin{eqnarray*}
  \|\zeta\|_{C^0(\D_2)}\le \sqrt{1+\s_0^2}\,\Big(\|u_2-u_1\|_{C^0(\D_2)}+\s_0^2\,\|\zeta\|_{C^0(\D_2)}\Big)
\end{eqnarray*}
and thus $\|\zeta\|_{C^0(\D_2)}\le C\,\|u_1-u_2\|_{C^0(\D_2)}$. Similarly, by differentiating \eqref{pattymelt 5}, by exploiting the fact that $u_1\in C^{2,\a}(\D_2)$ and thanks to \eqref{pattymelt 1}, one finds that
\begin{equation}
  \label{pattymelt 6}
\|\zeta\|_{C^1(\D_2)}\le C\,\|u_1-u_2\|_{C^1(\D_2)}\,.
\end{equation}
We finally define $\psi\in C^{1,\a\wedge\b}(\C_2\cap\Gamma(u_1))$ by the identity $\psi(z,u_1(z))=\zeta(z)$, $z\in\D_2$. In this way \eqref{pattymelt fine 2 C1alphabeta} and \eqref{pattymelt fine 2 C1} follow immediately from \eqref{pattymelt 1}, \eqref{pattymelt 7} and \eqref{pattymelt 6}, whereas \eqref{pattymelt 4} gives the second inclusion in \eqref{pattymelt fine 1}.
%\begin{equation}
%  \label{pattymelt 8}
%\Big\{x+\psi(x)\,\nu(x):x\in\C_2\cap\Gamma(u_1)\Big\}\subset \Gamma(u_2)\,.
%\end{equation}
The first inclusion in \eqref{pattymelt fine 1} is obtained by noticing that: (i) up to further decreasing the value of $\s_0$ we have
\begin{equation}
  \label{pattymelt 9}
  \left\{\begin{array}
    {l}
    x\in\C_2\cap\Gamma(u_1)\,,
    \\
    x+t\,\nu(x)\,,x+s\,\nu(x)\in\Gamma(u_2)
  \end{array}\right .\qquad\Rightarrow\qquad t=s\,;
\end{equation}
(ii) there exists $\eta>0$ (depending on $L$ only) such that every $y\in N_\eta(\C_2\cap\Gamma(u_1))$ has a unique projection over $\C_2\cap\Gamma(u_1)$. Since (by \eqref{pattymelt 1} and provided $\s_0$ is small enough) we can entail
\[
\C_1\cap\Gamma(u_2)\subset N_\eta(\C_2\cap\Gamma(u_1))\,,
\]
by (ii) we find that for every $y\in \C_1\cap\Gamma(u_2)$ there exists a unique $\hat{y}\in \C_2\cap\Gamma(u_1)$ such that
\[
y=\hat{y}+\dist(y,\C_2\cap\Gamma(u_1))\,\nu(\hat{y})\,.
\]
By the second inclusion in \eqref{pattymelt fine 1}, $\hat{y}\in \C_2\cap\Gamma(u_1)$ implies that $\hat{y}+\psi(\hat{y})\,\nu(\hat{y})\in \Gamma(u_2)$. By \eqref{pattymelt 9} we thus find $\dist(y,\C_2\cap\Gamma(u_1))=\psi(\hat{y})$, and thus $y=\hat{y}+\psi(\hat{y})\,\nu(\hat{y})$. The first inclusion in \eqref{pattymelt fine 1} is thus proved.
\end{proof}

\begin{lemma}\label{lemma fondamentale}
  If $n\ge 2$, $\a\in[0,1]$, $\beta\in(0,1)$, $\Lambda\ge 0$, and $E$ is an open set with $0\in\pa E$ and
  \begin{equation}\label{wang 2.1}
  \C_1\cap E=\big\{z+s\,e_n:z\in\D_1\,,v(z)<s<1\big\}\,,
  \end{equation}
  for some $v\in C^{2,\a}(\D_1)$ with $v(0)=0$ and $\nabla v(0)=0$, then there exists $r\in(0,1/64)$ (depending on $n$, $\a$, $\beta$, $\Lambda$ and $\|v\|_{C^{2,\a}(\D_1)}$) with the following property.

  If $\{E_k\}_{k\in\N}$ is a sequence of $(\Lambda,r_0)$-minimizing sets in $B_{32\,r}$ with $|B_{32\,r}\cap (E_k\Delta E)|\to 0$ as $k\to\infty$, then there exist $k_0\in\N$ and $\{\psi_k\}_{k\ge k_0}\subset C^{1,\a\wedge\beta}(\C_{2\,r}\cap\pa E)$ such that
  \begin{gather}\label{tacodeli fine 1}
  \C_r\cap\pa E_k\subset (\Id+\psi_k\nu_{E})(\C_{2r}\cap\pa E)\subset \C_{4\,r}\cap\pa E_k\,,\qquad\forall k\ge k_0\,,
  \\
  \label{tacodeli fine 2}
  \sup_{k\ge k_0}\|\psi_k\|_{C^{1,\a\wedge\beta}(\C_{2r}\cap\pa E)}\le C\,,\qquad \lim_{k\to\infty}\|\psi_k\|_{C^1(\C_{2r}\cap\pa E)}=0\,,
  \end{gather}
  where $C=C(n,\a,\beta,\Lambda,\|v\|_{C^{2,\a}(\D_1)})$. Moreover, when $n=2$, one can take $\beta=1$.
\end{lemma}

\begin{proof}
 First, we note that by \eqref{wang 2.1} one has
 \begin{equation}
 \label{wang 2.2}
  \C_1\cap \pa E=\big\{z+v(z)\,e_n:z\in\D_1\big\}\,.
 \end{equation}
 Second, we set $M=\|v\|_{C^{2,\a}(\D_1)}$, and exploit $v(0)=0$ and $\nabla v(0)=0$ to find $r\in(0,1/64)$ (depending on $n$, $\Lambda$, and $M$) in such a way that
 \begin{equation}
  \label{all00}
  \exc_{64\,r}(E)+\Lambda\,(64r)\le\s\,,\qquad \|v\|_{C^1(\D_{4r})}^*\le \s\,,
 \end{equation}
 for a positive constant $\s$ to be chosen later depending on $n$, $\a$, $\beta$, $\Lambda$ and $M$. Since $0\in\pa E$, $E_k$ is a $(\Lambda,r_0)$-minimizing set in $B_{32\,r}$, and $|(E_k\Delta E)\cap B_{32r}|\to 0$ as $k\to\infty$, by \cite[Theorem 21.14-(ii)]{maggiBOOK} there exists $\{x_k\}_{k\in\N}$ with $x_k\in\pa E_k$ and $x_k\to 0$ as $k\to\infty$. By \cite[Proposition 22.6]{maggiBOOK}, for a.e. $t\in(16\,r,32\,r)$,
 \[
 \lim_{k\to\infty}\exc_{x_k,t}(E_k)=\lim_{k\to\infty}\exc_{t}(E_k-x_k)
 =\exc_t(E)\le C(n) \,\exc_{64\,r}(E)\,.
 \]
 By  \eqref{all00} there exists $k_0\in\N$ such that \begin{eqnarray}
  \label{epsstarrr}
  \exc_{x_k,t}(E_k)+\Lambda\,t\le C(n)\,\s\,,\qquad\forall  k\ge k_0\,.
 \end{eqnarray}
 Provided $\s$ is suitably small with respect to the constant $\e_*(n)$ introduced in Theorem \ref{thm small excess criterion}, one finds that for every $k\ge k_0$ there exists $w_k:\D_{x_k,t/2}\to\R$ such that
 \begin{eqnarray}\label{notethat}
  \C_{x_k,t/2}\cap E_k&=&\big\{z+s\,e_n:z\in\D_{x_k,t/2}\,,w_k(z)\le s\le \frac{t}2\big\}\,,
  \\
  \label{notethat2}
  \C_{x_k,t/2}\cap \pa E_k&=&\big\{z+w_k(z)\,e_n:z\in\D_{x_k,t/2}\big\}\,,
 \\
 \label{patty}
 \|w_k\|_{C^{1,\b}(\D_{x_k,t/2})}^*&\le&C(n,\b)\,.
 \end{eqnarray}
 (Note that \eqref{notethat} follows by \eqref{notethat2}, \eqref{wang 2.1} and the fact that $|B_{32\,r}\cap (E_k\Delta E)|\to 0$.) By composing the functions $w_k$ with vanishing horizontal and vertical translations, and since $t/2>8\,r$, we actually find that, up to further increasing the value of $k_0$, then for every $k\ge k_0$ there exists $v_k:\D_{8\,r}\to\R$ such that
\begin{eqnarray}
\label{wang 2.3}
  \C_{8\,r}\cap E_k&=&\big\{z+s\,e_n:z\in\D_{8\,r}\,,v_k(z)\le s\le 8\,r\big\}\,,
  \\
  \label{wang 2.4}
  \C_{8\,r}\cap \pa E_k &=&\big\{z+v_k(z)\,e_n:z\in\D_{8\,r}\big\}\,,
\\
  \label{wang 3.1}
  \|v_k\|_{C^{1,\b}(\D_{8\,r})}^*&\le& C(n,\b)\,.
\end{eqnarray}
If we set $L=\max\{M/r,C(n,\b)\}$ with $C(n,\b)$ as in \eqref{wang 3.1}, then
by \eqref{wang 3.1} and by definition of $M$ we have
\[
\max\big\{\|v\|_{C^{2,\a}(\D_{4\,r})}^*,\|v\|_{C^{1,\b}(\D_{4\,r})}^*\big\}\le L\,,\qquad\forall k\ge k_0\,,
\]
Let $\s_0=\s_0(n,\a,\b,L)=\s_0(n,\a,\b,\Lambda,M)$ be determined as in Lemma \ref{lemma facile}. By \eqref{wang 2.1}, \eqref{wang 2.3} and $|B_{32\,r}\cap (E_k\Delta E)|\to 0$ we have $v_k\to v$ in $L^1(\D_{8\,r})$, thus by \eqref{wang 3.1} we find $v_k\to v$ in $C^1(\D_{8\,r})$, so that, up to further increasing $k_0$, decreasing $\s$ in terms of $\s_0$, and thanks to \eqref{all00},
\[
\max\big\{\|v\|_{C^1(\D_{4\,r})}^*,\|v_k\|_{C^1(\D_{4\,r})}^*\big\}\le\s_0\,,\qquad\forall k\ge k_0\,.
\]
We thus apply Lemma \ref{lemma facile} and find $\psi_k\in C^{1,\a\wedge\b}(\C_{2\,r}\cap\pa E)$ with the required properties.
\end{proof}

\subsection{Infiltration lemma and consequences}\label{section infiltratio lemma} In this section we exploit an infiltration lemma (Lemma \ref{lemma infiltration} -- which is a special case of \cite[Lemma 4.6]{LeonardiTamanini}, see also \cite[Theorem 3.1]{leonardifluid} for a similar result in the context of immiscible fluids) together with Theorem \ref{thm small excess criterion} to address various regularity properties of $(\Lambda,r_0)$-minimizing clusters, and to prove some basic convergence properties, see Theorem \ref{thm boundary hausdorff} and Theorem \ref{thm normal representation part one}.

\begin{lemma}[Infiltration lemma]\label{lemma infiltration}
  There exists a positive constant $\eta_0=\eta_0(n)<\om_n$ with the following property: if $\E$ is a $(\Lambda,r_0)$-minimizing cluster in $A$, then there exists a positive constant $r_1\le r_0$ (depending on $\Lambda$ and $r_0$ only) such that, if
  \begin{equation}
    \label{inf1}
      \sum_{h\in H}|\,\E(h)\cap B_{x,r}|\le \eta_0\,r^n\,,
  \end{equation}
  for some $r\le r_1$, $H\subset\{0,\dots,N\}$, and $x\in\R^n$ with $B_{x,r}\cc A$, then
  \begin{equation}
    \label{inf2}
      \sum_{h\in H}|\,\E(h)\cap B_{x,r/2}|=0\,.
  \end{equation}
\end{lemma}

\begin{proof}
  By arguing as in \cite[Lemma 30.2]{maggiBOOK} one sees that if $\E$ is a $N$-cluster in $A$ such that
  \begin{equation}
    \label{inf3}
      P(\E;B_{x,r})\le P(\F;B_{x,r})+C_0\,|\,\vol(\E)-\vol(\F)|\,,
  \end{equation}
  whenever $\E(h)\Delta \F(h)\cc B_{x,r}\cc A$ for some $x\in\R^n$, $r<r_0$ and every $h=1,...,N$, then
%  \begin{figure}
%  \input{infiltra.pstex_t}\caption{{\small A situation excluded by the infiltration lemma: if the volume of the cusp belonging to the chamber $\E(2)$ is below a critical fraction of the volume of $B_{x,r}$, then the $(\Lambda,r_0)$-perimeter minimality of $\E$ forbids the infiltration of $\E(2)$ into $B(x,r/2)$.}}
%  \end{figure}
  \eqref{inf1} implies \eqref{inf2} with $r_1=\min\{r_0,1/8C_0\}$. This is achieved by exploiting the perturbed minimality inequality \eqref{inf3} on comparison clusters $\F$ having the property that, if $0\le h\le N$, then either $\F(h)\subset\E(h)$ or $\E(h)\subset\F(h)$. We now notice that, on such clusters $\F$ one has
  \[
  \d(\E,\F)=\sum_{h=1}^N|\,|\E(h)|-|\F(h)|\,|\le \sqrt{N}\,|\,\vol(\E)-\vol(\F)\,|\,.
  \]
  Therefore, if $\E$ is a $(\Lambda,r_0)$-minimizing cluster in $A$, then \eqref{inf3} holds on every comparison cluster $\F$ as above with $C_0=\sqrt{N}\Lambda$, and we can argue as in \cite[Lemma 30.2]{maggiBOOK} to prove the lemma (with $r_1=\min\{r_0,1/8\sqrt{N}\Lambda\}$).
\end{proof}

%%%%%In our improved convergence theorems we shall assume that the limit cluster $\E$ is bounded (i.e., there exists $R>0$ such that $\E(h)\subset B_R$ for every $h=1,...,N$). This assumption is actually equivalent to asking that $\E$ has finite perimeter.
%%%%%
%%%%%\begin{corollary}\label{remark perimeter bounded}
%%%%%  If $\E$ is a $(\Lambda,r_0)$-minimizing cluster in $\R^n$, then $\E$ is bounded if and only if $P(\E)<\infty$.
%%%%%\end{corollary}
%%%%%
%%%%%\begin{proof}
%%%%%  Clearly, if $\E$ is a $(\Lambda,r_0)$-minimizing cluster in $\R^n$, then $P(\E;K)<\infty$ for every compact set $K\subset\R^n$. This shows that if $\E$ is bounded, then $P(\E)<\infty$. To prove the converse implication, we argue by contradiction: precisely, we set $E=\bigcup_{h=1}^N\E(h)$ and consider $\{x_{h}\}_{h\in\N}\subset\pa E$ with $|x_{h}|\to \infty$ as $h\to \infty$ and $|x_{h}-x_{k}|>2$ for all $h\neq k$. Under this assumption we have $\sum_{h\in\N} P(E;B(x_{h},1)) \leq P(E) <+\infty$, so that, by the relative isoperimetric inequality in balls,
%%%%%  \[
%%%%%  0=\lim_{h\to\infty}P(E;B(x_{h},1))\ge c(n)\lim_{h\to\infty}\min\{|E\cap B(x_{h},1)|,\ |\E(0)\cap B(x_{h},1)|\}^{(n-1)/n}\,.
%%%%%  \]
%%%%%  This leads to a contradiction with the lower density bounds \eqref{density ndim}, and proves the assertion.
%%%%%\end{proof}

\begin{corollary}[Almost everywhere regularity]\label{corollary interfaces}
  If $\E$ is a $(\Lambda,r_0)$-minimizing cluster in $A$, then $\pa^*\E$ is a $C^{1,\a}$-hypersurface for every $\a\in(0,1)$ ($C^{1,1}$ if $n=2$), it is relatively open inside $A\cap\pa\E$, and $\H^{n-1}(\S_A(\E))=0$. Moreover, if $n=2$, then we can replace $C^{1,\a}$ with $C^{1,1}$.
\end{corollary}

\begin{proof}
  {\it Step one}: We prove that there exists $c(n)\in(0,1)$ and $r_1\le r_0$ (depending on $\E$), such that, if $0\le h\le N$, $x\in\pa\E(h)$, and $r<r_1$ is such that $B_{x,r}\cc A$, then
  \begin{eqnarray}\label{density ndim}
    c(n)\le \frac{|\E(h)\cap B_{x,r}|}{\om_n\,r^n}\le (1-c(n))\,,
    \\\label{density n-1dim}
    c(n)\,\le\frac{P(\E(h);B_{x,r})}{r^{n-1}}\le C(n,\Lambda)\,(1+r)\,.
  \end{eqnarray}
  Indeed, Lemma \ref{lemma infiltration} implies \eqref{density ndim} with $c(n)=\eta_0(n)/\om_n$; see \cite[Section 30.2]{maggiBOOK}. Up to further decreasing the value of $c(n)$, the lower bound in \eqref{density n-1dim} follows from \eqref{density ndim} and the relative isoperimetric inequality on balls, see \cite[Proposition 12.37]{maggiBOOK}. Finally, by testing \eqref{almost-minimizer cluster} on $\F(h)=\E(h)\setminus B_{x,r}$, $1\le h\le N$, we find that $P(\E(h);B_{x,r})\le n\om_n\,r^{n-1}+\Lambda\,\om_n\,r^n$, whence the upper bound in \eqref{density n-1dim}.

  \medskip

  \noindent {\it Step two}: We show that if $x\in\E(h,k)=\pa^*\E(h)\cap\pa^*\E(k)$, then there exists $r_x\in(0,r_0)$ such that $|\E(j)\cap B_{x,r_x}|=0$ if $j\ne h,k$ and $B_{x,r_x}\cc A$. Indeed, by standard density estimates (see \cite[Exercise 29.6]{maggiBOOK}), we have
  \[
  \lim_{r\to 0^+}\frac{|\,\E(h)\cap B_{x,r}|}{\om_n\,r^n}+\frac{|\,\E(k)\cap B_{x,r}|}{\om_n\,r^n}=1\,,
  \]
  so that the existence of $r_x$ follows from Lemma \ref{lemma infiltration}. As a consequence, \eqref{almost-minimizer cluster} implies that both $\E(h)$ and $\E(k)$ are $(\Lambda,r_0)$-minimizing sets on $B_{x,r_x}$. By Theorem \ref{thm small excess criterion} and Remark \ref{remark reduced boundary}, $\pa^*\E$ is a $C^{1,\a}$-hypersurface for every $\a\in(0,1)$ ($C^{1,1}$ if $n=2$) and it is relatively open inside $A\cap \pa\E$. The  lower $(n-1)$-dimensional estimate in \eqref{density n-1dim} implies $\H^{n-1}(\S_A(\E))=0$ by a classical argument (see for example \cite[Theorem 16.14]{maggiBOOK}).
\end{proof}

%%\begin{remark}
%%  {\rm Notice that if $x\in\S(\E)\cap\pa\E(h)$ for some $h=0,\dots,N$, then $x\in\pa\E(h)\setminus\pa^*\E(h)$. Since our argument does not imply that $\E(h)$ is a $(\Lambda,r_0)$-minimizing set on $B_{x,r}$ for any $r>0$, we cannot apply the analysis of singularities for $(\Lambda,r_0)$-minimizing sets to the chamber $\E(h)$, and thus deduce that $\pa\E(h)\setminus\pa^*\E(h)$ has co-dimension at least $8$; and, in fact, this last assertion is false, as $\pa\E(h)\setminus\pa^*\E(h)$ can be a $(n-2)$-dimensional surface, as it happens, for example, in the case of double bubbles.}
%%\end{remark}

\begin{corollary}[Local finiteness away from the singular set]\label{corollary local finiteness}
  If $\E$ is a $(\Lambda,r_0)$-minimizing $N$-cluster in $A$, $\rho>0$, and $A'\cc A$ is open, then $(A'\cap\pa\E)\setminus \cl(I_\rho(\S_A(\E)))$ is the union of finitely many disjoint connected hypersurfaces.
\end{corollary}

\begin{proof}
  By Corollary \ref{corollary interfaces}, we can directly assume that $\pa^*\E=\bigcup_{i\in\N}S_i$, where each $S_i$ is a nonempty connected $C^1$-hypersurface with $S_i\cap S_j=\emptyset$ for $i\ne j$. If we set $S_i^\rho=(A'\cap S_i)\setminus \cl(I_\rho(\S_A(\E)))$ then $\{S_i^\rho\}_{i\in\N}$  is a disjoint family of connected $C^1$-hypersurfaces whose union is equal to $(A'\cap\pa\E)\setminus \cl(I_\rho(\S_A(\E)))$. We claim that only finitely many elements of $\{S_i^\rho\}_{i\in\N}$ are nonempty. If this were not the case, then, up to extracting subsequences, we could find $\{x_i\}_{i\in\N}\subset(A'\cap\pa\E)\setminus \cl(I_\rho(\S_A(\E)))$ with $x_i\in S_i$ for every $i\in \N$ and $x_i\to x$ for some $x\in (\cl(A')\cap\pa\E)\setminus I_\rho(\S_A(\E))$. Since $x\in\pa^*\E$, by Theorem \ref{thm small excess criterion} and step two in the proof of Corollary \ref{corollary interfaces}, there exists $r_x>0$ and $\nu\in\SS^{n-1}$ such that $\pa\E\cap \C^\nu_{x,r_x}=\pa^*\E\cap\C^\nu_{x,r_x}=(\Id+v\,\nu)(\D^\nu_{x,r_x})$ for some $v\in C^1(\D^\nu_{x,r_x})$. By connectedness, we infer that $S_{i}\cap \C_{x,r_{x}}^{\nu} = S_{j}\cap \C_{x,r_{x}}^{\nu}$, which contradicts the assumption on $S_{i}$ and $S_{j}$.
\end{proof}

\begin{corollary}[Bounded mean curvature]\label{corollary mean curvature}
  If $\E$ is a $(\Lambda,r_0)$-minimizing cluster in $A$, then $A\cap\pa\E$ is a locally $\H^{n-1}$-rectifiable set with bounded mean curvature in $A$, and
  \begin{equation}
    \label{Lambda bounded mean curvature}
      \|{\bf H}_{\pa\E}\|_{L^\infty(\H^{n-1}\llcorner(A\cap\pa\E))}\le \Lambda\,.
  \end{equation}
%  In particular, if $\E$ is bounded, then
%  \begin{equation}
%    \label{stima diametro}
%    \diam(\pa\E)\ge\frac{n-1}{8\,\Lambda}\,.
%  \end{equation}
\end{corollary}

\begin{proof} Since $\pa^*\E$ is locally $\H^{n-1}$-rectifiable in $A$ and $\H^{n-1}(\S_A(\E))=0$, one finds immediately that $A\cap\pa\E$ is a locally $\H^{n-1}$-rectifiable set in $A$. By Riesz theorem and Lebesgue--Besicovitch differentiation theorem, in order to prove \eqref{Lambda bounded mean curvature} it suffices to show that
\begin{equation}
  \label{ovvia}
    \int_{\pa\E}\Div_{\pa\E}T\,d\H^{n-1}\le (1+\eta)\,\Lambda\,P(\E;B_{x,r})\,,
\end{equation}
whenever $B_{x,r}\cc A$, $r<r_0$, $T\in C^1_c(B_{x,r};\R^n)$ with $|T|\le 1$. To this end, let $\{f_t\}_{|t|<\e}$ be the flow with initial velocity $T$, so that (see, e.g., \cite[Theorem 17.5]{maggiBOOK})
\[
P(f_t(E);B_{x,r})=P(E;B_{x,r})+t\,\int_{\pa^*E}\Div_{\pa E}T\,d\H^{n-1}+O(t^2)\,,
\]
for every set $E$ of finite perimeter in $B_{x,r}$. By Lemma \ref{lemma distanza L1} (see Appendix \ref{section volume fixing}) one sees that for every $\eta>0$ it is possible to decrease $\e>0$ in such a way that
\[
|f_t(E)\Delta E|\le(1+\eta)\,P(E;B_{x,r})\,|t|\,,\qquad\forall |t|<\e\,,
\]
for every Borel set $E\subset\R^n$. Up to further decreasing the value of $\e$ we have $f_t(\E(h))\Delta\E(h)\cc B_{x,r}$ for every $h=1,...,N$, so that by \eqref{almost-minimizer cluster} one finds
\begin{eqnarray*}
  P(\E;B_{x,r})&\le& P(f_t(\E);B_{x,r})+  \frac{\Lambda}2\,\sum_{h=0}^N|\E(h)\Delta f_t(\E(h))|
  \\
  &=& P(\E;B_{x,r})+t\,\int_{\pa\E}\Div_{\pa\E}T\,d\H^{n-1}+O(t^2)+(1+\eta)\,\Lambda\,|t|\,P(\E;B_{x,r})\,,
\end{eqnarray*}
and immediately deduces \eqref{ovvia}.
\end{proof}

We now start to consider the situation when
\begin{equation}\label{ipotesi ripetuta mille volte}
  \begin{split}
    &\mbox{$\{\E_k\}_{k\in\N}$ are $(\Lambda,r_0)$-minimizing $N$-clusters in $A$}
     \\
    &\mbox{and $\E$ is a $N$-cluster in $A$ with $\d_A(\E_k,\E)\to 0$ as $k\to\infty$.}
  \end{split}
\end{equation}
Note that in this situation, by arguing exactly, say, as in the proof of \cite[Theorem 21.14]{maggiBOOK}, one has that $\E$ is also a $(\Lambda,r_0)$-minimizing cluster in $A$. As a further corollary of the infiltration lemma and of Theorem \ref{thm small excess criterion} we have the following theorem.

\begin{theorem}[Hausdorff convergence of boundaries]
  \label{thm boundary hausdorff}
  If \eqref{ipotesi ripetuta mille volte} holds, then for every $A'\cc A$ one has $\hd_{A'}(\pa\E_k,\pa\E)\to 0$ as $k\to\infty$, and actually
  \begin{equation}
    \label{boundary haudorff interfaces}
    \lim_{k\to\infty}\hd_{A'}\Big(\pa\E_k(i)\cap\pa\E_k(j),\pa\E(i)\cap\pa\E(j)\Big)=0\,,
    \qquad\mbox{for every $0\le i<j\le N$}\,.
  \end{equation}
  Moreover, for every $\e>0$ there exist $k_0\in\N$ such that
  \begin{eqnarray}
    \label{inclusions bordi}
      \S_{A'}(\E_k)\subset I_\e(\S_A(\E))\,,\qquad\forall k\ge k_0\,.
  \end{eqnarray}
%  \begin{eqnarray}
%    \label{inclusions bordi}
%      \pa\E_k\subset I_\e(\pa\E)\,,\qquad\pa\E\subset I_\e(\pa\E_k)\,,\qquad
%      \S(\E_k)\subset I_\e(\S(\E))\,,
%  \end{eqnarray}
\end{theorem}

\begin{remark}
  {\rm We are not able, in general, to prove the inclusion $\S_{A'}(\E)\subset I_\e(\S_A(\E_k))$ for $k$ large, and thus infer the full Hausdorff convergence $\S_A(\E_k)$ to $\S_A(\E)$ in every $A'\cc A$. We can achieve this if $n=2$, see Theorem \ref{thm singular hausdorff} below, and if $n=3$, see \cite{CiLeMaIC2}.}
\end{remark}

\begin{remark}
  {\rm Note that if $A'\cap\pa\E(i)\cap\pa\E(j)=\emptyset$, then \eqref{boundary haudorff interfaces} forces $A'\cap\pa\E_k(i)\cap\pa\E_k(j)=\emptyset$ for every $k$ large enough. Indeed, $\hd_{A'}(\emptyset,T)=0$ if $T\cap A'=\emptyset$, with $\hd_{A'}(\emptyset,T)=+\infty$ whenever $T\cap A'\ne\emptyset$.}
\end{remark}

\begin{proof}[Proof of Theorem \ref{thm boundary hausdorff}] {\it Step one}: We prove \eqref{boundary haudorff interfaces}. To this end, let us fix $0\le i<j\le N$, set
\[
S^k_{i,j}=\pa\E_k(i)\cap\pa\E_k(j)\,,\qquad S_{i,j}=\pa\E(i)\cap\pa\E(j)\,,
\]
and show that for every $\e>0$ there exists $k_0\in\N$ such that
\begin{eqnarray}
  \label{taco1}
  A'\cap S^k_{i,j}\subset I_\e(S_{i,j})\,,\qquad A'\cap S_{i,j}\subset I_\e(S^k_{i,j})\,,\qquad\forall k\ge k_0\,.
\end{eqnarray}
We prove the first inclusion in \eqref{taco1} by contradiction. Let us consider $x_k\in A'\cap S^k_{i,j}$ with $\dist(x_k,S_{i,j})>\e$ for every $k\in\N$. (Note that if $S_{i,j}=\emptyset$, then $\dist(x,S_{i,j})=+\infty$ for every $x\in\R^n$ and contradicting the first inclusion in \eqref{taco1} exactly amounts in saying that $A'\cap S^k_{i,j}\ne\emptyset$.)
%If (up to extracting subsequences) $|x_k|\to\infty$, then, as $\E$ is bounded, one has $|\E(h)\cap B_{x_k,r_1}|=0$ for every $1\le h\le N$, $k\ge k_0$, and with $r_1$ as in Lemma \ref{lemma infiltration}; correspondingly, up to increasing the value of $k_0$,
%\[
%|\E_k(h)\cap B_{x_k,r_1}|\le\d(\E_k,\E)<\eta_0\,r_1^n\,,\qquad\forall k\ge k_0\,,
%\]
%and thus $|\E_k(h)\cap B_{x_k,r_1/2}|=0$ for every $1\le h\le N$ and $k\ge k_0$ by Lemma \ref{lemma infiltration}. In particular, by \eqref{spt normalization}, $x_k$ do not belong to $\pa\E_k$ for $k$ large enough, a contradiction.
Up to extracting subsequences, we may assume that $x_k\to x$ for some $x\in\cl(A')\subset A$. Since $\dist(x,S_{i,j})\ge \e$, by \eqref{spt normalization} there exists $r_x<\dist(x,\pa A)$ such that
\begin{eqnarray*}
  \mbox{either}\qquad|B_{x,r_x}\cap\E(i)|=0\,,\qquad\mbox{or}\qquad|B_{x,r_x}\cap\E(i)|=\om_n\,r_x^n\,,
  \\
  \mbox{or}\qquad|B_{x,r_x}\cap\E(j)|=0\,,\qquad\mbox{or}\qquad|B_{x,r_x}\cap\E(j)|=\om_n\,r_x^n\,.
\end{eqnarray*}
For $r_1$ as in Lemma \ref{lemma infiltration}, let $s_x=\min\{r_x,r_1\}/2$, then for $k\ge k_0$ one has
\begin{eqnarray*}
  \mbox{either}\qquad|B_{x_k,2s_x}\cap\E_k(i)|<\eta_0\,(2s_x)^n\,,\qquad\mbox{or}\qquad|B_{x_k,2s_x}\cap\E_k(i)|>(\om_n-\eta_0)\,(2s_x)^n\,,
  \\
  \mbox{or}\qquad|B_{x_k,2s_x}\cap\E_k(j)|<\eta_0\,(2s_x)^n\,,\qquad\mbox{or}\qquad|B_{x_k,2s_x}\cap\E_k(j)|>(\om_n-\eta_0)\,(2s_x)^n\,,
\end{eqnarray*}
and thus, by Lemma \ref{lemma infiltration},
\begin{eqnarray*}
  \mbox{either}\qquad|B_{x_k,s_x}\cap\E_k(i)|=0\,,\qquad\mbox{or}\qquad|B_{x_k,s_x}\cap\E_k(i)|=\om_n\,s_x^n\,,
  \\
  \mbox{or}\qquad|B_{x_k,s_x}\cap\E_k(j)|=0\,,\qquad\mbox{or}\qquad|B_{x_k,s_x}\cap\E_k(j)|=\om_n\,s_x^n\,.
\end{eqnarray*}
By \eqref{spt normalization}, $x_k\in A'\setminus S^k_{i,j}$ for $k$ large, a contradiction. We now prove the second inclusion in \eqref{taco1}: by contradiction, there exist $x\in A'\cap S_{i,j}$ and $\e>0$ such that $B_{x,\e}\cap S^k_{i,j}=\emptyset$, i.e., by \eqref{spt normalization},
\begin{eqnarray*}
  \mbox{either}\qquad|B_{x,\e}\cap\E_k(i)|=0\,,\qquad\mbox{or}\qquad|B_{x,\e}\cap\E_k(i)|=\om_n\,\e^n\,,
  \\
  \mbox{or}\qquad|B_{x,\e}\cap\E_k(j)|=0\,,\qquad\mbox{or}\qquad|B_{x,\e}\cap\E_k(j)|=\om_n\,\e^n\,,
\end{eqnarray*}
for infinitely many values of $k$; by letting $k\to\infty$ along such values we thus find that $x\not\in S_{i,j}$.

\medskip

\noindent {\it Step two}: We prove \eqref{inclusions bordi}. Should \eqref{inclusions bordi} fail, we could find $\e>0$ and $x_k\in \S_{A'}(\E_k)$ with $\dist(x_k,\S(\E))>\e$ for infinitely many $k\in\N$. By step one, up to extracting subsequences, $x_k\to x$ as $k\to\infty$ for some $x\in A\cap\pa\E$. Since $\dist(x,\S(\E))\ge\e$, we have $x\in\pa^*\E$. By step two in the proof of Corollary \ref{corollary interfaces}, there exist $0\le h<h'\le N$ and $2\,r_*<\min\{r_1,\dist(x,\pa A)\}$ such that $x\in\E(h,h')$ and $B_{x,2\,r_*}\subset \E(h)\cup\E(h')$. Hence, for some $k_0\in\N$ we have
  \[
  |\E_k(h)\cap B_{x_k,2\,r_*}|+|\E_k(h')\cap B_{x_k,2\,r_*}|\ge(\om_n-\eta_0)\,r_*^n\,,\qquad\forall k\ge k_0\,.
  \]
By Lemma \ref{lemma infiltration}, $\E_k(j)\cap B_{x_k,r_*}=\emptyset$ for every $k\ge k_0$ and $j\ne h,h'$, so that $\E_k(h)$ is a $(\Lambda,r_0)$-minimizing set in $B_{x_k,r_*}$. By arguing as in Lemma \ref{lemma fondamentale} we find that
\begin{equation}
  \label{birthday1}
  \exc^\nu_{x,r}(\E(h))=\lim_{k\to\infty}\exc^\nu_{x_k,r}(\E_k(h))\,,\qquad\mbox{for a.e. $r<r_*$}\,.
\end{equation}
Since $x\in\E(h,h')$, by Remark \ref{remark reduced boundary} there exist $r_{**}<\min\{r_*,r_0\}$ and $\nu\in \SS^{n-1}$ such that
\begin{equation}
  \label{birthday2}
  \exc_{x,r_{**}}^\nu(\E(h))+\Lambda\,r_{**}\le\frac{\e_*(n)}{2^n}\,,
\end{equation}
where $\e_*(n)$ is defined as in Theorem \ref{thm small excess criterion}.
%Since, trivially, $\exc_{x,r}^\nu(\E(h))\le(r_{**}/r)^{n-1}\,\exc_{x,r_{**}}^\nu(\E(h))$ for every $r<r_{**}$,
By \eqref{birthday1} and \eqref{birthday2} we conclude that, for some $r\in(r_{**}/2,r_{**})$ and up to increasing $k_0$, $\exc^\nu_{x_k,r}(\E_k(h))+\Lambda\,r<\e_*$ for every $k\ge k_0$. By Theorem \ref{thm small excess criterion}, $B_{x_k,r/2}\cap\pa^*\E_k(h)$ is a $C^{1,\a}$-hypersurface, against $x_k\in\S_{A'}(\E_k)$.
\end{proof}

We now set
\[
[\pa\E]_\rho=(A\cap\pa\E)\setminus I_\rho(\S_A(\E))\,,
\]
recall the definition \eqref{Nespilon} of normal $\e$-neighborhood $N_\e(S)$ of a manifold $S\subset\R^n$, and then combine Theorem \ref{thm small excess criterion} and Theorem \ref{thm boundary hausdorff} to obtain the following weak improved convergence theorem.
%show that if $\{\E_k\}_{k\in\N}$ is a sequence of $(\Lambda,r_0)$-minimizing clusters in $A$ with $\d_A(\E_k,\E)\to 0$  for some $\E$ with $\pa^*\E$ of class $C^{2,1}$, then, for every $\rho<\rho_0$ and $A'\cc A$ we can cover $(A'\cap\pa\E_k)\setminus I_{2\rho}(\S_A(\E))$ with $(\Id+\psi_k\nu_\E)(A'\cap[\pa \E]_\rho)\subset\pa^*\E_k$, where $\psi_k\to 0$ in $C^1(A'\cap[\pa\E]_\rho)$ as $k\to\infty$ and $\nu_\E$ is a $C^{1,1}$-unit normal vector field to $\pa^*\E$.

\begin{theorem}[Normal representation theorem]\label{thm normal representation part one}
  If $\Lambda\ge0$, $\a\in(0,1)$, and $\E$ is a $N$-cluster in $A\subset\R^n$ such that $\pa^*\E$ is a $C^{2,1}$-hypersurface, then there exist positive constants $\rho_0$ (depending on $\E$) and $C$ (depending on $\a$, $\Lambda$, and $\E$) with the following property.

  If \eqref{ipotesi ripetuta mille volte} holds, then for every $A'\cc A$ and $\rho<\rho_0$ there exist $k_0\in\N$, $\e\in(0,\rho)$, $\Om$ open with $A'\cc\Om\cc A$, and $\{\psi_k\}_{k\ge k_0}\subset C^{1,\a}(\Om\cap[\pa\E]_\rho)$ such that
  \begin{equation}
    \label{zetak parametrizzano}
      (A'\cap\pa\E_k)\setminus I_{2\rho}(\S_A(\E))\subset (\Id+\psi_k\nu_\E)(\Om\cap[\pa\E]_\rho) \subset\pa^*\E_k\,,
  \end{equation}
  \begin{equation}
    \label{zetak normale para}
      N_\e(A'\cap[\pa\E]_{\rho})\cap\pa\E_k=(\Id+ \psi_k\,\nu_\E)(A'\cap[\pa\E]_\rho)\,,
  \end{equation}
  with
  \begin{equation}
    \label{zetak tendono a zero}
      \lim_{k\to\infty}\|\psi_k\|_{C^1(\Om\cap[\pa\E]_\rho)}=0\,,\qquad \sup_{k\ge k_0}\|\psi_k\|_{C^{1,\a}(\Om\cap[\pa\E]_\rho)}\le C\,.
  \end{equation}
  Moreover, when $n=2$ one can set $\a=1$ in this statement.
\end{theorem}

%\begin{remark}\label{remark improve normal}
%  {\rm Ideally, we would like to show that for every $\e>0$ there exists $\rho>0$ and $k_0\in\N$ such that $[\pa\E_k]_\e$ is covered by a normal deformation of $[\pa\E]_\rho$ for every $k\ge k_0$; that is, we would like to replace $\pa\E_k\setminus I_{2\rho}(\S(\E))$ in \eqref{zetak parametrizzano} with $\pa\E_k\setminus I_{\e}(\S(\E_k))$. We shall obtain this result in dimension $n=2$ and $n=3$ (see Theorem \ref{thm normal representation part two} and Theorem \ref{thm normal representation part two 3d}) by exploiting the Hausdorff convergence of singular sets (see Theorem \ref{thm singular hausdorff} and Theorem \ref{thm singular hausdorff 3d}).}
%\end{remark}

\begin{proof}
 Since $\pa^*\E$ is a $C^{2,1}$-hypersurface, for every $x\in\pa^*\E$ there exist $r_x>0$, $\nu_x\in\SS^{n-1}$ and $v_x\in C^{2,1}(\D^{\nu_x}_{x,64\,r_x})$ with $v_x(x)=0$, $\nabla v_x(x)=0$, and
 \begin{equation}
   \label{vai2}
    \pa\E\cap \C^{\nu_x}_{x,64\,r_x}=(\Id+v_x\,\nu_x)(\D^{\nu_x}_{x,64\,r_x})\,,\qquad \C^{\nu_x}_{x,64\,r_x}\cc A\,.
 \end{equation}
 By Theorem \ref{thm boundary hausdorff}, $\E$ is a $(\Lambda,r_0)$-minimizing cluster in $A$, so that by step two in the proof of Corollary \ref{corollary interfaces} there also exist $0\le h_x<h'_x\le N$ such that, up to further decreasing $r_x$, one has
 \begin{equation}
   \label{vai1}
    |\E(j)\cap \C^{\nu_x}_{x,64\,r_x}|=0\,,\qquad\forall j\ne h_x,h_x'\,,
 \end{equation}
 and thus, taking \eqref{vai2} into account and without loss of generality,
 \begin{equation}
   \label{vai3}
    \C^{\nu_x}_{x,64\,r_x}\cap\E(h_x)=\Big\{z+s\,\nu_x:z\in \D^{\nu_x}_{x,64\,r_x}\,,v_x(z)<s<64\,r_x\Big\}\,.
 \end{equation}
 By Lemma \ref{lemma infiltration} and by \eqref{vai1} there exists $k_x\in\N$ such that
 \begin{equation}
   \label{vai4}
     |\E_k(j)\cap B_{x,32\,r_x}|=0\,,\qquad\forall j\ne h_x,h_x'\,,\qquad\forall k\ge k_x\,,
 \end{equation}
 so that $\E_k(h_x)$ is a $(\Lambda,r_0)$-minimizing set in $B_{x,32\,r_x}$ for $k\ge k_x$. By Lemma \ref{lemma fondamentale} there exist $s_x\in(0,r_x)$ and, up to increasing $k_x$, functions $\psi_{x,k}\in C^{1,\a}(\C^{\nu_x}_{x,2\,s_x}\cap\pa\E(h_x))$ with
 \begin{equation}
   \C^{\nu_x}_{x,s_x}\cap\pa\E_k(h_x)\subset(\Id+\psi_{x,k}\,\nu_{\E_k(h_x)})(\C^{\nu_x}_{x,2\,s_x}\cap\pa\E(h_x))\subset\C^{\nu_x}_{x,4\,s_x}\cap\pa\E_k(h_x)
 \end{equation}
 \begin{equation}
   \sup_{k\ge k_0}\|\psi_{x,k}\|_{C^{1,\a}(\C^{\nu_x}_{x,2\,s_x}\cap\pa\E(h_x))}\le C\,,\qquad\lim_{k\to\infty}
   \|\psi_{x,k}\|_{C^1(\C^{\nu_x}_{x,2\,s_x}\cap\pa\E(h_x))}=0\,,
 \end{equation}
 where $C$ depends on $\a$, $\Lambda$ and $\E$.

 Let $\rho_0>0$ be such that $[\pa\E]_{\rho_0}\ne\emptyset$. For every $\rho\in(0,\rho_0)$ we can find $\{x_i\}_{i=1}^M\subset A'\cap[\pa\E]_{\rho}\subset\pa^*\E$ such that (for $s_i=s_{x_i}$ and $\nu_i=\nu_{x_i}$) one has
 \begin{eqnarray}
  \label{wang 0}
 A'\cap [\pa\E]_\rho\subset\subset\bigcup_{i=1}^M\,\C^{\nu_i}_{x_i,s_i}\,,\qquad \C^{\nu_i}_{x_i,64\,s_i}\cc A\,.
 \end{eqnarray}
 Since $\pa^*\E$ is a $C^2$-hypersurface we can find $\e(\rho)\in(0,\rho)$ such that every point in $N_{\e(\rho)}(A'\cap[\pa\E]_{\rho})$ has a unique projection onto $A'\cap[\pa\E]_\rho$ and
 \begin{equation}
  \label{all1}
  N_{\e(\rho)}(A'\cap [\pa\E]_{\rho})\subset I_{\e(\rho)}(A'\cap[\pa\E]_{\rho})\subset \bigcup_{i=1}^M\,\C^{\nu_i}_{x_i,s_i}\,.
\end{equation}
 By arguing as in the proof of Lemma \ref{lemma facile}, we see that $\psi_{i,k}=\psi_{j,k}$ on $\C^{\nu_i}_{x_i,2\,s_i}\cap\C^{\nu_j}_{x_j,2\,s_j}\cap \pa \E$ for every $i,j$. In particular, if we set
 \[
 \Omega=\bigcup_{i=1}^M\C^{\nu_i}_{x_i,2\,s_i}\,,
 \]
 then we can define $\psi_k\in C^{1,\a}(\Omega\cap\pa\E)$ for every $k\ge k_0=\max\{k_{x_i}:1\le i\le M\}$ by letting $\psi_k=\psi_{x_i,k}$ on $\C^{\nu_i}_{x_i,2\,s_i}\cap\pa\E$. In this way,
 \begin{eqnarray}
  \label{ziozio cantante2}
  \pa\E_k\cap\bigcup_{i=1}^M\C^{\nu_i}_{x_i,s_i}\subset (\Id+\psi_k\nu_\E)(\Omega\cap\pa\E)\subset \pa^*\E_k\,,\qquad\forall k\ge k_0\,,
 \end{eqnarray}
 \begin{equation}
  \label{ziozio cantante}
  \sup_{k\ge k_0}\|\psi_k\|_{C^{1,\a}(\Omega\cap\pa\E)}\le C\,,\qquad\lim_{k\to\infty}\|\psi_k\|_{C^1(\Omega\cap\pa\E)}=0\,.
 \end{equation}
By \eqref{all1}, \eqref{ziozio cantante2}, $A'\cap[\pa\E]_\rho\subset\Omega$, and since $\Id+\psi_k\nu_\E$ is a normal deformation of $\Omega\cap\pa\E$,
\begin{eqnarray*}
N_{\e(\rho)}(A'\cap[\pa\E]_\rho)\cap\pa\E_k&\subset&(\Id+\psi_k\nu_\E)(\Omega\cap\pa\E)\cap N_{\e(\rho)}(A'\cap[\pa\E]_\rho)
\\
&=&(\Id+\psi_k\nu_\E)(A'\cap[\pa\E]_\rho)\subset N_{\e(\rho)}(A'\cap [\pa\E]_\rho)\cap\pa\E_k\,,
\end{eqnarray*}
where the last inclusion follows by the second inclusion in \eqref{ziozio cantante2} provided $\|\psi_k\|_{C^0(\Omega\cap\pa\E)}<\e(\rho)$ for every $k\ge k_0$. This proves \eqref{zetak normale para}. Finally, by Theorem \ref{thm boundary hausdorff}, up to increasing $k_0$, $A'\cap\pa\E_k\subset I_{\e(\rho)}(\pa\E)$ for every $k\ge k_0$, so that $\e(\rho)<\rho$ gives us
\[
(A'\cap\pa\E_k)\setminus I_{2\rho}(\S_A(\E))\subset A'\cap\Big(I_{\e(\rho)}(\pa\E)\setminus I_{2\rho}(\S_A(\E))\Big)\subset A'\cap I_{\e(\rho)}([\pa\E]_\rho)\subset I_{\e(\rho)}(A'\cap[\pa\E]_\rho)\,.
\]
By combining this last inclusion with \eqref{all1} we find that
\[
(A'\cap\pa\E_k)\setminus I_{2\rho}(\S_A(\E))\subset \pa\E_k\cap \bigcup_{i=1}^M\C^{\nu_i}_{x_i,s_i}\,,
\]
and thus deduce \eqref{zetak parametrizzano} from \eqref{ziozio cantante2}.
\end{proof}

\subsection{Blow-ups of $(\Lambda,r_0)$-minimizing clusters}\label{section blowups clusters} If $\E$ is a $N$-cluster in $A$ and $x\in A$, then the blow-up of $\E$ at $x$ at scale $r>0$ is the $N$-cluster $\E_{x,r}$ in $(A-x)/r$
%(note that $(A-x)/r$ eventually contains any given compact set in $\R^n$ as $r\to 0^+$)
defined by setting
\[
\E_{x,r}(h)=\frac{\E(h)-x}r\,,\qquad 1\le h\le N\,.
\]
We set
\[
\theta(\pa\E,x,r)=\frac{P(\E;B_{x,r})}{r^{n-1}}=\theta(\pa\E_{x,r},0,1)\,,\qquad \theta(\pa\E,x)=\lim_{r\to 0^+}\theta(\pa\E,x,r)\,,
\]
provided this last limit exists. By a classical argument based on comparison with cones (see, for example \cite[Theorem 28.4]{maggiBOOK}), one sees that if $\E$ is a $(\Lambda,r_0)$-minimizing $N$-cluster in $A$, $x\in A\cap\pa\E$, and $r_*\in(0,r_0)$ is such that $\om_n\,r_*^n<\min\{|\E(h)\cap A|:1\le h\le N\}$, then
  \begin{equation}
  \label{monotonicity}
  \theta(\pa\E,x,r)\,e^{(n-1)\,\om_n\,\Lambda\,r}\qquad \mbox{ is increasing on $(0,r_*)$}\,,
  \end{equation}
so that $\theta(\pa\E,x)$ is defined for every $x\in A\cap\pa\E$. Moreover, the same argument shows that if $\Lambda=0$ and $\theta(\pa\E,x,r)$ is constant on $r\in(0,r_*)$, then $B_{x,r_*}\cap \pa\E$ is a cone with vertex at $x$. Now let us say that a $M$-cluster $\K$ in $\R^n$ is a {\it cone-like minimizing cluster} if $\K(i)$ is an open cone with vertex at the origin for each $i=1,...,M$, $|\K(0)|=|\R^n\setminus\bigcup_{i=1}^M\K(i)|=0$, and
\begin{equation}
  \label{puppa min}
  P(\K;B_R)\le P(\F;B_R)\,,
\end{equation}
whenever $R>0$ and $\F$ is an $M$-cluster in $\R^n$ with $\F(i)\Delta\K(i)\cc B_R$ for every $i=1,...,M$. Moreover, given a $N$-cluster $\E$ in $A$ and an injective map $\s:\{1,...,M\}\to\{0,...,N\}$, let us denote by $\s(\E)$ the $M$-cluster in $A$ defined by setting
\[
\s\E(i)=\E(\s(i))\,,\qquad i=1,...,M\,.
\]

\begin{theorem}[Tangent cone-like minimizing clusters]\label{theorem tangent clusters}
  If $\E$ is a $(\Lambda,r_0)$-minimizing $N$-cluster in $A$, $x\in A\cap\pa\E$, and $s_j\to 0$ as $j\to\infty$, then there exist a subsequence $\{s_j'\}_{j\in\N}$ and a cone-like minimizing $M$-cluster $\K$ (with  $2\le M\le N$) such that $\theta(\pa\E,x)=\theta(\pa\K,0)$ and
  \begin{eqnarray}\label{blowup limit}
  \lim_{j\to\infty}\d_{B_R}\big(\s\E_{x,s_j'},\K\big)=0\,,\qquad\forall R>0\,.
  \end{eqnarray}
  for some injective map $\s:\{1,\dots,M\}\to\{0,\dots,N\}$. (Note that given $R>0$ one has $B_R\cc(A-x)/r$ as soon as $r$ is small enough.) Moreover, $x\in\S_A(\E)$ if and only if $0\in\S(\K)$.
  %such that
%  \begin{equation}
%  \label{local convergence def}
%  \lim_{j\to\infty}\sum_{i=1}^M\,\Big|\Big(\K(i)\Delta \E_{x,r_{k(j)}}(\s(i))\Big)\cap B_R\Big|=0\,,\qquad\forall R>0\,.
%  \end{equation}
\end{theorem}

\begin{proof}
  Once again this follows by a classical argument. We refer to \cite[Theorem 28.6]{maggiBOOK} for a proof in the case of $(\Lambda,r_0)$-minimizing sets.
\end{proof}

We conclude this section with a technical lemma, which is the starting point in showing (under the situation described in \eqref{ipotesi ripetuta mille volte}) the Hausdorff convergence of $\S(\E_k)$ to $\S(\E)$ when $n=2,3$.

\begin{lemma}\label{lemma tecnico utile}
  Let $n\ge2$ be fixed. Either $\hd_{A'}(\S(\E_k),\S(\E))\to 0$ as $k\to\infty$ whenever \eqref{ipotesi ripetuta mille volte} holds and $A'\cc A$, or there exist a cone-like minimizing $M$-cluster $\K$ in $\R^n$ and a sequence $\{\F_j\}_{j\in\N}$ of $(\de_j,\de_j^{-1})$-minimizing $M$-clusters $\F_j$ in $B_2$ with
  \[
  0\in\S(\K)\,,\qquad \S_{B_2}(\F_j)=\emptyset\qquad\forall j\in\N\,,\qquad \lim_{j\to\infty}\max\big\{\de_j,\d_{B_2}(\F_j,\K)\big\}=0\,.
  \]
\end{lemma}

\begin{proof}
  Let us assume that for some $\E_k$, $\E$ and $A$ as in \eqref{ipotesi ripetuta mille volte} there exists $A'\subset A$ such that $\limsup_{k\to\infty}\hd_{A'}(\S(\E_k),\S(\E))>0$. By Theorem \ref{thm boundary hausdorff} and by \eqref{inclusions bordi} in Theorem \ref{thm boundary hausdorff} and up to extracting subsequences, we may directly assume the existence of $x \in\S_{A'}(\E)$ and $\e>0$ such that $B_{x ,\e}\cc A$,
  \begin{equation}
  \label{niente}
  B_{x ,\e}\cap\S_A(\E_k)=\emptyset\qquad\forall  k\in\N\,,
  \end{equation}
  and such that $x_k\to x $ for some $x_k\in A\cap\pa\E_k$. In particular, up to discarding finitely many values of $k$, we may assume that $x_k\in A'\cap\pa^*\E_k$ for every $k$, and finally, up to translating $\E_k$, that $x_k=x $ for every $k$. Summarizing, we have $\E_k$ and $\E$ as in \eqref{ipotesi ripetuta mille volte} such that there exists
  \[
  x \in\S_A(\E)\cap\bigcap_{k\in\N} \pa^*\E_k\,.
  \]
  By Theorem \ref{theorem tangent clusters} we can find a cone-like minimizing $M$-cluster $\K$ in $\R^n$ with $\theta(\pa\E,x)=\theta(\pa\K,0)$ (so that $0\in\S(\K)$ by $x\in\S_A(\E)$), an injective map $\s:\{1,...,M\}\to\{0,...,N\}$, and a sequence $s_j\to 0^+$ as $j\to\infty$ such that \eqref{blowup limit} holds (with $s_j$ directly in place of $s_j'$). Correspondingly, we consider $\{k(j)\}_{j\in\N}$ such that
  \begin{equation}
  \label{veloce}
  \d_{B_{x ,\e}}(\E_{k(j)},\E)= o(s_j^n)\qquad\mbox{as $j\to\infty$}\,,
 %\qquad\max_{0\le i<\ell\le N}\hd_{A'}\Big(\pa\E_{k(j)}(i)\cap\pa\E_{k(j)}(\ell),\pa\E(i)\cap\pa\E(\ell)\Big)=o(s_j)\,,
  \end{equation}
  and finally define $(s_j\,\Lambda,r_0/s_j)$-minimizing $M$-clusters $\F_j$ in $(A-x)/s_j$ by setting
  \[
  \F_j(i)=\frac{\E_{k(j)}(\s(i))-x}{s_j}\,,\qquad\mbox{that is}\qquad \F_j=\s\big( (\E_{k(j)})_{x,s_j}\big)\,.
  \]
  By \eqref{blowup limit} and \eqref{veloce} for every fixed $R>0$ one has $\d_{B_R}(\F_j,\K)\to 0$, while \eqref{niente} implies that $\S_{B_R}(\F_j)=\emptyset$ provided $j$ is large enough.
\end{proof}

\section{Improved convergence for planar clusters}\label{section planar stuff} In this section we finally prove Theorem \ref{thm main planar}. First, in section \ref{section planar clusters}, we address the structure of $(\Lambda,r_0)$-minimizing clusters in $\R^2$, and deduce from this structure result and Lemma \ref{lemma tecnico utile} the Hausdorff convergence of singular sets. Next, in section \ref{section proof of main theorem}, and specifically in Theorem \ref{thm key planar}, we complete the preparations needed to exploit Theorem \ref{thm main diffeo k} in the proof of Theorem \ref{thm main planar}. This last argument is then presented at the end of the section.

\subsection{$(\Lambda,r_0)$-minimizing clusters in $\R^2$}\label{section planar clusters} In view of Theorem \ref{theorem tangent clusters}, the starting point in the analysis of almost-minimizing clusters near their singular sets is the classification of cone-like minimizing clusters. Such a classification is currently known only in $\R^2$ and $\R^3$. Referring to \cite{CiLeMaIC2} for the latter case, we work from now on in $\R^2$. Let us denote by $\Y_2$ the cone-like minimizing $3$-cluster in $\R^2$ defined by
\begin{equation}
  \label{steiner partition}
  \Y_2(i)=\Big\{(t\cos\theta,t\sin\theta):\ t>0\,,(i-1)\,\frac{2\pi}3<\theta< i\,\frac{2\pi}3\Big\}\,,\qquad i=1,2,3\,.
\end{equation}
Up to rotations around the origin, $\Y_2$ is the only cone-like minimizing cluster in $\R^2$ (other than the one defined by a pair of complementary half-planes, of course); see, for example, \cite[Proposition 30.9]{maggiBOOK}. As a consequence, by Theorem \ref{theorem tangent clusters} one has that if $\E$ is a $(\Lambda,r_0)$-minimizing cluster in $A\subset\R^2$, then $\pa^*\E=\{x\in A\cap\pa\E:\theta(\pa\E,x)=2\}$ and
\begin{equation}\label{sigma E n=2}
    \S_A(\E)=\Big\{x\in A\cap\pa\E:\theta(\pa\E,x)=\theta(\Y_2,0)=3\Big\}\,.
\end{equation}
We now localize Definition \ref{cluster cka}, and then, in Theorem \ref{thm planar clusters}, describe the structure of planar almost-minimizing clusters.

\begin{definition}\label{def cluster cka local}
  {\rm Let $\E$ be a cluster in $A\subset\R^2$ open. One says that $\E$ is a {\it $C^{k,\a}$-cluster in $A$} if there exist at most countable families $\{\g_i\}_{i\in I}$ of connected $C^{k,\a}$-curves with boundary relatively closed in $A$, and $\{p_j\}_{j\in J}$ of points of $A$, which are both locally finite in $A$ (that is, given $A'\cc A$ we have $\g_i\cap A'\ne\emptyset$ and $p_j\in A'$ only for finitely many $i\in I$ and $j\in J$), and such that
\begin{equation}
  \label{struttura E planar A'}
  \begin{split}
  A\cap\pa\E=\bigcup_{i\in I}\g_i\,,\qquad \pa^*\E=\bigcup_{i\in I}\INT(\g_i)\,,
  \\
  \S_{A}(\E)=A\cap\bigcup_{i\in I}\bd(\g_i)=A\cap\bigcup_{j\in J}\{p_j\}\,.
  \end{split}
\end{equation}
}
\end{definition}

\begin{theorem}\label{thm planar clusters}
If $\E$ is a $(\Lambda,r_0)$-minimizing cluster in $A\subset\R^2$, then $\E$ is a $C^{1,1}$-cluster in $A$. Moreover, each $\g_i$ has distributional curvature bounded by $\Lambda$ and each $p_j$ is a common boundary point of exactly three different curves from $\{\g_i\}_{i\in I}$ which form three 120 degrees angles at $p_j$. Finally, $\diam(\g_i)\ge1/2\Lambda$ for every $i\in I$ such that $\g_i\cc A$ and $\bd(\g_i)=\emptyset$. (In particular, if $\Lambda=0$, then $\bd(\g_i)\ne\emptyset$ for every $i\in I$.)
\end{theorem}

\begin{proof}
 By exploiting the argument of \cite[Theorem 30.7]{maggiBOOK} (which addresses the case of planar isoperimetric clusters, but actually uses only a minimality condition of the form \eqref{almost-minimizer cluster}, and that can be easily localized to a given open set) we just need to prove that the curves $\g_i$ have distributional curvature bounded by $\Lambda$ and the diameter lower bound when $\g_i\cc A$ with $\bd(\g_i)=\emptyset$. By Corollary \ref{corollary mean curvature} we have that
 \begin{equation}
   \label{fico davvero}
    \int_{\pa\E}\Div_{\pa\E}T\,d\H^1=\int_{\pa\E}T\cdot{\bf H}_{\pa\E}\,d\H^1\,,\qquad\forall T\in C^1_c(A;\R^2)\,,
 \end{equation}
 where $|{\bf H}_{\pa\E}|\le \Lambda$. In particular,
 \begin{equation}
   \label{fico davvero x}
    \int_{\g_i}\Div_{\g_i}T\,d\H^1=\int_{\g_i}T\cdot{\bf H}_{\pa\E}\,d\H^1\,,
 \end{equation}
 for every $T\in C^1_c(A';\R^2)$ such that $\spt\,T\cap\pa\E=\spt\,T\cap\INT(\g_i)$. Since $|{\bf H}_{\pa\E}|\le \Lambda$ this proves that each $A'\cap\g_i$ has distributional mean curvature bounded by $\Lambda$. If, in addition, $\g_i\cc A'\cc A$ and $\bd(\g_i)=\emptyset$, then we can test \eqref{fico davvero x} with $T(x)=\zeta(x)(x-x_0)$ where $x_0\in \R^2$ is such that $\g_i\subset B_{x_0,2\,\diam(\g_i)}$ and $\zeta\in C^1_c(A')$ with $\zeta=1$ on $\g_i$ and $\spt\zeta\cap\pa\E=\spt\zeta\cap\g_i$, to find that $\H^1(\g_i)\le 2\Lambda\,\diam(\g_i)\,\H^1(\g_i)$, as required.
\end{proof}

\begin{remark}[Topology of boundaries of planar $(\Lambda,r_0)$-minimizing clusters]\label{remark M}
  {\rm If $\E$ is a bounded $(\Lambda,r_0)$-minimizing cluster in $\R^2$, then Theorem \ref{thm planar clusters} implies the existence of {\it finite} families of closed connected $C^{1,1}$-curves with boundary $\{\g_i\}_{i\in I}$ (whose distributional curvature is bounded by $\Lambda$) and of finitely many points $\{p_j\}_{j\in J}$ such that each $p_j$ is the common end-point of three different curves from $\{\g_i\}_{i\in I}$, which form three 120 degrees angles at $p_j$. Moreover, \eqref{struttura E planar A'} takes the form
  \begin{equation}
  \label{struttura E planar}
  \pa\E=\bigcup_{i\in I}\g_i\,,\qquad \pa^*\E=\bigcup_{i\in I}\INT(\g_i)\,,\qquad
  \S(\E)=\bigcup_{i\in I}\bd(\g_i)=\bigcup_{j\in J}\{p_j\}\,.
  \end{equation}
  Let $I''$ denotes the set of those $i\in I$ such that $\g_i$ is diffeomorphic to $[0,1]$ (so that $\g_i$ is diffeomorphic to $\SS^1$ for every $i\in I'=I\setminus I''$, this will be the notation used in the proof of Theorem \ref{thm key planar}). For each $i\in I''$, $\g_i$ has exactly two end-points, both belonging to $\S(\E)$, and for every $x\in\S(\E)$ there exist three curves from $\{\g_i\}_{i\in I''}$ sharing $x$ as a common end-point: therefore we find that
  \[
  \#(I'')=\frac32\,\H^0(\S(\E))\,.
  \]}
  %If we now apply Euler's formula to the planar graph having the singular points from $\S(\E)$ as its vertexes, the curves $\{\g_i\}_{i\in I}$ as its edges, and the $N'$ connected components of the chambers of $\E$ as its faces, then we find $2=\H^0(\S(\E))-\#(I)+N'$. Since $\#(I)=(3/2)\,\H^0(\S(\E))$, we have proved \eqref{topologia 2d}.}
\end{remark}

\begin{remark}\label{remark topologia 2d E0}
  {\rm With the notation of the previous remark, we claim that $I''=I$ whenever $\E$ is a planar isoperimetric cluster (that is, $\E$ is a minimizer in \eqref{partitioning problem} with $N\ge 2$ and $n=2$; notice that $\E$ is necessarily bounded). Indeed, arguing by contradiction, let us assume there exists $i\in I$ such that $\g_i$ is $C^1$-diffeomorphic to $\SS^1$. Since $\g_i\cap\S(\E)=\emptyset$, the constant curvature condition on interfaces of $\E$ implies that $\g_i$ is, in fact, a circle. Moreover, since $N\ge 2$, we must have $\#(I)\ge 2$. Since $\#(I)\ge 2$, we can translate $\g_i$ along a suitable direction until it intersects for the first time $\pa\E\setminus\g_i$ at some point $x$. Denoting by $\E'$ the resulting cluster, we have that $P(\E')=P(\E)$ and $\vol(\E')=\vol(\E)$, so that $\E'$ is a minimizing cluster in $\R^2$. The fact that, in a neighborhood of $x$, $\pa\E'$ is the union of two tangent circular arcs, leads to a contradiction with Theorem \ref{thm planar clusters} (applied to $\E'$).}
%
%  \begin{equation}
%    \label{topologia 2d}
%      \#(I)=3(N'-2)\,,\qquad \H^0(\S(\E))=2(N'-2)\,,
%  \end{equation}
%  where $N'$ is the sum of the numbers of connected components of the chambers $\E(h)$ of $\E$ over $h=0,\dots,N$. (In particular, $N'=N+1$ if every chamber, including the exterior chamber, is connected.)
%  }
\end{remark}

We now upgrade \eqref{inclusions bordi} to the full Hausdorff convergence of singular sets.

\begin{theorem}[Hausdorff convergence of singular sets]
  \label{thm singular hausdorff}
  If $\{\E_k\}_{k\in\N}$ is a sequence of $(\Lambda,r_0)$-minimizing clusters in $A\subset\R^2$ with $\d_A(\E_k,\E)\to 0$ as $k\to\infty$, then
  \[
  \lim_{k\to\infty}\hd_{A'}(\S_A(\E_k),\S_A(\E))=0\qquad\forall A'\cc A\,.
  \]
\end{theorem}

\begin{proof}
 We argue by contradiction. In this way, by Lemma \ref{lemma tecnico utile} there exists a sequence $\{\F_j\}_{j\in\N}$ of $(\de_j,\de_j^{-1})$-minimizing $M$-clusters in $B_2\subset\R^2$ such that
 \begin{equation}
     \label{ilre}
 \S_{B_2}(\F_j)=\emptyset\qquad\forall j\in\N\,,\qquad \lim_{j\to\infty}\max\big\{\de_j,\d_{B_2}(\F_j,\Y_2)\big\}=0\,,
 \end{equation}
 where $\Y_2$ is defined as in \eqref{steiner partition}. By Theorem \ref{thm boundary hausdorff},
\begin{equation}
  \label{puppa a}
  \lim_{j\to\infty}\max_{1\le i<\ell\le 3}\hd_B\Big(\pa\F_j(i)\cap \pa\F_j(\ell),\pa\Y_2(i)\cap\pa\Y_2(\ell)\Big)=0\,,
\end{equation}
while, by Theorem \ref{thm normal representation part one}, for every $\de$ small enough one can find $\{\psi_j\}_{j\ge j_0}\subset C^1(B\cap[\pa\Y_2]_\de)$ such that (on taking into account that $I_{2\de}(\Sigma(\Y_2))=B_{2\de}$)
\begin{equation}
  \label{puppa b}
  \pa\F_j\cap (B\setminus B_{2\de})\subset(\Id+\psi_j\nu)(B\cap [\pa\Y_2]_\de)\,,\qquad\forall j\ge j_0\,,
\end{equation}
where $\nu$ denotes a continuous normal vector field to $\pa^*\Y_2$.
%Let us set $\ell_{i,j}=\pa\Y_2(i)\cap\pa \Y_2(j)$ for $1\le i<j\le 3$, so that $\ell_{1,2}\cap\ell_{1,3}\cap\ell_{2,3}=\{0\}=\S(\Y_2)$.
%We claim that, arguing by contradiction, for every $\de>0$ small enough we can find a $(\de,1/\de)$-minimizing $N$-cluster $\F_\de$ in $\R^2$ with
%  \begin{gather}
%  \label{puppa1}
%  B\cap\S(\F_\de)=\emptyset\,,
%  \\\label{puppa2}
%  B\cap\pa\F_\de(\s(i))\cap\pa\F_\de(\s(j))\subset I_\de(\ell_{i,j})\,,\qquad\forall 1\le i<j\le 3\,,
%  \\\label{puppa3}
%  \Big(B\cap\pa\F_\de\Big)\setminus B_{2\de}\subset (\Id+\psi_\de\nu)(B\cap [\pa\Y_2]_\de)\,,
%%  \\\label{puppa4}
%%  |B\cap\F_\de(h)|=0\,,\qquad\mbox{if $h\ne\s(i)$, $i=1,2,3$}\,,
%  \end{gather}
%  where $\psi_\de\in C^1([\pa\Y_2]_\de)$ with $\|\psi_\de\|_{C^1(B\cap[\pa\Y_2]_\de)}\to 0$ as $\de\to 0$, and $\s:\{1,2,3\}\to\{0,...,N\}$ is injective.
 % Let us first notice that this claim leads to a contradiction. Since $\F_\de$ is a $(\de,1/\de)$-minimizing $N$-cluster in $\R^2$,
 By Theorem \ref{thm planar clusters} there exists a finite family of connected $C^{1,1}$-curves with boundary $\{\g_i\}_{i\in I}$, relatively closed in $B$, such that $B\cap\pa\F_j=B\cap\bigcup_{i\in I}\g_i$ and $\S_B(\F_j)=\bigcup_{i\in I}B\cap\bd(\g_i)$, so that, by \eqref{ilre}, $B\cap\bd(\g_i)=\emptyset$ for every $i\in I$.
% \replace{If $X=B\cap\pa\F_j(1)\cap\pa\F_j(2)$, then by \eqref{puppa a}
% \begin{equation}
%   \label{amen 2}
%   X\cap B\subset I_\de(\pa\Y_2(1)\cap\pa \Y_2(2))\,.
% \end{equation}
% By \eqref{puppa b} and \eqref{amen 2}, $X\cap (B\setminus B_{2\de})$ is diffeomorphic to $(0,1]$; in particular, there exists $i_0\in I$ such that $X\cap(B\setminus B_{2\de})=\g_{i_0}\cap (B\setminus B_{2\de})$, so that, in conclusion,
% \begin{equation}\label{split}
%   \begin{split}
%    \mbox{$\g_{i_0}\cap(B\setminus B_{2\de})$ is homeomorphic to $(0,1]$}\,,
%    \\
%    \mbox{$\g_{i_0}\cap B\subset I_\de(\pa\Y_2(1)\cap\pa \Y_2(2))$}\,.
%   \end{split}
% \end{equation}
% By \eqref{split} and since $\g_{i_0}$ is relatively closed in $B$ and connected, we find $B_{2\de}\cap\bd(\g_{i_0})\ne\emptyset$, against the fact that $B\cap\bd(\g_i)=\emptyset$ for every $i\in I$}
Let $\g_{i\ell}$ denote the connected curve in $\pa\F_{j}$ that contains $(\Id+\psi_j\nu)(B\cap [\pa\Y_2(i)\cap \pa\Y_{2}(\ell)]_\de)$, for $1\leq i<\ell \leq 3$. By \eqref{puppa b} we notice that
 \begin{equation}\label{solotrecurveGP}
 \pa\F_{j} \cap (B\setminus B_{2\de}) = \bigcup_{1\le i<\ell\le 3} \g_{i\ell}\cap (B\setminus B_{2\de})
 \end{equation}
 while by \eqref{puppa a} we get $\g_{i\ell}\cap B\subset I_\de(\pa\Y_2(i)\cap\pa \Y_2(\ell))$ for all $1\le i<\ell\le 3$. By combining this last fact with \eqref{solotrecurveGP}, we deduce that $\bd(\g_{i\ell})\cap B_{2\de}\neq\emptyset$, against the fact that $B\cap \bd(\g_{i})=\emptyset$ for every $i\in I$.
 \end{proof}

\subsection{Proof of the improved convergence theorem for planar clusters}\label{section proof of main theorem}  We now prove Theorem \ref{thm main planar}. We start by setting some notation. Let us consider $\Lambda$, $r_0$, $\E$ and $\E_k$ as in Theorem \ref{thm main planar}. Since $\pa\E$ is bounded, by Theorem \ref{thm boundary hausdorff} also $\pa\E_k$ is bounded, and thus according to \eqref{struttura E planar} there exist {\it finite} families of $C^{2,1}$-curves $\{\g_i\}_{i\in I}$ and $C^{1,1}$-curves $\{\g_i^k\}_{i\in I_k}$, and finite families of points $\{p_j\}_{j\in J}$ and $\{p_j^k\}_{j\in J_k}$ such that
\begin{eqnarray*}
    \pa\E=\bigcup_{i\in I}\g_i\,,\qquad \pa^*\E=\bigcup_{i\in I}\INT(\g_i)\,,\qquad
  \S(\E)=\bigcup_{i\in I}\bd(\g_i)=\bigcup_{j\in J}\{p_j\}\,,
  \\
    \pa\E_k=\bigcup_{i\in I_k}\g_i^k\,,\qquad \pa^*\E_k=\bigcup_{i\in I_k}\INT(\g_i^k)\,,\qquad
  \S(\E_k)=\bigcup_{i\in I_k}\bd(\g_i^k)=\bigcup_{j\in J_k}\{p_j^k\}\,.
\end{eqnarray*}
Moreover, each $p_j$ is the common boundary point of exactly three curves from $\{\g_i\}_{i\in I}$, and an analogous assertion holds for $p_j^k$ and $\{\g_i^k\}_{i\in I}$.

\begin{theorem}
  \label{thm key planar} Under the assumptions of Theorem \ref{thm main planar}, there exist positive constants $\rho_0$ and $L$, depending on $\Lambda$ and $\E$ only, such that the following properties hold:
  \begin{enumerate}
    \item[(i)] there exists $k_0\in\N$ such that for each $k\ge k_0$, up to a relabeling of $I_k$ and $J_k$, one has $I=I_k$ and $J=J_k$, with $\bd(\g_i)\ne\emptyset$ if and only if $\bd(\g_i^k)\ne\emptyset$ for every $i\in I$, and
        \begin{equation}
          \label{heybeckner}
        \lim_{k\to\infty}|p_j^k-p_j|+\hd(\g_i^k,\g_i)=0\,,\qquad\forall i\in I,j\in J\,;
        \end{equation}
        moreover
        \begin{equation}\label{ext by foliation planar case k}
        \|\g_i^k\|_{C^{1,1}}\le L\,,\qquad\forall i\in I_k\,,
        \end{equation}
        and if $p_j\in\bd(\g_i)$ then $p_j^k\in\bd(\g_i^k)$ with
        \begin{equation}
          \label{heybeckner4}
          \lim_{k\to\infty}|\nu_{\g_i}^{co}(p_j)-\nu_{\g_i^k}^{co}(p_j^k)|=0\,;
        \end{equation}
%        moreover, for every $i\in I$ there exists an extension by foliation $(\e_i^k,d_i^k)$ of $\g_i^k$ with
%  \begin{equation}\label{ext by foliation planar case k}
%    \max\Big\{\frac1{\e_i^k},\|d_i^k\|_{C^{1,1}(\R^2)}\Big\}\le C_0\,;
%  \end{equation}
%    \item[(ii)] for every $i\in I$ and $k\ge k_0$, if $\bd(\g_i)=\{p_j,p_{j'}\}$, $\bd(\g_i^k)=\{p_j^k,p_{j'}^k\}$, and if $\tau_i^k\in C^{0,\a}(\g_i^k;\SS^1)$ and $\tau_i\in C^{1,1}(\g_i;\SS^1)$ denote tangent unit vector fields to $\g_i^k$ and $\g_i$ respectively,
%        %, defining $p_j^k$ and $p_j$ as start-points on $\g_i^k$ and $\g_i$ respectively
%        then, up to a change of orientation,
%        \begin{equation}
%          \label{heybeckner4}
%                  \lim_{k\to\infty}|\tau_i(p_j)-\tau_i^k(p_j^k)|+|\tau_i(p_{j'})-\tau_i^k(p_{j'}^k)|=0\,;
%        \end{equation}
    \item[(ii)] for every $\rho< \rho_0$ there exist $k(\rho)\in\N$ and $\{\psi_k\}_{k\ge k(\rho)}\subset C^{1,1}([\pa\E]_\rho)$ such that
        \begin{equation}
        \label{zetak parametrizzano PLUS}
        [\pa\E_k]_{3\rho} \subset (\Id+\psi_k\nu)([\pa\E]_\rho)\subset\pa^*\E_k\,,\qquad\forall k\ge k(\rho)\,,
        \end{equation}
        where $\nu$ is a $C^{1,1}$-normal unit vector field to $\pa^*\E$ and
        \begin{equation}
        \label{zetak tendono a zero PLUS}
        \lim_{k\to\infty}\|\psi_k\|_{C^1([\pa\E]_\rho)}=0\,,\qquad \sup_{k\ge k(\rho)}\|\psi_k\|_{C^{1,1}([\pa\E]_\rho)}\le L\,.
        \end{equation}
  \end{enumerate}
\end{theorem}

\begin{proof}
{\it Step one}: We prove statement (ii). By Theorem \ref{thm normal representation part one} (applied with $A=\R^2$ and $A'$ equal to an open ball such that $\E(h)\cc A'$ for every $h=1,...,N$) there exist $\rho_0,L>0$ such that for every $\rho\le\rho_0$ one can find $k(\rho)\in\N$, $\e(\rho)>0$ and $\{\psi_k\}_{k\ge k(\rho)}\subset C^{1,1}([\pa\E]_\rho)$ such that \eqref{zetak tendono a zero PLUS} holds, with
\begin{equation}
  \label{Neps0}
  \pa\E_k\setminus I_{2\rho}(\S(\E)) \subset (\Id+\psi_k\nu)([\pa\E]_\rho)\subset\pa^*\E_k\,,
\end{equation}
\begin{equation}
  \label{Neps}
  N_{\e(\rho)}([\pa\E]_\rho)\cap\pa\E_k=(\Id+\psi_k\nu)([\pa\E]_\rho)\,,
\end{equation}
for all $k\ge k(\rho)$. In turn, by Theorem \ref{thm singular hausdorff} (applied with $A=\R^2$ and $A'$ as above), we have $\hd(\S(\E_k),\S(\E))\to 0$ as $k\to\infty$. Hence, up to increasing the value of $k(\rho)$ we find $\S(\E)\subset I_\rho(\S(\E_k))$ for  $k\ge k(\rho)$, and thus $[\pa\E_k]_{3\rho}=\pa\E_k\setminus I_{3\rho}(\S(\E_k))\subset \pa\E_k\setminus I_{2\rho}(\S(\E))$. Thus \eqref{zetak parametrizzano PLUS} follows from \eqref{Neps0}.

\medskip

\noindent {\it Step two}: We prove (i) up to \eqref{heybeckner}. We first note that if $x_k,y_k\in\S(\E_k)$ with $x_k\ne y_k$ and $x_k\to x$ and $y_k\to y$ (so that $x,y\in\S(\E)$ by $\hd(\S(\E_k),\S(\E))\to 0$), then it must be $x\ne y$. Indeed, if $x=y$, then $\e_k=|x_k-y_k|\to 0$, and the sequence of clusters $\F_k=(\E_k-x_k)/\e_k$ would converge (up to subsequences and in the sense explained in Theorem \ref{theorem tangent clusters}) to a Steiner partition of $\R^2$. At the same time, this Steiner partition should have a singular point at unit distance from the origin, arising as the limit of of some subsequence of $(y_k-x_k)/\e_k$. This contradiction proves our remark, which coupled with the Hausdorff convergence of $\S(\E_k)$ to $\S(\E)$ allows us to assume without loss of generality that $J=J_k$ with
\begin{equation}
  \label{pkell va a pell}
  \lim_{k\to\infty}|p^k_j-p_j|=0\,,\qquad\forall j\in J\,.
\end{equation}
Let now $I'$ and $I''$ be the sets of those $i\in I$ such that $\g_i$ is homeomorphic, respectively, either to $\SS^1$ or to $[0,1]$, and similarly define $I'_k$ and $I''_k$ starting from $I_k$. By intersecting with $N_{\e(\rho_0)}([\g_i]_{\rho_0})$ in \eqref{Neps} and by directly assuming that $\|\psi_k\|_{C^0([\pa\E]_{\rho_0})}\le\e(\rho_0)\le\rho_0$ for every $k\ge k_0=k(\rho_0)$ we find
\[
N_{\e(\rho_0)}([\g_i]_{\rho_0})\cap\pa\E_k=(\Id+\psi_k\nu)([\g_i]_{\rho_0})\,,\qquad\forall i\in I\,,k\ge k_0\,.
\]
In particular, by exploiting the connectedness of the curves $\{\g_k^i\}_{i\in I_k}$, one defines for every $k\ge k_0$ a map $\s_k:I\to I_k$ in such a way that
\begin{eqnarray*}
&&(\Id+\psi_k\nu)([\g_i]_{\rho_0})\subset\g_{\s_k(i)}^k\,,
\\
&&(\Id+\psi_k\nu)([\g_i]_{\rho_0})\cap\g_{i'}^k=\emptyset\,,\qquad\forall i\in I\,,\forall i'\in I_k\setminus\{\s_k(i)\}\,;
\end{eqnarray*}
hence,
\begin{equation}
  \label{injective maps}
  (\Id+\psi_k\nu)([\g_i]_{\rho_0})=N_{\e(\rho_0)}([\g_i]_{\rho_0})\cap\pa\E_k=N_{\e(\rho_0)}([\g_i]_{\rho_0})\cap\g_{\s_k(i)}^k\,,\qquad\forall k\ge k_0\,, i\in I\,.
\end{equation}
To complete the proof of \eqref{heybeckner} it will suffice to show that
\begin{equation}
  \label{heybeckner what0}
  \mbox{$\s_k$ is a bijection with $\s_k(I')=I'_k$ and $\s_k(I'')=I''_k$}\,,
\end{equation}
\begin{equation}
  \label{heybeckner what}
  \lim_{k\to\infty}\hd(\g_i,\g_{\s_k(i)}^k)=0\,,\qquad\forall i\in I\,.
\end{equation}
We start by choosing $\eta>0$ such that
\begin{equation}
  \label{tutti per i cazzi loro}
  I_\eta(\g_i)\cap I_\eta(\g_{i'})=\emptyset\,,\qquad\forall i,i'\in I'\,.
\end{equation}
If $i\in I'$, then $[\g_i]_\rho=\g_i$ and $N_{\e(\rho)}(\g_i)=I_{\e(\rho)}(\g_i)$ for every $\rho>0$, so that \eqref{injective maps} gives
\begin{equation}
  \label{injective maps I'}
  (\Id+\psi_k\nu)(\g_i)=I_{\e(\rho_0)}(\g_i)\cap\pa\E_k=I_{\e(\rho_0)}(\g_i)\cap\g_{\s_k(i)}^k\,,\qquad\forall k\ge k_0\,, i\in I'\,.
\end{equation}
Since $(\Id+\psi_k\nu)(\g_i)$ is homeomorphic to $\SS^1$ and is contained in $\g_{\s_k(i)}^k$, by connectedness of $\g_{\s_k(i)}^k$ we conclude that $\s_k(i)\in I'_k$ with
\begin{eqnarray}
  \label{injective maps I' x}
  &&(\Id+\psi_k\nu)(\g_i)=I_{\e(\rho_0)}(\g_i)\cap\pa\E_k=\g_{\s_k(i)}^k\,,
  \\
  \label{injective maps I' xx}
  &&\hd(\g_i,\g_{\s_k(i)}^k)\le\|\psi_k\|_{C^0([\pa\E]_{\rho_0})}\le\rho_0\,,\qquad\forall k\ge k_0\,, i\in I'\,.
\end{eqnarray}
By combining \eqref{tutti per i cazzi loro}, \eqref{injective maps I' x} and \eqref{zetak tendono a zero PLUS} and up to requiring $\rho_0\le\eta$ we conclude that
\begin{eqnarray}\label{sabbath1}
\mbox{\eqref{heybeckner what} holds for every $i\in I'$, $\s_k(I')\subset I'_k$, $\s_k$ is injective on $I'$}\,.
\end{eqnarray}
Before showing that $\s_{k}(I') = I_{k}'$, we first prove that
\begin{equation}
  \label{sabbath2}
  \mbox{$\s_k$ is a bijection between $I''$ and $I''_k$}\,.
\end{equation}
To this end, we shall first need to prove \eqref{higgs} and \eqref{higgs bordo} below. In order to formulate \eqref{higgs} we introduce the following notation: given $j\in J$, let us denote by $a_j(1)$, $a_j(2)$, and $a_j(3)$ the three distinct elements in $I''$ such that the curves $\{\g_{a_j(\ell)}\}_{\ell=1}^3$ share $p_j$ as a common boundary point (as described in Theorem \ref{thm planar clusters}), and let $\{a_j^k(\ell)\}_{\ell=1}^3\subset I_k''$ be defined analogously starting from $p_j^k$. We claim that, up to permutations in the index $\ell\in\{1,2,3\}$, one has
\begin{equation}
  \label{higgs}
a_j^k(\ell)=\s_k(a_j(\ell))\,,\qquad\forall j\in J\,,k\ge k(\rho)\,,\ell\in\{1,2,3\}\,.
\end{equation}
Indeed, by Theorem \ref{thm planar clusters}, up to decrease the value of $\eta>0$, we find that, for every $j\in J$,
\begin{equation}
  \label{def eta}
  \pa\E\cap B_{p_j,\eta}=\bigcup_{\ell=1}^3\g_{a_j(\ell)}\cap B_{p_j,\eta}\,,
\qquad
\{p_j\}=\S(\E)\cap B_{p_j,\eta}=\bigcup_{\ell=1}^3\bd(\g_{a_j(\ell)})\cap B_{p_j,\eta}\,.
\end{equation}
Since $\e(\rho_0)\le\rho_0$, up to further decreasing $\rho_0$ depending on $\eta$, we can entail by Theorem \ref{thm boundary hausdorff} and \eqref{pkell va a pell} that
\begin{equation}
  \label{def eta k}
  \pa\E_k\subset I_{\e(\rho_0)}(\pa\E)\,,\qquad \S(\E_k)\cap B_{p_j,\eta}=\{p_j^k\}\subset B_{p_j,\e(\rho_0)}\,,\qquad\forall j\in J\,,k\ge k_0\,.
\end{equation}
By \eqref{def eta} and provided $\rho_0$ is small enough,
\begin{eqnarray*}
  I_{\e(\rho_0)}(\pa\E)\cap B_{p_j,\eta}&=&\bigcup_{\ell=1}^3I_{\e(\rho_0)}(\g_{a_j(\ell)})\cap B_{p_j,\eta}
  \\
  &\subset& B_{p_j,2\rho_0}\cup\bigcup_{\ell=1}^3\Big(N_{\e(\rho_0)}([\g_{a_j(\ell)}]_{\rho_0})\cap B_{p_j,\eta}\Big)\,,\qquad\forall j\in J\,.
\end{eqnarray*}
By $\pa\E_k\subset I_{\e(\rho_0)}(\pa\E)$ and by \eqref{injective maps} one thus finds
\begin{eqnarray}\label{def eta 10}
\pa\E_k\cap B_{p_j,\eta}\subset\Big(\pa\E_k\cap B_{p_j,2\rho_0}\Big)\cup\bigcup_{\ell=1}^3\Big(
\g_{\s_k(a_j(\ell))}^k\cap B_{p_j,\eta}\Big)\,.
\end{eqnarray}
Let now $\om$ be the connected component of $\g_{a_j^k(1)}^k\cap\cl(B_{p_j,\eta})$ which contains $p_j^k$. In this way, $\om$ is a connected $C^{1,1}$-curve with boundary, homeomorphic to $[0,1]$, with $p_j^k\in\bd(\om)\cap B_{p_j,\eta}$. It cannot be $\om\cc B_{p_j,\eta}$, because otherwise it would be $\om=\g^k_{a_j^k(1)}\cc B_{p_j,\eta}$, and thus $\S(\E_k)\cap B_{p_j,\eta}\setminus\{p_j^k\}\ne\emptyset$, against \eqref{def eta k}. Hence $\om\cap\pa B_{p_j,\eta}\ne\emptyset$. At the same time, by \eqref{def eta 10},
\[
\om\cap B_{p_j,\eta}\subset\Big(\om\cap B_{p_j,2\rho_0}\Big)\cup\bigcup_{\ell=1}^3\Big(
\om\cap\g_{\s_k(a_j(\ell))}^k\cap B_{p_j,\eta}\Big)\,,
\]
and since $\om$ is connected with $\om\cap\pa B_{p_j,\eta}\ne\emptyset$, it must be $\om\cap\g_{\s_k(a_j(\ell))}^k\ne\emptyset$ for some $\ell\in\{1,2,3\}$, thus $\g_{a_j^k(1)}^k\cap\g_{\s_k(a_j(\ell))}^k\ne\emptyset$. Up to relabeling $\ell\in\{1,2,3\}$, we have thus proved that
\[
\g^k_{a_j^k(\ell)}\cap\g^k_{\s_k(a_j(\ell))}\ne\emptyset\,,\qquad\forall j\in J\,,k\ge k_0\,,\ell\in\{1,2,3\}\,,
\]
from which \eqref{higgs} follows by the connectedness of the curves $\{\g_i^k\}_{i\in I}$. Having proved \eqref{higgs}, we now introduce the notation needed to formulate \eqref{higgs bordo}: given $i\in I''$, let $b_i(1)$ and $b_i(2)$ denote the two distinct elements of $J$ such that $\bd(\g_i)=\{b_i(1),b_i(2)\}$, and define similarly $b_i^k(m)$ ($m=1,2$) for each $i\in I''_k$. Then, up to permutations in the index $m\in\{1,2\}$,
\begin{equation}
  \label{higgs bordo}
  b_{\s_k(i)}^k(m)=b_i(m)\,,\qquad\forall i\in I''\,,k\ge k_0\,,m=1,2\,.
\end{equation}
Indeed, if $i\in I''$ then $i=a_{b_i(1)}(\ell)$ for some $\ell\in\{1,2,3\}$, therefore, by \eqref{higgs},
\[
\s_k(i)=\s_k(a_{b_i(1)}(\ell))=a_{b_i(1)}^k(\ell)\,,
\]
that is,
\[
p_{b_i(1)}^k\in\bd(\g_{\s_k(i)})=\{p_{b_{\s_k(i)}^k(1)},p_{b_{\s_k(i)}^k(2)}\}\,,\qquad\mbox{thus}\qquad
b_i(1)\in\{b_{\s_k(i)}^k(1),b_{\s_k(i)}^k(2)\}\,,
\]
as required. With \eqref{higgs} and \eqref{higgs bordo} in force, we now prove \eqref{sabbath2}. The fact that $\s_k(I'')\subset I''_k$ is immediate from $I''=\{a_j(\ell):j\in J\,,\ell\in\{1,2,3\}\}$ and \eqref{higgs}. If now $i,i'\in I''$ are such that $\s_k(i)=\s_k(i')$ then by \eqref{higgs bordo}
\begin{eqnarray*}
\{j\in J:p_j\in\bd(\g_i)\}&=&\{b_i(m)\}_{m=1}^2=\{b_{\s_k(i)}^k(m)\}_{m=1}^2=\{b_{\s_k(i')}^k(m)\}_{m=1}^2
\\
&=&\{b_{i'}(m)\}_{m=1}^2=\{j\in J:p_j\in\bd(\g_{i'})\}\,,
\end{eqnarray*}
so that $\bd(\g_i)=\bd(\g_{i'})$, and thus $i=i'$; this proves that $\s_k$ is injective on $I''$. Finally, by Remark \ref{remark M}, it must be $\#\,I''=(3/2)\,\#\,J=(3/2)\,\#\,J_k=\#\,I_k''$, so that $\s_k$ is actually a bijection between $I''$ and $I''_k$, and \eqref{sabbath2} is proved.

Let us now show that
\begin{equation}
  \label{bd xx 1}
\lim_{k\to\infty}\hd(\g_i,\g_{\s_k(i)}^k)=0\,,\qquad\forall i\in I''\,.
\end{equation}
We first notice that, by \eqref{higgs bordo},
\[
\{j\in J:p_j^k\in\bd(\g_{\s_k(i)}^k)\}=\{b_{\s_k(i)}^k(m)\}_{m=1}^2=\{b_i(m)\}_{m=1}^2
=\{j\in J:p_j\in\bd(\g_i)\}\,,
\]
so that \eqref{pkell va a pell} gives
\begin{equation}
  \label{bd xx}
  \lim_{k\to\infty}\hd(\bd(\g_i),\bd(\g_{\s_k(i)}^k))=0\,,\qquad\forall i\in I''\,.
\end{equation}
Next, if $i\in I''$, then by \eqref{injective maps I'} one has $\g_{\s_k(i)}^k\cap I_{\e(\rho_0)}(\g_{i'})=\emptyset$ for every $i'\in I'$, while \eqref{injective maps} gives $\g_{\s_k(i)}^k\cap N_{\e(\rho_0)}([\g_{i'}]_{\rho_0})=\emptyset$ for every $i'\in I''\setminus\{i\}$; since $\pa\E_k\subset I_{\e(\rho_0)}(\pa\E)$ for $k\ge k_0$, we thus find
\[
\g_{\s_k(i)}^k\subset I_{2\,\rho_0}(\g_i)\cup\bigcup_{i'\in I''}I_{2\rho_0}(\bd(\g_{i'}))\,,\qquad\forall i\in I''\,,k\ge k_0\,.
\]
Since $I_{2\,\rho_0}(\g_i)$ is disjoint from $\bigcup_{i'\in I''}I_{2\rho_0}(\bd(\g_{i'}))$ thanks to \eqref{def eta}, we conclude that $\g_{\s_k(i)}^k\subset I_{2\,\rho_0}(\g_i)$ for every $i\in I''$ and $k\ge k_0$. At the same time, by \eqref{injective maps}, \eqref{zetak tendono a zero PLUS}, and \eqref{bd xx}
\[
[\g_i]_{\rho_0}\subset I_{\rho_0}(\g_{\s_k(i)}^k)\,,\qquad I_{\rho_0}(\bd(\g_i))\subset I_{2\rho_0}(\g_{\s_k(i)}^k)\,,\qquad\forall i\in I''\,,k\ge k_0\,,
\]
that is, $\g_i\subset I_{2\rho_0}(\g_{\s_k(i)}^k)$ for every $i\in I''$ and $k\ge k_0$. We have thus proved \eqref{bd xx 1}.

In order to complete the proof of \eqref{heybeckner what0} and \eqref{heybeckner what} we are thus left to show that $\s_k(I')=I'_k$. We argue by contradiction, and assume the existence of $i_*\in I'_k\setminus \s_k(I')$. Since $I_{\e(\rho_0)}(\g_i)\cap\pa\E_k=\g_{\s_k(i)}^k$ for every $i\in I'$ (recall \eqref{injective maps I' x}), by connectedness we deduce that
\begin{equation}
  \label{zapata 1}
  \g_{i_*}^k\cap \bigcup_{i\in I'}I_{\e(\rho_0)}(\g_i)=\emptyset\,.
\end{equation}
Since $N_{\e(\rho_0)}([\g_i]_{\rho_0})\cap\pa\E_k=N_{\e(\rho_0)}([\g_i]_{\rho_0})\cap\g_{\s_k(i)}^k$ for every $i\in I$ (recall \eqref{injective maps}), if $\g_{i_*}^k\cap N_{\e(\rho_0)}([\g_i]_{\rho_0})\ne\emptyset$, then, by connectedness of $\g_{\s_k(i)}^k$, one finds $i_*=\s_k(i)\in\s_k(I)$, a contradiction: hence,
\begin{equation}
  \label{zapata2}
 \g_{i_*}^k\cap \bigcup_{i\in I''}N_{\e(\rho_0)}([\g_i]_{{\rho_0}})=\emptyset\,.
\end{equation}
Since $\pa\E_k\subset I_{\e(\rho_0)}(\pa\E)=\bigcup_{i\in I}I_{\e(\rho_0)}(\g_i)$, by \eqref{zapata 1} and \eqref{zapata2} we find
\[
\g_{i_*}^k\subset
\bigcup_{i\in I''}I_{\e(\rho_0)}(\g_i)\setminus\bigcup_{i\in I''}N_{\e(\rho_0)}([\g_i]_{\rho_{0}})\subset \bigcup_{j\in J}B_{p_j,\eta}\,,
\]
and since the balls $\{B_{p_j,\eta}\}_{j\in J}$ are disjoint by \eqref{def eta}, we conclude that for every $i_*\in I'_k\setminus\s_k(I')$ there exists a unique $j\in J$ such that $\g_{i_*}^k\subset B_{p_j,2\rho_0}$; however, by Theorem \ref{thm planar clusters},
\[
\frac1{\Lambda}\le\diam(\g_{i_*}^k)<2\rho_0\,,
\]
which leads to a contradiction if $\rho_0$ is sufficiently small.
\medskip

\noindent {\it Step three}: We prove \eqref{ext by foliation planar case k}. We directly consider the case when $\bd(\g_{i}^{k})\neq \emptyset$, and omit the (analogous) details for the case $\bd(\g_{i}^{k})= \emptyset$. Let us set $\ell_i^k=\H^1(\g_i^k)$, consider $\a_i^k\in C^{1,1}([0,\ell_i^k];\R^2)$ to be an arc-length parametrization of $\g_i^k$, and define unit normal vector fields $\nu_i^k\in C^{0,1}(\g_k;\SS^1)$ by setting $\nu_i^k(\a_i^k(t))=(\a_i^k)'(t)^\perp$, with the convention that $v^\perp=(v_2,-v_1)$ for every $v=(v_1,v_2)\in\R^2$. According to Definition \ref{basta S C1alpha}, we just need to show that, up to further increasing the value of $L$
\begin{equation}
  \label{bellastoriafrate}
  |\nu_i^k(x)\cdot(y-x)|\le L\,|x-y|^2\,,\qquad|\nu_i^k(x)-\nu_i^k(y)|\le L\,|x-y|\,,\qquad\forall x,y\in\g_i^k\,.
\end{equation}
Indeed, if $x,y\in\g_i^k$ with $s,t\in[0,\ell_i^k]$ such that $x=\a_i^k(s)$ and $y=\a_i^k(t)$, then, by  $\Lip((\a_i^k)')\le\Lambda$,
\[
|\nu_i^k(x)\cdot(y-x)|\le C\,|s-t|^2\,,\qquad|\nu_i^k(x)-\nu_i^k(y)|\le C\,|s-t|\,;
\]
we are thus left to show that
\begin{equation}
  \label{maguardate}
  |s-t|\le C\,|\a_i^k(s)-\a_i^k(t)|\,,\qquad \forall s,t\in[0,\ell_i^k]\,.
\end{equation}
If $|s-t|\le1/\Lambda$, then \eqref{maguardate} follows with $C\ge 2$ by noticing that
\[
|\a_i^k(s)-\a_i^k(t)|=\Big|\int_t^s(\a_i^k)'(r)\,dr\Big|\ge|t-s|-\Lambda\,\frac{|t-s|^2}2\,,
\]
once again thanks to $\Lip((\a_i^k)')\le\Lambda$. If $\g_i[x,y]$ denote the arc of $\g_i$ with end-points $x,y\in\g_i$, then by compactness
\[
\min_{i\in I}\,\inf\Big\{|x-y|:x,y\in \g_i\,,\H^1(\g_i[x,y])\ge\frac1{2\Lambda}\Big\}\ge c\,,
\]
where $c>0$ depends on $\E$ and $\Lambda$ only. Since for every $i\in I$ we have $\hd(\g_i^k,\g_i)\to 0$ as $k\to\infty$, we can thus entail
\[
\min_{i\in I}\,\inf\Big\{|x-y|:x=\a_i^k(s)\,, y=\a_i^k(t)\,,|s-t|\ge\frac1{\Lambda}\Big\}\ge \frac{c}2\,,
\]
so that \eqref{maguardate} holds on $|s-t|>1/\Lambda$ provided $C\ge 2\Lambda/c$. This completes the proof of \eqref{bellastoriafrate}, thus of \eqref{ext by foliation planar case k}.

\medskip

\noindent {\it Step four}: We prove \eqref{heybeckner4}. Let us fix $j\in J$, and consider $p_j^k\in\S(\E_k)$ and $i_1,i_2,i_3\in I$ such that $\{p_j^k\}=\bd(\g_{i_1}^k)\cap \bd(\g_{i_2}^k)\cap \bd(\g_{i_3}^k)$. Since each $\g_i^k$ is a compact connected $C^{1,1}$-curve with distributional curvature bounded by $\Lambda$ one finds that, for every $i=i_1,i_2,i_3$,
\begin{equation}
  \label{heybeckner3}
\lim_{r\to0^+}\sup_{k\in\N}\hd_B\Big(\frac{\g_i^k-p_j^k}r,\R_+\,[\tau_i^k(p_j^k)]\Big)=0\,,
\end{equation}
where we have set $\R_+[\tau]=\{t\,\tau:t\ge0\}$ for every $\tau\in\SS^1$, and we have set
\[
\tau_i=-\nu_{\g_i}^{co}\,,\qquad \tau_i^k=-\nu_{\g_i^k}^{co}\,,
\]
for the sake of brevity. We thus find
\begin{eqnarray}\label{heybeckner2}
  \hd_B\Big(\R_+\,[\tau_i(p_j)],\R_+\,[\tau_i^k(p_j^k)]\Big)&\le&
  \sup_{k\in\N}\hd_B\Big(\frac{\g_i^k-p_j^k}r,\R_+\,[\tau_i^k(p_j^k)]\Big)
  \\\nonumber
  &&+\hd_B\Big(\frac{\g_i-p_j}r,\R_+\,[\tau_i(p_j)]\Big)+2\,\frac{\hd(\g_i^k,\g_i+(p_j^k-p_j))}r\,,
\end{eqnarray}
where we have also used the fact that, for $k$ large enough,
\begin{equation}\label{heyjoe3}
\hd_B\Big(\frac{\g_i^k-p_j^k}r,\frac{\g_i-p_j}r\Big)\le 2\,\frac{\hd(\g_i^k,\g_i+(p_j^k-p_j))}r\,.
\end{equation}
At this point we can choose a sequence $r_{k}\to 0^+$, such that the right-hand side of \eqref{heyjoe3} with $r = r_{k}$ is infinitesimal as $k\to\infty$. By also exploiting \eqref{heybeckner} and \eqref{heybeckner3}, this gives $\hd_B(\R_+\,[\tau_i(p_j)],\R_+\,[\tau_i^k(p_j^k)])\to 0$ as $k\to\infty$, that is \eqref{heybeckner4}.
\end{proof}

\begin{proof}
  [Proof of Theorem \ref{thm main planar}] Let $\E$ be a $C^{2,1}$-cluster in $\R^2$, $\{\E_k\}_{k\in\N}$ be a sequence of $(\Lambda,r_0)$-minimizing clusters such that $\d(\E_k,\E)\to 0$ as $k\to\infty$, and let $L$, $\rho_0$ and, for each $\rho\le\rho_0$, $k(\rho)\in\N$, be the constants given by Theorem \ref{thm key planar}. Denote by $\mu_0$ and $C_0$ the smallest and the largest constants, respectively, associated by Theorem \ref{thm main diffeo k} to some $\g_i$ such that $\bd(\g_i)\ne\emptyset$. In this way, $\mu_0$ and $C_0$ depend on $\Lambda$ and $\E$ only. Up to further decreasing the value of $\mu_0$, we can also assume that $\mu_0^2\le\rho_0$. Given $\mu<\mu_0$, we now want to find $k(\mu)\in\N$ such that for every $k\ge k(\mu)$ there exists a $C^{1,1}$-diffeomorphism $f_k$ between $\pa\E$ and $\pa\E_k$ with
  \begin{eqnarray}\label{reso1}
  \|f_k\|_{C^{1,1}(\pa\E)}&\le& C_0\,,
  \\\label{reso2}
  \lim_{k\to\infty}\|f_k-\Id\|_{C^1(\pa\E)}&=&0\,,
  \\\label{reso3}
  \|\ttau_{\E}(f_k-\Id)\|_{C^1(\pa^*\E)}&\le&\frac{C_0}\mu\,\|f_k-\Id\|_{C^0(\S(\E))}\,,
  \\\label{reso4}
  \ttau_{\E}(f_k-\Id)&=&0\,,\quad\mbox{on $[\pa\E]_\mu$}\,.
  \end{eqnarray}
  {\it Let us fix $i\in I$ such that $\bd(\g_i)\ne\emptyset$.} Since $\mu^{2}<\mu_0^2\le\rho_0$, Theorem \ref{thm key planar} ensures that $\{\g_i^k\}_{k\ge k_0}$ satisfies the assumptions (i) and (ii) of Theorem \ref{thm main diffeo k}. By Theorem \ref{thm main diffeo k} for every $k\ge k(\mu)$ one finds a $C^{1,1}$-diffeomorphism $f_i^k$ between $\g_i$ and $\g_i^k$ with $f_i^k(p_j)=p_j^k$, $f_i^k(p_{j'})=p_{j'}^k$ ($j$ and $j'$ as in statement (ii) of Theorem \ref{thm key planar}) and
  \begin{eqnarray}
    \label{diffeo curve4 z}
    \|f_i^k\|_{C^{1,1}(\g_i)}&\le&C_0\,,
    \\
    \label{diffeo curve5 z}
    \|(f_i^k-\Id)\cdot\tau_i\|_{C^1(\g_i)}&\le& \frac{C_0}\mu\,\|f_i^k-\Id\|_{C^0(\bd(\g_i))}\,,
    \\
    \label{diffeo curve1 z}
    (f_i^k-\Id)\cdot\tau_i&=&0\qquad\mbox{on $[\g_i]_{\mu}$}\,;
%  \end{eqnarray}
%  moreover, if $k\ge k_*(\rho)$, then
%  \begin{eqnarray}
%    \label{diffeo curve3 z}
%    \sup_{k\ge k_*(\rho)} \|f_i^k-\Id\|_{C^1(\g_i)}\le C_0\,\frac{\rho}\mu\,,
%  \end{eqnarray}
%  which of course implies
%  \begin{eqnarray}
\\
    \label{diffeo curve3 zz}
    \lim_{k\to\infty}\|f_i^k-\Id\|_{C^1(\g_i)}&=&0\,.
  \end{eqnarray}
  {\it Let us now fix $i\in I$ such that $\bd(\g_i)=\emptyset$.} Up to further decreasing $\mu_0$, $\g_i$ is a connected component of $[\pa\E]_\mu$, and thus by statement (ii) in Theorem \ref{thm key planar}, $\{\psi_k\}_{k\ge k(\rho)}\subset C^{1,1}([\pa\E_0]_\rho)$ are such that
  \begin{equation}
  \label{resonance}
  \g_i^k=(\Id+\psi_k\nu)(\g_i)\,,\qquad \lim_{k\to\infty}\|\psi_k\|_{C^1(\g_i)}=0\,,\qquad \sup_{k\in\N}\|\psi_k\|_{C^{1,1}(\g_i)}\le C_0\,.
  \end{equation}
  We set $f_i^k=\Id+\psi_k\,\nu$ for every $i\in I$ such that $\bd(\g_i)=\emptyset$, and finally define $f_k(x)=f_i^k(x)$ for $x\in \g_i$. The resulting map $f_k$ defines a $C^{1,1}$-diffeomorphism between $\pa\E$ and $\pa\E_k$ (see Definition \ref{def 1}) with \eqref{reso1}--\eqref{reso4} in force.
\end{proof}

\section{Some applications of the improved convergence theorem}\label{section applications} We now prove Theorem \ref{thm finitely many types} and Theorem \ref{thm potenziale}. To this end, let us notice that if $\{\E_k\}_{k\in\N}$ is a sequence of planar isoperimetric clusters with $\sup_{k\in\N}P(\E_k)<\infty$, then there exist $x_k\in\R^2$ and a planar $N$-cluster $\E_0$ such that, up to extracting subsequences, $x_k+\E_k\to\E_0$. This is a simple consequence of (i) the inequality $2\,\diam(E)\le P(E)$, which holds for every indecomposable set of finite perimeter $E$ in $\R^2$ (this, of course, after the normalization \eqref{spt normalization}); (ii) the fact that $\R^2\setminus\E(0)$ is indecomposable whenever $\E$ is an isoperimetric cluster (as it can be easily inferred by arguing as in Remark \ref{remark topologia 2d E0}).

\begin{proof}
  [Proof of Theorem \ref{thm finitely many types}] We argue by contradiction, and assume that there exists a sequence $\{\E_k\}_{k\in\N}$ of isoperimetric $N$-clusters with $\vol(\E_k)\to m_0$ such that $[\E_k]_\approx\ne[\E_j]_\approx$ whenever $k\ne j$. Let $\phi:\R^N_+\to(0,\infty)$ denote the infimum in \eqref{partitioning problem}, then it is easily seen that $\phi$ is locally bounded. In particular, $\sup_{k\in\N}P(\E_k)<\infty$, and thus there exists a $N$-cluster $\E_0$ and $x_k\in\R^2$ such that, up to extracting subsequences, $x_k+\E_k\to\E_0$ as $k\to\infty$. We claim that, for $k$ large enough, $x_k+\E_k$ is a $(\Lambda,r_0)$-minimizing cluster in $\R^2$, where $\Lambda$ and $r_0$ are independent from $k$. To this end, let $\e_0$, $r_0$, and $C_0$ be the constants associated with $\E_0$ by Theorem \ref{thm volume fixing} and let $k_0$ be such that $\d(x_k+\E_k,\E_0)<\e_0$ for $k\ge k_0$. Given $\F$ with $\F(h)\Delta (x_k+\E_k(h))\cc B_{x,r_0}$ for $h=1,...,N$, by applying Theorem \ref{thm volume fixing} with $\E=x_k+\E_k$ we find $\F'_k$ such that
  \[
  \vol(\F'_k)=\vol(x_k+\E_k)=\vol(\E_k)\,,\qquad P(\F'_k)\le P(\F)+C_0\,\d(x_k+\E_k,\F)\,.
  \]
  so that, by the isoperimetric property of $\E_k$, $P(x_k+\E_k)\le P(\F'_k)\le P(\F)+C_0\,\d(x_k+\E_k,\F)$. Thus $x_k+\E_k$ is a $(\Lambda,r_0)$-minimizing cluster in $\R^2$ for $k$ large enough. By Theorem \ref{thm boundary hausdorff} we infer that $\E$ is also a $(\Lambda,r_0)$-minimizing cluster in $\R^2$, and thus conclude by Theorem \ref{thm main planar} that $x_k+\E_k\approx\E$ for $k$ large enough. Since $x_k+\E_k\approx\E_k$, we have found a contradiction to $[\E_k]_\approx\ne[\E_j]_\approx$ for $k\ne j$.
\end{proof}

\begin{proof}
  [Proof of Theorem \ref{thm potenziale}] {\it Step one}: We first prove that, if $\E$ is a minimizer in \eqref{partitioning problem with potential delta} with $\de\in(0,\de_0)$ and $|m-m_0|<\de_0$, then $\E\approx\E_0$. We argue by contradiction, and consider a sequence $\{\E_k\}_{k\in\N}$ of minimizers in
  \begin{equation}
  \label{partitioning problem with potential k}
  \lambda_k=\inf\Big\{P(\E)+\de_k\sum_{h=1}^N\int_{\E(h)}g(x)\,dx:\vol(\E)=m_k\Big\}\,,\qquad k\in\N\,,
  \end{equation}
  where $\de_k\to 0$ and $m_k\to m_0$ as $k\to\infty$, and $[\E_k]_\approx\ne[\E_0]_\approx$ for every $k\in\N$. Let $\{\F_k\}_{k\in\N}$ be a sequence of isoperimetric clusters with $\vol(\F_k)=m_k$. Since $m_k\to m_0$ implies $\sup_{k\in\N}P(\F_k)<\infty$, by the argument presented at the beginning of this section there exists $R>0$ such that, up to translations,  $\F_k(h)\cc B_R$ for every $h=1,...,N$ and $k\in\N$. By comparing $\E_k$ and $\F_k$ in \eqref{partitioning problem with potential k} we find
  \begin{equation}
    \label{viacolvento}
      P(\E_k)+\de_k\sum_{h=1}^N\int_{\E_k(h)}g\le P(\F_k)+\de_k\sum_{h=1}^N\int_{\F_k(h)}g\le P(\F_k)+\de_k\,|m_k|\,\sup_{B_R}g
  \end{equation}
  and since $P(\F_k)\le P(\E_k)$ we thus find that for every $r>0$
  \[
  \inf_{\R^2\setminus B_r}g\,\sum_{h=1}^N|\E_k(h)\setminus B_r|\le|m_k|\,\sup_{B_R}g\,.
  \]
  By $g(x)\to\infty$ as $|x|\to\infty$, we conclude that
  \begin{equation}
    \label{viacolvento2}
      \lim_{r\to\infty}\sup_{k\in\N}\sum_{h=1}^N|\E_k(h)\setminus B_r|=0\,.
  \end{equation}
  Since \eqref{viacolvento} also implies $\sup_{k\in\N}P(\E_k)<\infty$, by \eqref{viacolvento2} we conclude that up to extracting subsequences, $\d(\E_k,\E)\to 0$ as $k\to\infty$, where $\E$ is a planar cluster with $\vol(\E)=m_0$. In particular, recalling that $\E_0$ denotes the unique isoperimetric cluster with $\vol(\E_0)=m_0$, we have
  \begin{equation}
    \label{viacolvento3}
      P(\E_0)\le P(\E)\le\liminf_{k\to\infty}P(\E_k)\,.
  \end{equation}
  Now, by \cite[Theorem 29.14]{maggiBOOK} there exist positive constants $\e$, $\eta$ and $C$, a smooth map $\Phi\in C^1((-\eta,\eta)^N\times\R^2;\R^2)$, and a disjoint family of balls $\{B_{z_i,\e}\}_{i=1}^M$ such that, for every $v\in (-\eta,\eta)^N$, the $N$-cluster defined by $\E_{0,v}(h):=\Phi(v,\E_0(h))$, $h=1,...,N$, satisfies
  \begin{eqnarray*}
    \E_{0,v}(h)\Delta\E_0(h)\cc A=\bigcup_{i=1}^MB_{z_i,\e}\,,
    \quad
    P(\E_{0,v})\le P(\E_0)+C\,|v|\,,
    \quad
    \vol(\E_{0,v})=\vol(\E_0)+v\,.
  \end{eqnarray*}
  For $k$ large, $v_k=\vol(\E_k)-\vol(\E_0)\in (-\eta,\eta)^N$, so that $\vol(\E_{0,v_k})=m_k$ and, by $g\ge0$
  \begin{eqnarray*}
  P(\E_k)+\de_k\sum_{h=1}^N\int_{\E_k(h)}g \le P(\E_{0,v_k})+\de_k\,\sum_{h=1}^N\int_{\E_{0,v_k}(h)}\,g\le P(\E_0)+C\,|v_k|+\de_k\,\sup_{B_{2S}}\,g
  \end{eqnarray*}
  where $S$ is such that $\bigcup_{h=1}^N\E_0(h)\cup A\cc B_S$. Letting $k\to\infty$ we find that
  \[
  \limsup_{k\to\infty}P(\E_k)\le P(\E_0)\,,
  \]
  so that, by \eqref{viacolvento3}, $P(\E)=P(\E_0)$. Since $\vol(\E)=m_0$, we find $\E\approx\E_0$ (through an isometry), and we may thus assume, without loss of generality, that $\E=\E_0$. By arguing as in the previous proof (with some minor modification because of the presence of the potential), we see that, for $k$ large enough, $\E_k$ is a $(\Lambda,r_0)$-minimizer with $\Lambda$ and $r_0$ uniform in $k$. Since $\d(\E_k,\E_0)\to 0$ as $k\to\infty$, by Theorem \ref{thm main planar} we find that $\E_k\approx\E_0$ for $k$ large enough, a contradiction.

  \bigskip

  \noindent {\it Step two}: The argument of step one can be easily adapted to show the existence of minimizers in \eqref{partitioning problem with potential delta}, together with the existence of $R_0$ (depending on $\E_0$, $\de_0$ and $g$ only) such that $\E(h)\subset B_{R_{0}}$ for every $h=1,...,N$ and every minimizer $\E$. In particular, there exists $C_0$ depending on $g$ and $R_0$ only such that
  \begin{equation}
    \label{tensorsss}
      P(\E)\le P(\F)+C_0\,\de\,\d(\E,\F)\,,
  \end{equation}
  whenever $\vol(\E)=\vol(\F)$ and $\F(h)\subset B_{2R_{0}}$. Let us fix $x_1,x_2\in\E(h,k)$, $T_i\in C^1_c(B_{x_i,r};\R^n)$ ($i=1,2$) with $|\E(j)\cap B_{x_i,r}|=0$ if $i\ne h,k$ and $r<|x_1-x_2|$, and with
  \[
  \int_{\pa^*\E(h)}T_i\cdot\nu_{\E(h)}\,d\H^{n-1}=\eta_i>0\,,\qquad \sup_{\R^n}|T_i|\le 1\,.
  \]
  By a standard argument we can construct a one-parameter family of diffeomorphisms $f_t$ with $f_t(x)=x+t\,(T_1(x)-(\eta_1/\eta_2)T_2(x))+O(t^2)$ such that $\vol(f_t(\E))=\vol(\E)$. For $t$ small enough $\F=f_t(\E)$ is admissible in \eqref{tensorsss}, with
  \[
  \d(\E,f_t(\E))\le 2|f_t(\E(h))\Delta\E(h)|\le 2\,P(\E(h);B_{x_1,r}\cup B_{x_2,r})\,|t|\,,
  \]
  by Lemma \ref{lemma distanza L1}. Since
  \[
  P(f_t(\E))=P(\E)+t\,\int_{\pa^*\E(h)}(T_1-(\eta_1/\eta_2)T_2)\cdot\nu_{\E(h)}\,H_{\E(h,k)}+O(t^2)\,,
  \]
  and $P(\E(h);B_{x_1,s}\cup B_{x_2,s})=\om_{n-1}\,s^{n-1}(1+O(1))$ as $s\to 0^+$, by \eqref{tensorsss} we conclude that
  \[
  \int_{\pa^*\E(h)}(T_1-(\eta_1/\eta_2)T_2)\cdot\nu_{\E(h)}\,H_{\E(h,k)}\le 2\,C_0\,\de\,\om_{n-1}r^{n-1}(1+O(1))\,.
  \]
  Let now $T_i=T_i^j\to 1_{B_{x_i,r}}\nu_{\E(h)}$ in $L^1(\H^1\llcorner\pa\E(h))$ as $j\to\infty$, so that
  \[
  \int_{B_{x_1,r}\cap\pa^*\E(h)}H_{\E(h,k)}- \frac{\eta_1}{\eta_2}\int_{B_{x_1,r}\cap\pa^*\E(h)}H_{\E(h,k)}
  \le 2\,C_0\,\de\,\om_{n-1}r^{n-1}(1+O(1))\,.
  \]
  By the mean value theorem, as $r\to 0^+$, we find that $H_{\E(h,k)}(x_1)-H_{\E(h,k)}(x_2)\le 2\,C_0\,\de$, that is,
  \[
   \max_{0\le h<k\le N}\|H_{\E(h,k)}-H_{h,k}^\de\|_{C^0(H_{\E(h,k)})}\le C\,\de\,,
  \]
  for some $H_{h,k}^\de\in\R$. At the same time, by arguing for example as in \cite[Lemma 3.7(ii)]{CicaleseLeonardi}, one see that $H_{\E(h,k)}$ has to converge in the sense of distributions to $H_{\E_0(h,k)}$ as $\de\to 0^+$, and thus prove \eqref{convessita}.
\end{proof}

\appendix

\section{Proof of Theorem \ref{thm inverse fnct uniform}}\label{section implicit}

\begin{proof}[Proof of Theorem \ref{thm inverse fnct uniform}] In the following, we denote by $C$ a generic constant depending on $n$, $k$, $\a$, and $L$ only. Let us set $\l_{\rm min},\l_{\rm max}:S_0\to\R$ as $\l_{\rm min}(x)=\inf\{|\nabla^{S_0} f(x)v|:v\in T_{x}S_{0},|v|=1\}$ and $\l_{\rm max}(x)=\|\nabla^{S_0} f(x)\|$. By \eqref{inverse function1} we find that
  \[
  \frac1{L}\le J^{S_0}f(x)\le\l_{\rm min}(x)\,\l_{{\rm max}}(x)^{k-1}\le \l_{\rm min}(x)\,L^{k-1}\,,
  \]
  that is $\l_{\rm min}(x)\ge L^{-k}$ for every $x\in S_0$. In particular, by also using \eqref{vai ganzo} we find that
\begin{equation}\label{xysotto}
|\nabla^{S_0} f(x)(y-x)|=|\nabla^{S_0} f(x)\pi^{S_0}_x(y-x)|\geq \frac{|\pi^{S_0}_x(y-x)|}{L^k}\geq \frac{|y-x|}{2L^{k}}\,,\qquad\forall y\in B_{x,1/L}\cap S_0\,.
\end{equation}
We now assume $\e_0<1/L$ and fix $y\in B_{x,\e_0}\cap S_0\setminus\{x\}$. Since $\textstyle\dist_{S_{0}}(x,y)>0$  we can find $\gamma\in C^1([0,1];S_{0})$ such that $\gamma(0)=x$, $\gamma(1)=y$ and
 \begin{equation}\label{distinmezzo}
 \textstyle\dist_{S_{0}}(x,y)\le\displaystyle\int_0^1|\dot\g(t)|\,dt\le \textstyle2\,\dist_{S_0}(x,y)\,.
 \end{equation}
By \eqref{xysotto},
  \begin{align}
 \nonumber |f(y)-f(x)|&
 %=\left|\int_0^1\nabla^{S_0} f(\gamma(t))\dot\gamma(t)\, dt\right|
 =\Big|\nabla^{S_0} f(x)(y-x)-\int_0^1(\nabla^{S_0} f(\gamma(t))-\nabla^{S_0} f(x))\dot\gamma(t)\,dt\Big|
 \\
 \nonumber  &\geq\frac{|y-x|}{2L^{k}}-\int_0^1\|\nabla^{S_0} f(\gamma(t))-\nabla^{S_0} f(x)\||\dot\gamma(t)|\, dt\,.
 \end{align}
 By \eqref{inverse function1}, \eqref{distinmezzo}, and \eqref{geo-euc}
 \begin{eqnarray*}
   \int_0^1\|\nabla^{S_0} f(\gamma(t))-\nabla^{S_0} f(x)\||\dot\gamma(t)|\, dt&\le&
    L\,\int_0^1\,|x-\g(t)|^\a\,|\dot\gamma(t)|\, dt
   \\
   &\le& L\,\int_0^1\,\Big(\int_0^t|\dot\g(s)|\,ds\Big)^\a\,|\dot\gamma(t)|\, dt
   \\
   &\le& L\,2^{1+\a}\,\textstyle\dist_{S_{0}}(x,y)^{1+\a}\le L\,(2L)^{1+\a}\,|x-y|^{1+\a}\,.
 \end{eqnarray*}
 We thus conclude (up to further decreasing the value of $\e_0$) that if $x\in S_0$ and $y\in B_{x,\e_0}\cap S_0$, then
 \begin{equation}\label{iniettiva-1}
   |f(x)-f(y)|\ge|y-x|\left(\frac{1}{2L^{k}}-L\,(2L)^{1+\alpha}\e_{0}^{\alpha}\right)\geq\frac{|y-x|}{4L^{k}}\,.
 \end{equation}
 This shows that $f$ is injective on $B_{x,\e_0}\cap S_0$ for every $x\in S_0$. If now \eqref{inverse function2} is in force with $\rho_0\le\e_0/4$, then by $\diam(S_0)\le L$ one finds that for every $x,y\in S_0$ with $|x-y|\ge\e_0$
  \[
  |f(x)-f(y)|\ge|x-y|-|f(x)-x|-|f(y)-y|\ge\e_0-2\,\rho_0\ge\frac{\e_0}2\ge\frac{\e_0}{2\,L}\,|x-y|\,,
  \]
  so that, in conclusion, $f$ is injective on $S_0$ with
  \begin{equation}
    \label{chefreddo0}
    |f^{-1}(p_1)-f^{-1}(p_2)|\le C\,|p_1-p_2|\,,\qquad\forall p_1,p_2\in S=f(S_0)\,.
  \end{equation}
  We are thus left to prove that
  \begin{equation}
    \label{chefreddo1}
    \|\nabla^{S}f^{-1}(p_1)-\nabla^{S}f^{-1}(p_2)\|\le C\,|p_1-p_2|^\a\,,\qquad\forall p_1,p_2\in S\,.
  \end{equation}
  Indeed, by \eqref{vai ganzo1}, \eqref{inverse function1} and \eqref{chefreddo0} we can entail
  \begin{equation}
    \label{vai ganzo2}
    \|\pi^S_p-\pi^S_q\|\le C\,|p-q|^\a\,,\qquad\forall p,q\in S\,.
  \end{equation}
  Let us now fix $p_1,p_2\in S$ and set
  \[
  M_i=\nabla^Sf^{-1}(p_i)\,,\quad \pi_i=\pi^S_{p_i}\,,\quad x_i=f^{-1}(p_i)\,,\quad N_i=\nabla^{S_0}f(x_i)\,,\quad \pi_i^0=\pi_{x_i}^{S_0}\,.
  \]
  By exploiting the relations
  \begin{gather}
  \pi_i^0M_i=M_i=M_i\pi_i\,,\qquad  \pi_iN_i=N_i=N_i\pi_i^0\,,
   \\
  N_1M_1\pi_1=\pi_1\,,\qquad N_2M_2\pi_2=\pi_2\,,\qquad M_1N_1\pi_1^0=\pi_1^0\,,\qquad M_2N_2\pi_2^0=\pi_2^0\,,
  \end{gather}
  one finds that
  \begin{eqnarray*}
    &&M_1(N_2-N_1)M_2+M_2(N_2-N_1)M_1
    \\
    &=&M_1N_2M_2-M_1N_1M_2+M_2N_2M_1-M_2N_1M_1
    \\
    &=&M_1N_2M_2\pi_2-M_1N_1\pi_1^0M_2+M_2N_2\pi_2^0M_1-M_2N_1M_1\pi_1
    \\
    &=&M_1\pi_2-\pi_1^0M_2+\pi_2^0M_1-M_2\pi_1
    \\
    &=&2(M_1-M_2)+(M_1+M_2)(\pi_2-\pi_1)+(\pi_2^0-\pi_1^0)(M_1+M_2)\,.
  \end{eqnarray*}
  By \eqref{vai ganzo1} and \eqref{vai ganzo2}, and since $\|M_i\|\le C$ by \eqref{chefreddo0}, we thus find
  \begin{eqnarray*}
    2\|M_2-M_1\|&\le&2\,\|M_1\|\|M_2\|\|N_2-N_1\|+\|M_1+M_2\|\Big(\|\pi_2-\pi_1\|+\|\pi_2^0-\pi_1^0\|\Big)
    \\
    &\le& C\,\Big(\|N_2-N_1\|+|p_2-p_1|^\a+|x_2-x_1|^\a\Big)
    \\
    &\le& C\,\Big((1+L)\,|x_2-x_1|^\a+|p_2-p_1|^\a\Big)\le C\,|p_2-p_1|^\a\,,
  \end{eqnarray*}
  where in the last line we have first used $[\nabla^{S_0}f]_{C^{0,\a}(S_0)}\le L$ and then \eqref{chefreddo0}. This completes the proof of \eqref{chefreddo1}, thus of the theorem.
\end{proof}

\section{Volume-fixing variations}\label{section volume fixing} Comparison sets used in variational arguments usually arise as compactly supported perturbations of the considered minimizer. In order to use these constructions in volume constrained variational problems, one needs to restore changes in volume due to such local variations. In the study of minimizing clusters, this kind of tool is provided in \cite[Proposition VI.12]{Almgren76}; see also \cite[Section 29.6]{maggiBOOK}. The following theorem is a version of Almgren's result which is suitably adapted to the problems considered in here. In particular, it adds to \cite[Corollary 29.17]{maggiBOOK} the conclusions \eqref{volumefix thesis 4} and \eqref{volumefix thesis V}.

\begin{theorem}[Volume-fixing variations]\label{thm volume fixing}
  If $\E_0$ is a $N$-cluster in $\R^n$, then there exist positive constants $r_0$, $\e_0$, $R_0$ and $C_0$ (depending on $\E_0$) with the following property: if $\E$ and $\F$ are $N$-clusters in $\R^n$ with
  \begin{eqnarray}
    \label{volumefix hp 2}
    \d(\E,\E_0)&\le& \e_0\,,
    \\
    \label{volumefix hp 1}
    \F(h)\Delta\E(h)&\cc& B_{x,r_0}\,,\qquad\forall h=1,...,N\,,
  \end{eqnarray}
  for some $x\in\R^n$, then there exists a $N$-cluster $\F'$ such that
  \begin{eqnarray}\label{volumefix thesis 1}
    \F'(h)\Delta\F(h)&\cc& B_{R_0}\setminus\ov{B_{x,r_0}}\,,\qquad\forall h=1,...,N\,,
    \\\label{volumefix thesis 2}
    \vol(\F')&=&\vol(\E)\,,
    \\\label{volumefix thesis 3}
    |P(\F')-P(\F)|&\le& C_0\,P(\E)\,|\vol(\F)-\vol(\E)|\,,
    \\\label{volumefix thesis 4}
    |\d(\F',\E)- \d(\F,\E)|&\le&C_0\,P(\E)\,|\vol(\F)-\vol(\E)|\,.
  \end{eqnarray}
  Moreover, if $g:\R^n\to[0,\infty)$ is locally bounded, then
  \begin{eqnarray}
    \label{volumefix thesis V}
    \sum_{h=0}^N\int_{\F'(h)\Delta\F(h)}g&\le&C_0\,\|g\|_{L^\infty(B_R)}\,P(\E)\,\,|\vol(\F)-\vol(\E)|\,.
  \end{eqnarray}
\end{theorem}

We shall need the following slight refinement of \cite[Lemma 17.9]{maggiBOOK}.

\begin{lemma}
  \label{lemma distanza L1} If $g:\R^n\to[0,\infty)$ is locally bounded, $E$ is a set of locally finite perimeter in an open set $A$ and $T\in C^1_c(A;\R^n)$, then for every $\eta>0$ there exist $K\subset A$ compact and $\e>0$ (depending on $T$) such that if $\{f_t\}_{|t|<\e}$ is a flow with initial velocity $T$, then
  \begin{equation}
    \label{siamo precisi}
    \int_{f_t(E)\Delta E}g\,\le (1+\eta)\,\|T\|_{C^0(\R^n)}\,\|g\|_{L^\infty(K)}\,P(E;K)\,|t|\,,\qquad\forall |t|<\e\,.
  \end{equation}
\end{lemma}

\begin{proof}
  Since $(d(f_t)^{-1}/dt)|_{t=0}=-T$, if we set $\Phi_{s,t}(x)=s\,x+(1-s)\,(f_t)^{-1}(x)$ for $x\in\R^n$ and $s\in(0,1)$, then  for every $\eta>0$ there exists $\e>0$ such that $\{\Phi_{s,t}\}_{|t|<\e}$ is a family of diffeomorphism on $\R^n$ with
  \[
  \inf_{x\in\R^n}\,J\Phi_{s,t}(x)\ge1-\eta\,,\qquad\|\Id-(f_t)^{-1}\|_{C^0(\R^n)}\le (1+\eta)\,|t|\,\|T\|_{C^0(\R^n)}\,,\qquad\forall |t|<\e\,.
  \]
  Let $K\subset A$ compact be such that $\{f_t\ne\Id\}\subset K$ for every $|t|<\e$. By Fubini's theorem and by the area formula, if $u\in C^1(\R^n)$, then
  \begin{eqnarray*}
    \int_{\R^n}g\,|u-u((f_t)^{-1})|&\le&  (1+\eta)\,|t|\,\|T\|_{C^0(\R^n)}\int_{K}\,g(x)\,dx\int_0^1\,|\nabla u(\Phi_{s,t}(x))|\,ds
    \\
    &=& (1+\eta)\,|t|\,\|T\|_{C^0(\R^n)}\|g\|_{L^\infty(K)}\,\int_0^1\,ds\int_{K}\,\frac{|\nabla u(y)|}{J\Phi_{s,t}(\Phi_{s,t}^{-1}(y))}\,dy
    \\
    &\le& \frac{1+\eta}{1-\eta}\,|t|\,\|T\|_{C^0(\R^n)}\|g\|_{L^\infty(K)}\,\int_{K}|\nabla u|\,.
  \end{eqnarray*}
  By \cite[Theorem 13.8]{maggiBOOK} there exists $\{u_h\}_{h\in\N}\subset C^1(\R^n)$ such that $u_h\to 1_E$ a.e. on $A$ and $\limsup_{h\to\infty}\int_{K}|\nabla u_h|\le P(E;\ov{K})$. Since $|u_h-u_h((f_t)^{-1})|\to 1_{E\Delta f_t(E)}$ a.e. on $A$, we conclude the proof by Fatou's lemma.
\end{proof}

\begin{proof}
  [Proof of Theorem \ref{thm volume fixing}] One repeats the proof of \cite[Corollary 29.17]{maggiBOOK}, exploiting Lemma \ref{lemma distanza L1}  in place of \cite[Lemma 17.9]{maggiBOOK} in order to obtain \eqref{volumefix thesis 4} and \eqref{volumefix thesis V}. We thus omit the details.
\end{proof}

\bibliography{references}
\bibliographystyle{is-alpha}
\end{document}

%% file: brutti.pstex_t
\begin{picture}(0,0)%
\includegraphics{brutti.eps}%
\end{picture}%
\setlength{\unitlength}{3947sp}%
\begingroup\makeatletter\ifx\SetFigFont\undefined%
\gdef\SetFigFont#1#2#3#4#5{%
  \reset@font\fontsize{#1}{#2pt}%
  \fontfamily{#3}\fontseries{#4}\fontshape{#5}%
  \selectfont}%
\fi\endgroup%
\begin{picture}(3130,1014)(781,-1186)
\end{picture}%